\documentclass[11pt, reqno]{amsart}
\usepackage{amsmath, amsthm, amscd, amsfonts, amssymb, graphicx, color, mathtools}
\usepackage{mathrsfs}
\usepackage[bookmarksnumbered, colorlinks, plainpages]{hyperref}
\usepackage{enumerate}
\usepackage{setspace}
\usepackage{multicol}
\usepackage[margin=1in]{geometry}
\usepackage{comment}
\usepackage{dsfont}
\usepackage{booktabs}
\usepackage{caption}
\usepackage{subcaption}
\usepackage{pgfplots}
\pgfplotsset{compat=1.18}

\usepackage[normalem]{ulem}
\newcommand\tsout{\bgroup\markoverwith{\textcolor{red}{\rule[0.5ex]{2pt}{1.4pt}}}\ULon}
\newcommand{\stkout}[1]{\ifmmode\text{\tsout{\ensuremath{#1}}}\else\tsout{#1}\fi}
\allowdisplaybreaks

\theoremstyle{definition}
\newtheorem{theorem}{Theorem}[section]
\newtheorem{lemma}[theorem]{Lemma}
\newtheorem{proposition}[theorem]{Proposition}

\newtheorem{definition}[theorem]{Definition}

\newtheorem{remark}[theorem]{Remark}
\numberwithin{equation}{section}

\newcommand{\dif}{\mathrm{d}}
\newcommand{\bff}{\boldsymbol}
\newcommand{\bb}{\mathbb}

\newcommand{\dt}{\mathrm{d}t}

\newcommand{\ds}{\mathrm{d}s}

\newcommand{\dW}{\mathrm{d}W}

\newcommand{\R}{\mathcal{R}}
\newcommand{\red}[1]{{\color{red}{#1}}}

\newcommand{\blue}[1]{{\color{blue}{#1}}}
\newcommand{\norm}[2]{\left\|{#1}\right\|_{#2}}
\newcommand{\inpro}[2]{\left\langle#1,#2\right\rangle}

\newcommand{\abs}[1]{\left|{#1}\right|}
\newcommand{\one}{\mathds{1}}

\def\be{\begin{equation}\label}
\def\ee{\end{equation}}
\def\bd{\begin{definition}\label}
\def\ed{\end{definition}}
\def\bt{\begin{theorem}\label}
\def\et{\end{theorem}}
\def\bl{\begin{lemma}\label}
\def\el{\end{lemma}}
\def\br{\begin{remark}\label}
\def\er{\end{remark}}
\def\bal{\[\begin{aligned}}
\def\eal{\end{aligned}\]}

\begin{document}
\setcounter{page}{1}

\title[A mixed FEM for a class of fourth-order SPDEs with multiplicative noise]{A mixed finite element method for a class of fourth-order stochastic evolution equations with multiplicative noise}

\author[Beniamin Goldys]{Beniamin Goldys}
\address{School of Mathematics and Statistics, The University of Sydney, Sydney 2006, Australia}
\email{\textcolor[rgb]{0.00,0.00,0.84}{beniamin.goldys@sydney.edu.au}}

\author[Agus L. Soenjaya]{Agus L. Soenjaya}
\address{School of Mathematics and Statistics, The University of New South Wales, Sydney 2052, Australia}
\email{\textcolor[rgb]{0.00,0.00,0.84}{a.soenjaya@unsw.edu.au}}

\author[Thanh Tran]{Thanh Tran}
\address{School of Mathematics and Statistics, The University of New South Wales, Sydney 2052, Australia}
\email{\textcolor[rgb]{0.00,0.00,0.84}{thanh.tran@unsw.edu.au}}

\date{\today}

\begin{abstract}
We develop a fully discrete, semi-implicit mixed finite element method for approximating solutions to a class of fourth-order stochastic partial differential equations (SPDEs) with non-globally Lipschitz and non-monotone nonlinearities, perturbed by spatially smooth multiplicative Gaussian noise. The proposed scheme is applicable to a range of physically relevant nonlinear models, including the stochastic Landau--Lifshitz--Baryakhtar (sLLBar) equation, the stochastic convective Cahn--Hilliard equation with mass source, and the stochastic regularised Landau--Lifshitz--Bloch (sLLB) equation, among others. To overcome the difficulties posed by the interplay between the nonlinearities and the stochastic forcing, we adopt a `truncate-then-discretise' strategy: the nonlinear term is first truncated before discretising the resulting modified problem. We show that the strong solution to the truncated system converges in probability to that of the original problem. A fully discrete numerical scheme is then proposed for the truncated problem. Assuming initial data in $\mathbb{H}^2$, we utilise parabolic smoothing estimates and the temporal H\"older continuity of the solution to establish both convergence in probability and strong convergence (with quantitative rates) for the two fields used in the mixed formulation. Numerical simulations are provided to support the theoretical results.
\end{abstract}
\maketitle

\tableofcontents

\section{Introduction}
Motivated by physical applications, we consider the following fourth-order system of nonlinear SPDEs with non-monotone nonlinearities, perturbed by a spatially smooth multiplicative Gaussian noise:
\begin{subequations}\label{equ:sllbar}
	\begin{alignat}{2}
		&\dif \bff{u}
		= 
		\big( \lambda_1 \bff{H}
		- \lambda_2 \Delta \bff{H} 
		- \gamma \bff{u} \times \bff{H} 
		+ \mathcal{S}(\bff{u}) \big)\, \dt
		\; &&
		\nonumber\\
		&\qquad \quad
		+
		G(\bff{u}) \,\dif W(t) 
		\; && \quad\text{for $(t,\bff{x})\in(0,T)\times \mathscr{D}$,}
		\label{equ:sllbar a}
		\\[1ex]
		&\bff{H}
		= 
		\Delta \bff{u}
		+
		f(\bff{u})
		\; && \quad\text{for $(t,\bff{x})\in(0,T)\times\mathscr{D}$,}
		\label{equ:sllbar b}
		\\[1ex]
		&\bff{u}(0,\bff{x})= \bff{u}_0(\bff{x}) 
		\; && \quad\text{for } \bff{x}\in \mathscr{D},
		\label{equ:sllbar c}
		\\[1ex]
		&\frac{\partial \bff{u}}{\partial \bff{n}}= \bff{0}, 
		\;\displaystyle{\frac{\partial \bff{H}}{\partial \bff{n}}= \bff{0}} 
		\; && \quad\text{for } (t,\bff{x})\in (0,T) \times \partial \mathscr{D},
		\label{equ:sllbar d}
	\end{alignat}
\end{subequations}
where $\mathscr{D}\subset\bb{R}^d$, $d\leq 3$, is a bounded regular domain, and $\bff{u}:\Omega\times [0,T]\times \mathscr{D}\to \bb{R}^3$ is a vector-valued random variable. Here, $W$ is a real-valued Wiener process on a filtered probability space $(\Omega, \mathcal{F}, \bb{P}, \{\mathcal{F}_t\}_{t\geq 0})$ with respect to the usual filtration, and $G(\bff{u})$ is a Lipschitz function of $\bff{u}$ satisfying certain assumptions (details are elaborated in Section~\ref{subsec:assump}). The forcing term $f(\bff{u}):= \kappa \mu \bff{u}- \kappa |\bff{u}|^2 \bff{u}$ arises from the Ginzburg--Landau theory,
which is the negative variational derivative of $V(\bff{u}):=\kappa (\abs{\bff{u}}^2-\mu)^2 /4$, a double-well potential function. Define 
\[
\mathcal{S}(\bff{u}):= \mathcal{M}(\bff{u})+ \mathcal{C}(\bff{u}),
\]
where $\mathcal{M}(\bff{u})$ is a mass source term with at most quadratic growth and $\mathcal{C}(\bff{u})$ is a convective term given by
\begin{equation}\label{equ:torque}
	\mathcal{C}(\bff{u})
	:= 
	\beta_1 (\bff{\nu} \cdot \nabla) \bff{u} 
	+ 
	\beta_2 \bff{u} \times (\bff{\nu} \cdot \nabla) \bff{u},
\end{equation}
where $\bff{\nu}$ is specified in~\eqref{equ:nu}.
All numerical coefficients are non-negative.

The problem~\eqref{equ:sllbar} describes various problems in physics. When $\lambda_1,\lambda_2$, and $\gamma$ are positive, problem~\eqref{equ:sllbar} is the \emph{stochastic Landau--Lifshitz--Baryakhtar} (sLLBar) system with spin current~\cite{GolSoeTra24b, SoeTra23, XuLiuZha24}, which can be seen as a Cahn--Hilliard-type regularisation of the \emph{stochastic Landau--Lifshitz--Bloch} (sLLB) equation in micromagnetics~\cite{BrzGolLe20, Eva_etal12, LeSoeTra24, Soe24}. When $\gamma=0$, \eqref{equ:sllbar} is the stochastic bi-flux reaction-diffusion system~\cite{BevGalSimRio13} if $\beta_2=0$, a stochastic population growth/dispersal model with long-range effects~\cite{CohMur81} if $\beta_1=\beta_2=0$, the \emph{Cahn--Hilliard-Cook} (CHC) equation~\cite{KovLarMes11} if $\lambda_1=\beta_1=\mathcal{M}(\bff{u})=0$ (and the noise is additive), the \emph{stochastic convective Cahn--Hilliard equation with mass source} (sCHm)~\cite{Luc21, LeeKimYoo21} if $\lambda_1=\beta_2=0$, and the \emph{stochastic convective Allen--Cahn/Cahn--Hilliard} (sAC/CH) equation~\cite{AntKarMil16} if $\beta_2=0$.

The development of numerical methods for physically relevant SPDEs with non-globally Lipschitz and non-monotone nonlinearities perturbed by multiplicative noise is an active area of research (see e.g. \cite{BreDod21, CuiHon19, HuaShe23, HutJen15} and many others). As~\eqref{equ:sllbar} is a fourth-order equation, a conforming finite element method to solve the equation directly would require $C^1$-elements, which can be computationally costly. Numerically treating the problem in mixed form allows us to work with $C^0$-conforming finite elements and use the mixed finite element method (see~\eqref{equ:euler}), at the expense of introducing an auxiliary unknown and performing a more delicate analysis. To the best of our knowledge, no numerical scheme has been proposed for the problem~\eqref{equ:sllbar} in its generality, not even for the sLLBar equation (with or without spin current), the sCHm equation, or the sAC/CH equation.

On a related note, several numerical schemes have been proposed in the literature for the CHC equation with \emph{additive} noise, including a $C^1$-conforming semi-discrete scheme~\cite{ChaCaoZou18}, a fully implicit scheme~\cite{FurKovLarLin18}, and a fully explicit scheme combined with spectral Galerkin method~\cite{CaiQiWan23} (see also~\cite{FenLiZha20} for gradient-type noise, where strong convergence of an implicit scheme in $H^{-1}$ is shown). Note that even in this setting ($\gamma=\beta_2=0$), numerical schemes of implicit/explicit-type proposed here have not been analysed before. Furthermore, adding a mass source and a convective/precession term in~\eqref{equ:sllbar} causes a nontrivial difficulty in the analysis due to the non-conservative mass and the loss of gradient flow structure, which is already encountered even in the deterministic case~\cite{Lam22, ShiYanCuiMir24}.
On the other hand, setting $\lambda_2=0$ in~\eqref{equ:sllbar} gives the sLLB equation. A $C^1$-conforming finite element scheme for a regularised version of the sLLB equation (which is simpler than our problem \eqref{equ:sllbar}) is proposed in~\cite{GolJiaLe24}. As we can consider \eqref{equ:sllbar} to be a physically relevant regularisation of sLLB, our scheme provides a more practical method to approximate the solution to sLLB by taking $\lambda_2$ sufficiently small (in light of the convergence result for sLLBar to sLLB in~\cite[Section~8]{GolSoeTra24b}).

Regarding the error analysis, we describe here some difficulties at the discrete level that need to be overcome. To this end, let $A$ denote the Neumann Laplacian, $A_h$ the discrete Laplacian, and $\Pi_h$ the orthogonal projection onto some finite element space $\bb{V}_h$. Firstly, loosely speaking, the mixed finite element method aims to approximate $\bff{u}$ and $\bff{H}$ simultaneously, using finite element functions which belong to $\bb{H}^1$ (but not $\bb{H}^2$). This already makes the analysis of the mixed finite element scheme more challenging than its $C^1$-conforming counterpart, even in the deterministic case. 
Furthermore, the presence of the finite element projection $\Pi_h$ in front of the nonlinearities present in~\eqref{equ:sllbar} destroys some dissipativity properties of the continuous problem. For instance, while there exists a $C>0$ such that for any sufficiently regular function $\bff{v}$, 
\begin{align*}
	-\inpro{f(\bff{v})}{\bff{v}} \geq -C,
	\text{ and }
	-\inpro{\nabla f(\bff{v})}{\nabla \bff{v}} \geq -C\norm{\nabla\bff{v}}{\bb{L}^2}^2,
\end{align*}
we notice that for $\bff{v}_h\in \bb{V}_h$,
\begin{align*}
	-\inpro{\nabla \Pi_h f(\bff{v}_h)}{\nabla \bff{v}_h} \not\geq -C\norm{\nabla \bff{v}_h}{\bb{L}^2}^2.
\end{align*}
As such, moment bounds for $\norm{\bff{u}_h}{\bb{L}^2}$ and $\norm{\nabla \bff{u}_h}{\bb{L}^2}$, where $\bff{u}_h$ is the finite element approximation of $\bff{u}$, are difficult to attain. 

Similar difficulties are encountered in~\cite{FurKovLarLin18} and~\cite{FenLiZha20} for a simpler model. In their case, however, these issues could be overcome by exploiting the fact that the (scalar-valued) Cahn--Hilliard--Cook equation is the $H^{-1}$-gradient flow of a Lyapunov functional together with the mass conservation property to derive a moment bound for the $H^{-1}$ norm of the finite element solution. This bound could then be bootstrapped and used to derive moment bounds in stronger norms and strong convergence of the scheme in the $H^{-1}$ or $L^2$ norms. However, it is not clear how to adapt such arguments to our case since~\eqref{equ:sllbar} is \emph{not} a gradient flow in the presence of the cross product term, the forcing term $f(\bff{u})$, and the convective term $\mathcal{C}(\bff{u})$. Furthermore, these nonlinearities are non-monotone, thus the general results from~\cite{GyoMil09, LiuQia21} do not apply and the analysis needs to proceed differently here.
We also remark that the nonlinear term $\Delta f(\bff{u})$ is absent in the regularised sLLB model considered in~\cite{GolJiaLe24}. In addition, since $C^1$-elements are employed there, the complications described in this and the preceding paragraph do not arise in that work.

We outline the approach taken in this paper as follows. To overcome the difficulties mentioned before, while still employing $C^0$-conforming elements, we adopt the idea from~\cite{SheYan10} in the deterministic case by first truncating the potential function so as to have at most quadratic growth at infinity. Such pointwise truncation is both physically reasonable and a common practice~\cite{CafMul95, ConMelSul11, DorPie24, KesNocSch04, WanYu18}. In doing so, the forcing function $f$ is approximated by a globally Lipschitz $C^2$-smooth function $f_R$. We also truncate the mass source~\cite{LeeKimYoo21}. We show that the strong solution to the problem with truncated potential converges in probability to that of~\eqref{equ:sllbar} as the truncation parameter $R\uparrow\infty$. Note that even after these modifications, the problem~\eqref{equ:sllbar} is still a system of SPDEs with non-monotone and non-globally Lipschitz nonlinearities due to the cross product term $\bff{u}\times \bff{H}$ and the non-variational term $\mathcal{C}(\bff{u})$, thus the analysis is not straightforward.  

A mixed finite element method is proposed for the truncated problem. In the absence of cross-product terms (corresponding to the sCHm or sAC/CH equations), the resulting fully discrete scheme is of IMEX type. The error analysis is carried out within the variational framework for SPDEs and assumes initial data in $\bb{H}^2$, which guarantees the existence and uniqueness of analytically strong pathwise solutions~\cite{GolSoeTra24b}. No additional or unrealistic regularity assumptions are imposed. We establish stability estimates of the scheme in strong norms, including bounds on higher moments, and exploit parabolic smoothing properties together with the temporal H\"older continuity of the solution. These ingredients are combined to derive error bounds localised on events of large probability, leading to convergence in probability and strong convergence with quantitative rates for each field in the mixed formulation, following the approach of~\cite{BesMil19, CarPro12}. The main results are stated in Theorem~\ref{the:uhn un 1 L2}, Theorem~\ref{the:uhn prob}, and Theorem~\ref{the:uhn strong}.
Numerical experiments to corroborate the theoretical results are described in Section~\ref{sec:num exp}.

\section{Preliminaries}

\subsection{Notations}
We begin by defining some notations used in this paper. Let $\mathscr{D}$ be a convex Lipschitz domain or a domain with $C^2$-smooth boundary. The function space $\bb{L}^p := \bb{L}^p(\mathscr{D}; \bb{R}^3)$ denotes the usual space of $p$-th integrable functions defined on~$\mathscr{D}$ and taking values in $\bb{R}^3$, and $\bb{W}^{k,p} := \bb{W}^{k,p}(\mathscr{D}; \bb{R}^3)$ denotes the usual Sobolev space of 
functions on $\mathscr{D} \subset \bb{R}^d$ taking values in $\bb{R}^3$. We
write $\bb{H}^k := \bb{W}^{k,2}$.
The partial derivative
$\partial/\partial x_i$ will be written by $\partial_i$ for short. The partial derivative of~$f$ with respect to time $t$ will be denoted by $\partial_t$. The operator $\Delta$ denotes the Neumann Laplacian acting on $\bb{R}^3$-valued functions with domain
\[
	\text{D}(\Delta):= \left\{ \bff{v}\in \bb{H}^2 : \frac{\partial\bff{v}}{\partial \bff{n}}=0 \text{ on } \partial\mathscr{D} \right\}.
\]

If $X$ is a Banach space, $L^p(0,T; X)$ and $W^{k,p}(0,T;X)$ denote respectively the usual Lebesgue and Sobolev spaces of {strongly measurable} functions on $(0,T)$ taking values in $X$. The space $L^p(\Omega; X)$ denotes the space of {strongly measurable} $X$-valued random variable with finite $p$-th moment, where $(\Omega,\mathcal{F},\bb{P})$ is a probability space. For an interval $J$, the space $\mathcal{C}^0(J;X)$ and $\mathcal{C}^\alpha(J;X)$ denote respectively the space of continuous functions and the space of $\alpha$-H\"older continuous functions on $J$ taking values in $X$. For brevity, we will denote the spaces $L^p(0,T;X)$, $\mathcal{C}^0([0,T];X)$, and $\mathcal{C}^\alpha([0,T];X)$ by $L^p_T(X)$, $\mathcal{C}^0_T(X)$, and $\mathcal{C}^\alpha_T(X)$, respectively.

Throughout, we denote the scalar product in a Hilbert space $H$ by $\langle \cdot, \cdot\rangle_H$ and its corresponding norm by $\|\cdot\|_H$. The expectation of a random variable $Y$ will be denoted by $\bb{E}[Y]$. We do not distinguish between the scalar product of $\bb{L}^2$ vector-valued functions taking values in $\bb{R}^3$ and the scalar product of $\bb{L}^2$ matrix-valued functions taking values in $\bb{R}^{3\times 3}$, and denote them by $\langle\cdot,\cdot\rangle$.

In various estimates, the constant $C$ denotes a
generic constant which takes different values at different occurrences. If
the dependence of $C$ on some variables, e.g.~$R$ and $T$, is highlighted, we will write $C_{R,T}$.

\subsection{Assumptions}\label{subsec:assump}
Let $\varphi_R \in C_c^2(\bb{R}^3)$ be a $C^2$-smooth bump function whose support lies in the closed ball $B_{2R}(\bff{0})$, such that
\begin{equation}\label{equ:bump}
	\varphi_R(\bff{x})=
	\begin{cases}
		1, & \text{if } |\bff{x}| \leq R\\
		0, & \text{if } |\bff{x}| \geq 2R.
	\end{cases}
\end{equation}
The functions $f_R:\bb{R}^3\to \bb{R}^3$ and $\mathcal{M}_R:\bb{R}^3\to\bb{R}^3$ are defined by
	\begin{align}
    \label{equ:f R}
		f_R(\bff{u}):= \varphi_R(\bff{u}) f(\bff{u}),
        \\
    \label{equ:M R}
        \mathcal{M}_R(\bff{u}):= \varphi_R(\bff{u}) \mathcal{M}(\bff{u}),
	\end{align}
For equation~\eqref{equ:sllbar}, we set $\lambda_1=\lambda_2=\kappa=\mu=1$ for simplicity. We further assume the following:
\begin{enumerate}
	\item The map $G:\bb{R}^3\to \bb{R}^3$ is Lipschitz continuous with Lipschitz constant $C_G$. Moreover, there exists a constant $C>0$ such that
    \begin{align*}
        \norm{\nabla G(\bff{v})-\nabla G(\bff{w})}{\bb{L}^2}
        &\leq
        C \norm{\nabla\bff{v}-\nabla \bff{w}}{\bb{L}^2}, \quad\forall \bff{v},\bff{w}\in \bb{H}^1,
        \\
        \norm{G(\bff{v})}{\bb{H}^2}
        &\leq
        C\big(1+\norm{\bff{v}}{\bb{H}^2}\big),
        \quad\forall \bff{v}\in \bb{H}^2.
    \end{align*}
	\item The spin current vector field $\bff{\nu}\in L^\infty\big(\bb{R}^+;\bb{L}^\infty(\mathscr{D};\bb{R}^d)\big)$ is given. For simplicity, set
    \begin{equation}\label{equ:nu}
    \norm{\bff{\nu}}{L^\infty(\bb{R}^+;\bb{L}^\infty(\mathscr{D};\bb{R}^d))}=1.
    \end{equation}
\end{enumerate}
We remark that our results are also valid more generally for noise of the form $G(\bff{u}) \,\mathrm{d}\bff{W}$, where $\bff{W}$ is a $\mathrm{D}(\Delta)$-valued $Q$-Wiener process of the form $\bff{W}(t)= \sum_{k=1}^\infty \sqrt{q_k} \bff{e}_k W_k(t)$, where $\{W_k\}_{k\in \bb{N}}$ is a sequence of independent real-valued Brownian motions, $\{\bff{e}_k\}_{k\in \bb{N}}$ is an orthonormal basis of $\mathrm{D}(\Delta)$ such that $Q\bff{e}_k= q_k \bff{e}_k$, and $\mathrm{Tr}(Q):=\sum_{k=1}^\infty q_k<\infty$. In this case, we assume $G:\bb{H}^s\to \mathcal{L}(\bb{H}^s)$ is Lipschitz continuous for $s=1$ and is of linear growth for $s=2$. The assumptions here cover the case where $G(\bff{u})$ is the identity operator (additive noise) or the noise term given in~\cite{GolSoeTra24b} for the sLLBar equation. We consider just a single real-valued Brownian motion in this paper for simplicity of presentation.

\subsection{Existence, uniqueness, and regularity of solution}

The existence, uniqueness, and regularity of the (probabilistically and analytically) strong solution to the problem~\eqref{equ:sllbar} are studied in~\cite{GolSoeTra24b}. We summarise the relevant results here. First, define a self-adjoint operator
\begin{equation*}
		A:= \Delta^2 -\Delta, \quad \text{with }
		\mathrm{D}(A)= \mathrm{D}(\Delta^2).
\end{equation*}
The following theorem is essentially shown in~\cite[Theorem~2.2, Theorem~2.3, and Lemma~4.10]{GolSoeTra24b}.

\begin{theorem}
Let $\bff{u}_0\in \mathrm{D}(A^{\frac12})$ and $T>0$ be given. There exists a unique pathwise solution $\bff{u}$ of the problem~\eqref{equ:sllbar} with the following regularity: for any $\beta\in [0,\frac12]$, $\alpha\in [0,\frac12 -\beta]$, and $p\in[1,\infty)$,
\[
    \bff{u} \in L^p\big(\Omega; \mathcal{C}^\alpha([0,T]; \mathrm{D}(A^\beta))\big) \cap L^p\big(\Omega; L^2(0,T;\mathrm{D}(A))\big)\red{.}
\]
Moreover, the solution enjoys the smoothing property on $(0,T]$: for any $\beta\in [\frac12, 1)$, $\alpha\in (0,1-\beta)$, and $p\in[1,\infty)$,
\[
    \bff{u} \in L^p\big(\Omega; \mathcal{C}^\alpha((0,T]; \mathrm{D}(A^\beta))\big).
\]
\end{theorem}

We now consider problem~\eqref{equ:sllbar} with truncated nonlinearities, with all numerical coefficients set to 1. More precisely, for each $R>0$, the pair $(\bff{u}_R, \bff{H}_R)$ satisfies
\begin{subequations}\label{equ:trunc sllbar}
	\begin{alignat}{2}
		&\dif \bff{u}_R
		= 
		\big( \bff{H}_R
		- \Delta \bff{H}_R 
		- \bff{u}_R \times \bff{H}_R 
		+ \widetilde{\mathcal{S}}(\bff{u}_R) \big)\, \dt
		+
		G(\bff{u}_R) \,\dif W(t)
		\; && \quad\text{for $(t,\bff{x})\in(0,T)\times \mathscr{D}$,}
  \label{equ:trunc sllbar a}
		\\[1ex]
		&\bff{H}_R
		= 
		\Delta \bff{u}_R
		+
		f_R(\bff{u}_R)
		\; && \quad\text{for $(t,\bff{x})\in(0,T)\times\mathscr{D}$,}
  \label{equ:trunc sllbar b}
		\\[1ex]
		&\bff{u}_R(0,\bff{x})= \bff{u}_{0}(\bff{x}) 
		\; && \quad\text{for } \bff{x}\in \mathscr{D},
  \label{equ:trunc sllbar c}
		\\[1ex]
			&\frac{\partial \bff{u}_R}{\partial \bff{n}}= \bff{0}, 
		\;\displaystyle{\frac{\partial \bff{H}_R}{\partial \bff{n}}= \bff{0}} 
		\; && \quad\text{for } (t,\bff{x})\in (0,T) \times \partial \mathscr{D},
  \label{equ:trunc sllbar d}
	\end{alignat}
\end{subequations}
i.e. the problem~\eqref{equ:sllbar} with $f(\bff{u})$ replaced by $f_R(\bff{u}_R)$, and
\begin{equation}\label{equ:P u}
\widetilde{\mathcal{S}}(\bff{u}_R):= \mathcal{C}(\bff{u}_R)+ \mathcal{M}_R(\bff{u}_R).
\end{equation}

A variational formulation for the problem~\eqref{equ:trunc sllbar} can be written as follows: For every $t\in [0,T]$ and $\bb{P}$-a.s., $(\bff{u}_R,\bff{H}_R)$ solves
\begin{equation}\label{equ:weakform}
	\begin{alignedat}{1}
    \inpro{\bff{u}_R(t)}{\bff{\chi}} 
    &= 
    \inpro{\bff{u}_0}{\bff{\chi}}
    +
    \int_0^t \inpro{\bff{H}_R(s)}{\bff{\chi}} \ds
    +
    \int_0^t \inpro{\nabla \bff{H}_R(s)}{\nabla \bff{\chi}} \ds
    \\
    &\quad
    -
    \int_0^t \inpro{\bff{u}_R(s)\times \bff{H}(s)}{\bff{\chi}} \ds
    +
    \int_0^t \inpro{\widetilde{\mathcal{S}}(\bff{u}_R(s))}{\bff{\chi}} \ds 
        \\
    &\quad
    +
    \int_0^t \inpro{G(\bff{u}_R(s))}{\bff{\chi}} \mathrm{d}W(s),
    \\
    \inpro{\bff{H}_R(t)}{\bff{\phi}}
    &=
    -\inpro{\nabla \bff{u}_R(t)}{\nabla \bff{\phi}} 
    +
    \inpro{f_R(\bff{u}_R(t))}{\bff{\phi}},
    \end{alignedat}
\end{equation}
for all $\bff{\chi}, \bff{\phi}\in \bb{H}^1$. This variational formulation will be used for the analysis of the numerical method proposed in Section~\ref{sec:mixed fem}.
By similar argument as in \cite{GolSoeTra24b}, we have the following result.

\begin{proposition}\label{pro:Holder u}
Let $\bff{u}_0\in \mathrm{D}(A^{\frac12})$ and $T>0$ be given. For each $R>0$, there exists a unique pathwise solution $\bff{u}_R$ of the problem~\eqref{equ:trunc sllbar} with the following regularity: for any $\beta\in [0,\frac12]$, $\alpha\in [0,\frac12 -\beta]$, and $p\in[1,\infty)$,
\[
    \bff{u}_R \in L^p\big(\Omega; \mathcal{C}^\alpha([0,T]; \mathrm{D}(A^\beta))\big) \cap L^p\big(\Omega; L^2(0,T;\mathrm{D}(A))\big).
\]
Moreover, the solution enjoys the smoothing property on $(0,T]$: for any $\beta\in [\frac12, 1)$, $\alpha\in (0,1-\beta)$, and $p\in [1,\infty)$,
\[
    \bff{u}_R \in L^p\big(\Omega; \mathcal{C}^\alpha((0,T]; \mathrm{D}(A^\beta))\big).
\]
\end{proposition}

\begin{proposition}
Let $\bff{u}_R$ and $\bff{u}$ be the solution to the problems \eqref{equ:sllbar} and \eqref{equ:trunc sllbar}, respectively. Then $\bff{u}_R(t)\to \bff{u}(t)$ a.s. on $[0,T]$ as $R\uparrow\infty$. Furthermore, for any $\varepsilon>0$ and $p\geq 1$,
\[
    \bb{P}\left[\sup_{t\in [0,T]} \norm{\bff{u}_R(t)-\bff{u}(t)}{\bb{H}^2} > \varepsilon\right] \leq C_p R^{-p},
\] 
where $C_p$ is a constant depending on $T, p$, and $\mathscr{D}$.
\end{proposition}

\begin{proof}
For each $R>0$, let
\begin{align*}
    \tau_R &:= \inf \left\{t\geq 0: \norm{\bff{u}_R(t)}{\bb{H}^2} \geq R \right\} \wedge T,
    \\
    \sigma_R &:= \inf \left\{t\geq 0: \norm{\bff{u}(t)}{\bb{H}^2} \geq R \right\} \wedge T.
\end{align*}
Then $\bff{u}_R(t)=\bff{u}(t)$ a.s. for all $t\leq \tau_R$, and $\tau_R\uparrow T$ as $R\uparrow \infty$, which implies $\bff{u}_R(t)\to \bff{u}(t)$ a.s. on $[0,T]$ as $R\uparrow\infty$. Furthermore, for any $\varepsilon>0$,
\begin{align*}
    \bb{P}\left[\sup_{t\in [0,T]} \norm{\bff{u}_R(t)-\bff{u}(t)}{\bb{H}^2} > \varepsilon\right] 
    &\leq 
    \bb{P}\left[\sup_{t\in [0,\tau_R]} \norm{\bff{u}_R(t)-\bff{u}(t)}{\bb{H}^2} > \varepsilon\right] 
    +
    \bb{P}\left[\tau_R<T\right]
    \\
    &=
    \bb{P} \left[\left\{\tau_R<T\right\} \cap \left\{\sigma_R=T\right\}\right]
    +
    \bb{P} \left[\left\{\tau_R<T\right\} \cap \left\{\sigma_R< T\right\}\right]
    \\
    &\leq 
    \frac{1}{R^p} \bb{E} \left[\norm{\bff{u}}{L^\infty(\bb{H}^2)}^p \right],
\end{align*}
as required.
\end{proof}
As such, we can approximate the solution to the problem~\eqref{equ:sllbar} using~\eqref{equ:trunc sllbar} by taking $R$ sufficiently large. From this point onwards, we will focus on the numerical approximation of \eqref{equ:trunc sllbar} and write $\bff{u}$ in place of $\bff{u}_R$.

\subsection{Finite element approximation}

Let $\{\mathcal{T}_h\}_{h>0}$ be a family of quasi-uniform triangulations of $\mathscr{D}\subset \bb{R}^d$ with maximal mesh-size $h$, and let $\bb{V}_h\subset \bb{W}^{1,\infty}$ be the Lagrange finite element space 
\begin{equation*}
	\bb{V}_h := \{\bff{\phi}\in C(\overline{\mathscr{D}}; \bb{R}^3): \bff{\phi}|_K \in \bb{P}_1(K;\bb{R}^3), \; \forall K \in \mathcal{T}_h\},
\end{equation*}
where $\bb{P}_1(K; \bb{R}^3)$ denotes the space of linear polynomials on $K$ taking values in $\bb{R}^3$.
Let $T>0$ be fixed and $k$ be the time-step size. Furthermore, let $\bff{u}_h^n$ and $\bff{H}_h^n$, respectively, be the approximation in $\bb{V}_h$ of $\bff{u}_R(t)$ and $\bff{H}_R(t)$ at time $t=t_n:=nk$, where $n=0,1,2,\ldots, N$ and $N=\lfloor T/k \rfloor$. 

We begin by defining several operators which will be used in the analysis. Firstly, there exists an orthogonal projection operator $\Pi_h: \bb{L}^2 \to \bb{V}_h$ such that
\begin{align}\label{equ:orth proj}
	\inpro{\Pi_h \bff{v}-\bff{v}}{\bff{\chi}}=0,
	\quad
	\forall \bff{\chi}\in \bb{V}_h.
\end{align}
The operator $\Pi_h$ is stable~\cite{BanYse14, CroTho87, DouDupWah74} in $\bb{L}^p$ and $\bb{W}^{1,p}$ for $p\in [1,\infty)$: there exists a constant $C$ independent of $\bff{v}$ such that
\begin{align}
\label{equ:proj stab Lp}
    \norm{\Pi_h \bff{v}}{\bb{L}^p} &\leq C\norm{\bff{v}}{\bb{L}^p}, \quad \forall \bff{v}\in \bb{L}^p,
    \\
\label{equ:proj stab W1p}
    \norm{\nabla \Pi_h \bff{v}}{\bb{L}^p} &\leq C \norm{\nabla\bff{v}}{\bb{L}^p}, \quad \forall \bff{v}\in \bb{W}^{1,p}.
\end{align}
Moreover, it has the following approximation property:
\begin{align}\label{equ:proj approx}
	\norm{\bff{v}- \Pi_h\bff{v}}{\bb{L}^p}
	+
	h \norm{\nabla( \bff{v}-\Pi_h\bff{v})}{\bb{L}^p}
	\leq
	Ch^s \norm{\bff{v}}{\bb{W}^{s,p}}, \quad s\in \{1,2\}.
\end{align}
We mainly use~\eqref{equ:proj stab Lp}, \eqref{equ:proj stab W1p}, and~\eqref{equ:proj approx} for $p=2$.

Secondly, define the Ritz projection $\R_h:\bb{H}^1 \to \bb{V}_h$ by
\begin{align}\label{equ:Ritz}
    \inpro{\nabla \R_h \bff{v}-\nabla \bff{v}}{\nabla \bff{\chi}}=0, \quad \forall \bff{\chi}\in\bb{V}_h, 
    \quad \text{such that} \quad \inpro{\R_h \bff{v}-\bff{v}}{\bff{1}}=0.
\end{align}
The stability and approximation properties of the Ritz projection~\cite{LeyLi21, RanSco82} are assumed to hold. In particular, for $p\in (1,\infty)$,
\begin{align}
    \label{equ:Ritz approx}
    \norm{\bff{v}-\R_h \bff{v}}{\bb{L}^p}
    +
    h\norm{\nabla(\bff{v}-\R_h \bff{v})}{\bb{L}^p}
    &\leq
    Ch^s \norm{\bff{v}}{\bb{W}^{s,p}}, \quad s\in \{1,2\}.
\end{align}
Finally, the discrete Laplacian operator $\Delta_h: \bb{V}_h \to \bb{V}_h$ is defined by
\begin{align}\label{equ:disc laplacian}
	\inpro{\Delta_h \bff{v}_h}{\bff{\chi}}
	=
	- \inpro{\nabla \bff{v}_h}{\nabla \bff{\chi}},
	\quad 
	\forall \bff{v}_h, \bff{\chi} \in \bb{V}_h.
\end{align}
Consequently, for any $p,q \in [1,\infty]$ such that $1/p+1/q=1$, by H\"older's inequality we have
\begin{align}
    \label{equ:interp disc Lap nab L2}
    \norm{\nabla \bff{v}_h}{\bb{L}^2}^2
    &\leq
    \norm{\bff{v}_h}{\bb{L}^p} \norm{\Delta_h \bff{v}_h}{\bb{L}^q},
    \\
    \label{equ:interp disc Lap Delta L2}
    \norm{\Delta_h \bff{v}_h}{\bb{L}^2}^2
    &\leq
    \norm{\nabla \bff{v}_h}{\bb{L}^p} \norm{\nabla\Delta_h \bff{v}_h}{\bb{L}^q}.
\end{align}

\subsection{Identities and inequalities}

Some identities and inequalities that are frequently used in the analysis are collected in this section. Recall that $f(\bff{v})=\bff{v}-|\bff{v}|^2 \bff{v}$, where $\bff{v}:\mathscr{D}\to \bb{R}^3$. For $\varphi_R$ and $f_R$ defined in~\eqref{equ:bump} and~\eqref{equ:f R}, respectively, we have the following identities:
\begin{align}
    \label{equ:nab f}
    \nabla f(\bff{v})
    &=
    \nabla \bff{v}
    -
    2 \bff{v} (\bff{v}\cdot \nabla\bff{v}) 
	- |\bff{v}|^2 \nabla \bff{v},
    \\
    \label{equ:Delta f}
    \Delta f(\bff{v})
    &=
    \Delta \bff{v}
    -
    2|\nabla \bff{v}|^2 \bff{v} 
	- 2(\bff{v}\cdot \Delta \bff{v})\bff{v} 
	- 4 \nabla \bff{v} \ (\bff{v}\cdot \nabla\bff{v})^\top
	- |\bff{v}|^2 \Delta \bff{v},
    \\
    \label{equ:nab fR}
    \nabla f_R(\bff{v})
    &=
    \nabla \big[\varphi_R(\bff{v})\big] \big(f(\bff{v})\big)^\top 
    +
    \varphi_R(\bff{v}) \nabla \big[f(\bff{v})\big],
    \\
    \label{equ:Delta fR}
    \Delta f_R(\bff{v})
    &=
    \varphi_R(\bff{v}) \Delta \big[f(\bff{v})\big]
    +
    f(\bff{v}) \Delta\big[\varphi_R(\bff{v})\big]
    +
    2\nabla \big[f(\bff{v})\big] \cdot \nabla \big[\varphi_R (\bff{v})\big]^\top.
    \\
    \label{equ:Delta phi v}
    \Delta \varphi_R (\bff{v})
    &=
     \big(D\varphi_R(\bff{v})\big) \Delta \bff{v}
     +
    \sum_{i=1}^d \big(\partial_i \bff{v}\big)^\top \big(D^2 \varphi_R(\bff{v})\big) \big(\partial_i \bff{v}\big).
\end{align}
where $D\varphi_R(\bff{v})$ and $D^2 \varphi_R(\bff{v})$ denotes, respectively, the Jacobian and the Hessian of $\varphi_R$ evaluated at $\bff{v}$.


\begin{lemma}
	Let $\mathscr{D} \subset \bb{R}^d$ be an open bounded domain with Lipschitz boundary and $\epsilon>0$ be
	given. Then there exists a positive constant $C$ such that the following
	inequalities hold:
	\begin{enumerate}
		\renewcommand{\labelenumi}{\theenumi}
		\renewcommand{\theenumi}{{\rm (\roman{enumi})}}
        \item For any $\bff{v}\in \bb{H}^1$ and $p \in (2,6)$,
        \begin{align}\label{equ:v Lp interp}
            \norm{\bff{v}}{\bb{L}^p}
            \leq
            C_\epsilon \norm{\bff{v}}{\bb{L}^2}^2
            +
            \epsilon \norm{\nabla \bff{v}}{\bb{L}^2}^2.
        \end{align}
  
		\item For any $\bff{v},\bff{w} \in \bb{H}^s$, where $s>d/2$,
		\begin{align}
		\label{equ:prod Hs mat dot}
			\norm{\bff{v} \odot \bff{w}}{\bb{H}^s}
			&\leq
			C \norm{\bff{v}}{\bb{H}^s}
			\norm{\bff{w}}{\bb{H}^s},
		\end{align}
		Here $\odot$ denotes either the dot product or cross product.
  
        \item Let $\mathscr{D}$ be a convex polygonal or polyhedral domain with globally quasi-uniform triangulation. Let $\Delta_h$ be the discrete Laplacian operator defined in~\eqref{equ:disc laplacian}. For any $\bff{v}_h\in \bb{V}_h$,
		\begin{align}
			\label{equ:disc lapl L infty}
			\norm{\bff{v}_h}{\bb{L}^\infty}
			&\leq
			C \norm{\bff{v}_h}{\bb{L}^2}^{1-\frac{d}{4}} \left(\norm{\bff{v}_h}{\bb{L}^2}^\frac{d}{4} + \norm{\Delta_h \bff{v}_h}{\bb{L}^2}^\frac{d}{4} \right),
			\\
			\label{equ:disc lapl L6}
			\norm{\nabla \bff{v}_h}{\bb{L}^6}
			&\leq
			C \norm{\Delta_h \bff{v}_h}{\bb{L}^2}.
		\end{align}
	\end{enumerate}
\end{lemma}

\begin{proof}
	Inequality~\eqref{equ:v Lp interp} follows from the Gagliardo--Nirenberg inequalities. Inequality~\eqref{equ:prod Hs mat dot} is shown in~\cite[Lemma~2.2]{SoeTra23b}. The estimates~\eqref{equ:disc lapl L infty} and~\eqref{equ:disc lapl L6} are shown in~\cite[Appendix A]{GuiLiWan22}).
\end{proof}

Some estimates for the nonlinear terms are derived in the following lemmas.

\begin{lemma}
Let $\varphi_R$ and $f_R$ be the maps defined in~\eqref{equ:bump} and~\eqref{equ:f R}, respectively. Let $p,q\in [1,\infty]$. Then there exists a positive constant $C_R$, depending on $R$, such that the following inequalities hold:
\begin{enumerate}
    \item For any $\bff{v}:\mathscr{D}\to \bb{R}^3$,
    \begin{align}
        \label{equ:fR v L2}
        \norm{\nabla f_R(\bff{v})}{\bb{L}^p}
        &\leq
        C_R \norm{\nabla \bff{v}}{\bb{L}^p},
        \\
        \label{equ:fR v H2}
        \norm{\Delta f_R(\bff{v})}{\bb{L}^p}
        &\leq
        C_R \left( \norm{\nabla \bff{v}}{\bb{L}^{2p}}^2 + \norm{\Delta \bff{v}}{\bb{L}^p}\right).
    \end{align}
    \item Suppose that $1/p+1/q=1/2$. For any $\bff{v}:\mathscr{D}\to \bb{R}^3$ and $\bff{w}:\mathscr{D}\to \bb{R}^3$,
    \begin{align}
        \label{equ:fR v w H1}
        \norm{\nabla f_R(\bff{v})- \nabla f_R(\bff{w})}{\bb{L}^2}
        &\leq
        C_R \left(1+\norm{\bff{v}}{\bb{L}^\infty}^3\right) \norm{\nabla\bff{v}-\nabla\bff{w}}{\bb{L}^2}
        \nonumber\\
        &\quad
        +
        C_R \left(1+\norm{\bff{v}}{\bb{L}^\infty}^3\right) \norm{\nabla\bff{v}}{\bb{L}^p} \norm{\bff{v}-\bff{w}}{\bb{L}^q}.
    \end{align}
\end{enumerate}
\end{lemma}

\begin{proof}
Firstly, by~\eqref{equ:nab f} and \eqref{equ:nab fR} it is clear that we have
\begin{align*}
    \abs{\nabla f_R(\bff{v})}
    &\leq
    \abs{D\varphi_R(\bff{v})} \abs{\nabla \bff{v}} \abs{f(\bff{v})}
    +
    \abs{\varphi_R(\bff{v})} \abs{\bff{v}}^2 \abs{\nabla \bff{v}}
    \leq
    C_R \abs{\nabla \bff{v}},
\end{align*}
which implies~\eqref{equ:fR v L2}. Similarly, noting \eqref{equ:Delta f}, \eqref{equ:Delta fR}, and \eqref{equ:Delta phi v} we have
\begin{align*}
    \abs{\Delta f_R(\bff{v})} 
    &\leq 
    \abs{\varphi_R(\bff{v})} \left(\abs{\Delta\bff{v}}+ 6\abs{\nabla\bff{v}}^2 \abs{\bff{v}} + 3 \abs{\bff{v}}^2 \abs{\Delta\bff{v}}\right)
    \\
    &\quad
    +
    \left(\abs{\bff{v}}+\abs{\bff{v}}^3\right) \left(\abs{D\varphi_R(\bff{v})} \abs{\Delta \bff{v}}+ d \abs{\nabla \bff{v}}^2 \abs{D^2\varphi_R(\bff{v})}\right) 
    \\
    &\quad
    +
    \abs{D\varphi_R(\bff{v})} \abs{\nabla\bff{v}} \left(\abs{\nabla\bff{v}}+3 \abs{\bff{v}}^2 \abs{\nabla\bff{v}}\right)
    \\
    &\leq
    C_R \left(\abs{\nabla\bff{v}}^2 + \abs{\Delta \bff{v}}\right),
\end{align*}
which implies~\eqref{equ:fR v H2}. 
Next, using \eqref{equ:nab fR} again, we have
\begin{align}\label{equ:abs nab fr}
    \abs{\nabla f_R(\bff{v})-\nabla f_R(\bff{w})}
    &\leq
    \abs{\nabla\big[\varphi_R(\bff{v})-\varphi_R(\bff{w})\big]} \abs{f(\bff{v})}
    +
    \abs{\nabla\big[\varphi_R(\bff{w})\big]} \abs{f(\bff{v})-f(\bff{w})}
    \nonumber\\
    &\quad
    +
    \abs{\varphi_R(\bff{v})-\varphi_R(\bff{w})} \abs{\nabla\big[f(\bff{v})\big]}
    +
    \abs{\varphi_R(\bff{w})} \abs{\nabla\big[f(\bff{v})-f(\bff{w})\big]}.
\end{align}
We now estimate each term on the right-hand side. Firstly,
\begin{align*}
    \abs{\nabla\big[\varphi_R(\bff{v})-\varphi_R(\bff{w})\big]} \abs{f(\bff{v})}
    &\leq
    \abs{D\varphi_R(\bff{w})} \abs{\nabla\bff{v}-\nabla\bff{w}} \abs{f(\bff{v})}
    +
    \abs{D\varphi_R(\bff{w})-D\varphi_R(\bff{v})} \abs{\nabla\bff{v}} \abs{f(\bff{v})}
    \\
    &\leq
    C_R \left(1+\abs{\bff{v}}^3\right) \abs{\nabla\bff{v}-\nabla\bff{w}}
    +
    C_R \left(1+\abs{\bff{v}}^3 \right) \abs{\nabla\bff{v}} \abs{\bff{v}-\bff{w}}.
\end{align*}
Similarly, we have
\begin{align*}
    \abs{\nabla\big[\varphi_R(\bff{w})\big]} \abs{f(\bff{v})-f(\bff{w})}
    &\leq
    \abs{D\varphi_R(\bff{w})} \abs{\nabla\bff{w}-\nabla\bff{v}} \abs{f(\bff{v})-f(\bff{w})}
    +
    \abs{D\varphi_R(\bff{w})} \abs{\nabla\bff{v}} \abs{f(\bff{v})-f(\bff{w})}
    \\
    &\leq
    C_R \left(1+\abs{\bff{v}}^3\right) \abs{\nabla\bff{v}-\nabla\bff{w}}
    +
    C_R \left(1+\abs{\bff{v}}^2\right) \abs{\nabla \bff{v}} \abs{\bff{v}-\bff{w}},
\end{align*}
and
\begin{align*}
    \abs{\varphi_R(\bff{v})-\varphi_R(\bff{w})} \abs{\nabla\big[f(\bff{v})\big]}
    &\leq
    C_R \left(1+\abs{\bff{v}}^2\right) \abs{\nabla \bff{v}} \abs{\bff{v}-\bff{w}}.
\end{align*}
Finally, for the last term, we note that
\begin{align*}
    \nabla f(\bff{v})-\nabla f(\bff{w})
    &=
    \left(\nabla \bff{v}-\nabla\bff{w}\right)
    -2
    \Big(\bff{v}((\bff{v}-\bff{w})\cdot\nabla\bff{v})
    +
    (\bff{v}-\bff{w})(\bff{w}\cdot\nabla\bff{v})
    +
    \bff{w}(\bff{w}\cdot (\nabla\bff{v}-\nabla\bff{w}))\Big)
    \\
    &\quad
    -
    \left(((\bff{v}-\bff{w})\cdot (\bff{v}+\bff{w})) \nabla\bff{v}
    +
    \abs{\bff{w}}^2 (\nabla\bff{v}-\nabla\bff{w})\right).
\end{align*}
This implies
\begin{align*}
    \abs{\varphi_R(\bff{w})} \abs{\nabla\big[f(\bff{v})-f(\bff{w})\big]}
    &\leq
    C_R\abs{\nabla\bff{v}-\nabla\bff{w}}
    +
    C_R \left(1+\abs{\bff{v}}\right) \abs{\nabla\bff{v}} \abs{\bff{v}-\bff{w}}.
\end{align*}
Thus, continuing from \eqref{equ:abs nab fr} we obtain
\begin{align*}
    \abs{\nabla f_R(\bff{v})-\nabla f_R(\bff{w})}
    &\leq
    C_R \left(1+\abs{\bff{v}}^3\right) \abs{\nabla\bff{v}-\nabla\bff{w}}
    +
    C_R \left(1+\abs{\bff{v}}^3\right) \abs{\nabla \bff{v}} \abs{\bff{v}-\bff{w}},
\end{align*}
from which \eqref{equ:fR v w H1} follows by H\"older's inequality.
\end{proof}

\begin{lemma}
Let $\mathcal{C}$ be the map defined in \eqref{equ:torque} and $\bff{\nu}$ be given. For each $\epsilon>0$, there exists a positive constant $C$ such that for any $\bff{v}\in \bb{H}^1$ and $\bff{w}\in \bb{L}^\infty$,
\begin{align}
	\label{equ:Rv v}
	\big| \inpro{\mathcal{C}(\bff{v})}{\bff{v}}_{\bb{L}^2} \big|
	&\leq 
	C\left(1+\norm{\bff{\nu}}{\bb{L}^\infty(\mathscr{D};\bb{R}^d)}^2\right) \norm{\bff{v}}{\bb{L}^2}^2
	+
	\epsilon \norm{\nabla \bff{v}}{\bb{L}^2}^2,
	\\
    \label{equ:Rv w infty}
    \big| \inpro{\mathcal{C}(\bff{v})}{\bff{w}}_{\bb{L}^2} \big|
	&\leq
    C \left(1+\norm{\bff{\nu}}{\bb{L}^\infty(\mathscr{D};\bb{R}^d)}^2\right) \left(1+ \norm{\bff{w}}{\bb{L}^\infty}^2 \right) \norm{\bff{v}}{\bb{L}^2}^2
    +
    \epsilon \norm{\nabla \bff{v}}{\bb{L}^2}^2.
\end{align}
Now, let $p,q,r \in [1,\infty]$ be such that $1/p+1/q+1/r=1$. For each $\epsilon>0$, there exists a positive constant $C$ such that for any $\bff{v}\in \bb{W}^{1,q}\cap \bb{L}^p$ and $\bff{w} \in \bb{L}^r$,
\begin{align} 
	\label{equ:Rv w}
	\big| \inpro{\mathcal{C}(\bff{v})}{\bff{w}}_{\bb{L}^2} \big|
	&\leq
	C \norm{\bff{\nu}}{\bb{L}^\infty(\mathscr{D};\bb{R}^d)}^2 
    \left(1 
	+
	\norm{\bff{v}}{\bb{L}^p}^2 \right) \norm{\nabla \bff{v}}{\bb{L}^q}^2
	+
	\epsilon \norm{\bff{w}}{\bb{L}^r}^2.
\end{align}
\end{lemma}

\begin{proof}
Inequalities \eqref{equ:Rv v} and \eqref{equ:Rv w infty} follow directly from Young's inequality. To show \eqref{equ:Rv w}, we apply H\"older's and Young's inequalities to obtain
\begin{align*}
	\big| \inpro{\mathcal{C}(\bff{v})}{\bff{w}}_{\bb{L}^2} \big|
	&\leq
	C \norm{\bff{\nu}}{\bb{L}^p(\mathscr{D};\bb{R}^d)} \norm{\nabla \bff{v}}{\bb{L}^q} \norm{\bff{w}}{\bb{L}^r}
	+
	\norm{\bff{\nu}}{\bb{L}^\infty(\mathscr{D};\bb{R}^d)} \norm{\bff{v}}{\bb{L}^p} \norm{\nabla \bff{v}}{\bb{L}^q} \norm{\bff{w}}{\bb{L}^r}
	\\
	&\leq
	C \norm{\bff{\nu}}{\bb{L}^\infty(\mathscr{D};\bb{R}^d)}^2 
	\left(1 
	+
	\norm{\bff{v}}{\bb{L}^p}^2 \right) \norm{\nabla \bff{v}}{\bb{L}^q}^2
	+
	\epsilon \norm{\bff{w}}{\bb{L}^r}^2,
\end{align*}
as required. This completes the proof of the lemma.
\end{proof}

\section{A fully-discrete mixed finite element method}\label{sec:mixed fem}

In this section, we propose a mixed finite element method for~\eqref{equ:sllbar} with partially implicit Euler time-stepping. We start with $\bff{u}_h^0= \Pi_h \bff{u}(0) \in \bb{V}_h$. Let $t_n\in [0,T]$, where $n\in \{1,2,\ldots, N\}$ and $N=\lfloor T/k \rfloor$, given $\bff{u}_h^{n-1} \in \bb{V}_h$, we find $\mathcal{F}_{t_n}$-adapted and $\bb{V}_h\times \bb{V}_h$-valued random variables $\{(\bff{u}_h^n,\bff{H}_h^n)\}$ satisfying $\bb{P}$-a.s.,
\begin{align}\label{equ:euler}
	\left\{
	\begin{alignedat}{1}
		\inpro{\bff{u}_h^n-\bff{u}_h^{n-1}}{\bff{\chi}_h}
		&=
		k \inpro{\bff{H}_h^n}{\bff{\chi}_h}
		+
		k \inpro{\nabla \bff{H}_h^n}{\nabla \bff{\chi}_h}
		-
		k\gamma\inpro{\bff{u}_h^n \times \bff{H}_h^n}{\bff{\chi}_h}
		+
		k\inpro{\mathcal{C}(\bff{u}_h^n)}{\bff{\chi}_h}
		\\
		&\quad
        +
		k\inpro{\mathcal{M}_R(\bff{u}_h^{n-1})}{\bff{\chi}_h}
		+
		\inpro{G(\bff{u}_h^{n-1})}{\bff{\chi}_h} \overline{\Delta}W^n,
		\\
		\inpro{\bff{H}_h^n}{\bff{\phi}_h} 
		&=
		-\inpro{\nabla \bff{u}_h^n}{\nabla \bff{\phi}_h}
		+
		\inpro{f_R(\bff{u}_h^{n-1})}{\bff{\phi}_h},
	\end{alignedat}
	\right.
\end{align}
for all $\bff{\chi}_h, \bff{\phi}_h \in \bb{V}_h$. Here, $\overline{\Delta}W^n:= W(t_n)-W(t_{n-1}) \sim \mathcal{N}(0,k)$. 

In particular, when $\gamma=\beta_2=0$, this is a fully-discrete IMEX-type scheme for the sCHm or the sAC/CH equations. Subsequently, we set $\gamma=1$ in the analysis for ease of presentation.

\begin{lemma}
There exists a sequence $\{(\bff{u}_h^n,\bff{H}_h^n)\}$ of $\bb{V}_h\times\bb{V}_h$-valued random variables which solves~\eqref{equ:euler}.
\end{lemma}

\begin{proof}
Fix $\omega\in\Omega$. We aim to use induction and a form of Brouwer's fixed point theorem to show the existence of a sequence $\{\bff{u}_h^n(\omega)\}_{n=1}^N$ solving \eqref{equ:euler}. Suppose that $\bff{u}_h^0(\omega), \bff{u}_h^1(\omega),\ldots,\bff{u}_h^{n-1}(\omega)$ are given. Consider a continuous map $\mathcal{G}_n^\omega: \bb{V}_h \to \bb{V}_h$ defined by
\begin{align*}
    \mathcal{G}_n^\omega (\bff{v})
    &=
    \bff{v}-\bff{u}_h^{n-1}(\omega)
    -
    k \Delta_h \bff{v}
    -
    k \Pi_h f_R\big(\bff{u}_h^{n-1}(\omega)\big)
    +
    k\Delta_h^2 \bff{v}
    +
    k\Delta_h \Pi_h f_R\big(\bff{u}_h^{n-1}(\omega)\big)
        \\
    &\quad
    +
    k \bff{v} \times \left(\Delta_h \bff{v}+ \Pi_h f_R(\bff{u}_h^{n-1}(\omega))\right)
    -
    k \Pi_h \mathcal{C}(\bff{v})
    -
    k \Pi_h \mathcal{M}_R(\bff{u}_h^{n-1}(\omega))
    \\
    &\quad
    -
    \Pi_h G(\bff{u}_h^{n-1}(\omega)) \overline{\Delta}W^n(\omega).
\end{align*}
For all $\bff{v}\in\bb{V}_h$, by Young's inequality, Lipschitz continuity of $G$, $f_R$, and~$\mathcal{M}_R$, and \eqref{equ:fR v L2} we have
\begin{align*}
    \inpro{\mathcal{G}_n^\omega(\bff{v})}{\bff{v}}
    &\geq
    \frac12 \norm{\bff{v}}{\bb{L}^2}^2
    -
    \frac12 \norm{\bff{u}_h^{n-1}(\omega)}{\bb{L}^2}^2
    +
    k\norm{\nabla \bff{v}}{\bb{L}^2}^2
    +
    k\norm{\Delta_h \bff{v}}{\bb{L}^2}^2
    -
    k\inpro{\Pi_h f_R(\bff{u}_h^{n-1}(\omega))}{\bff{v}}
    \\
    &\quad
    -
    k\inpro{\nabla\Pi_h f_R(\bff{u}_h^{n-1}(\omega))}{\nabla\bff{v}}
    -
    k\inpro{(\bff{\nu}\cdot \nabla)\bff{v}}{\bff{v}}
    -
    k\inpro{\mathcal{M}_R(\bff{v})}{\bff{v}}
    \\
    &\quad
    -
    \inpro{G(\bff{u}_h^{n-1}(\omega)) \overline{\Delta}W^n(\omega)}{\bff{v}}
    \\
    &\geq
    \frac12 \norm{\bff{v}}{\bb{L}^2}^2
    -
    \frac12 \norm{\bff{u}_h^{n-1}(\omega)}{\bb{L}^2}^2
    +
    k\norm{\nabla \bff{v}}{\bb{L}^2}^2
    +
    k\norm{\Delta_h \bff{v}}{\bb{L}^2}^2
    -
    2C_R k \norm{\bff{u}_h^{n-1}(\omega)}{\bb{L}^2}^2
    -
    \frac{k}{8} \norm{\bff{v}}{\bb{L}^2}^2
    \\
    &\quad
    -
    \frac{C_R k}{4} \norm{\bff{u}_h^{n-1}(\omega)}{\bb{L}^2}^2
    -
    k\norm{\Delta_h \bff{v}}{\bb{L}^2}^2
    -
    k \norm{\nabla \bff{v}}{\bb{L}^2}^2
    -
    \frac{k}{4} \norm{\bff{v}}{\bb{L}^2}^2
    -
    C_R k \norm{\bff{v}}{\bb{L}^2}^2
    \\
    &\quad
    -
    2C_G  \norm{\bff{u}_h^{n-1}(\omega)}{\bb{L}^2}^2 \abs{\overline{\Delta} W^n(\omega)}^2
    -
    \frac{k}{8}\norm{\bff{v}}{\bb{L}^2}^2
    \\
    &=
    \frac12 \left(1-k-2C_R k\right) \norm{\bff{v}}{\bb{L}^2}^2
    -
    \frac14 \left(9C_R k+ 8C_G \abs{\overline{\Delta} W^n(\omega)}^2 \right) \norm{\bff{u}_h^{n-1}(\omega)}{\bb{L}^2}^2.
\end{align*}
Now, suppose that $k<(1+2C_R)^{-1}$, and set $\beta:=1-k-2C_R k>0$. Let
\[
\mathcal{B}_n:=\left\{\bff{\varphi}\in \bb{V}_h: \norm{\bff{\varphi}}{\bb{L}^2}^2= 4\beta^{-1} \Lambda_n(\omega)\right\},
\]
where
\[
    \Lambda_n(\omega):= \frac14 \left(9C_R k+ 8C_G \abs{\overline{\Delta} W^n(\omega)}^2 \right) \norm{\bff{u}_h^{n-1}(\omega)}{\bb{L}^2}^2 <\infty.
\]
Then we have $\inpro{\mathcal{G}_n^\omega(\bff{v})}{\bff{v}}\geq 0$ for all $\bff{v}\in \mathcal{B}_n(\omega)$. Brouwer's fixed point theorem~\cite[Corollary VI.1.1]{GirRav86} implies the existence of $\bff{u}_h^n(\omega)$ such that $\mathcal{G}_n^\omega\big(\bff{u}_h^n(\omega)\big)=0$, thus also of $\bff{H}_h^n(\omega)=\Delta_h \bff{u}_h^n(\omega) +\Pi_h f_R\big(\bff{u}_h^{n-1}(\omega)\big)$. The $\mathcal{F}_{t_n}$-measurability of the map $\bff{u}_h^n:\Omega\to \bb{V}_h$, thus also of $\bff{H}_h^n$, follows from the same argument as in~\cite[Theorem~2.2]{BanBrzNekPro14}.
\end{proof}

We establish some stability estimates in the following lemmas.

\begin{lemma}\label{lem:stab L2}
Let $p\in [1,\infty)$ be a natural number. Suppose that $\{(\bff{u}_h^n, \bff{H}_h^n)\}$ satisfies~\eqref{equ:euler}. There exists a positive constant $C$ such that
\begin{align}\label{equ:E u 2p general}
	&\bb{E} \left[ \max_{l\leq n} \norm{\bff{u}_h^l}{\bb{L}^2}^{2^p} \right]
	+
	\bb{E}\left[ \sum_{j=1}^n \norm{\bff{u}_h^j-\bff{u}_h^{j-1}}{\bb{L}^2}^2 \norm{\bff{u}_h^j}{\bb{L}^2}^{2^p-2} \right]
    \nonumber\\
	&\quad
	+
	\bb{E} \left[ \sum_{j=1}^n k \left(\norm{\nabla \bff{u}_h^j}{\bb{L}^2}^2 
	+ \norm{\Delta_h \bff{u}_h^j}{\bb{L}^2}^2 \right) \norm{\bff{u}_h^j}{\bb{L}^2}^{2^p-2}  \right]
	+
	\bb{E}\left[ \left(k \sum_{j=1}^n \norm{\nabla \bff{u}_h^j}{\bb{L}^2}^2 \right)^{2^{p-1}} \right] 
	\nonumber\\
	&\quad
	+
	\bb{E} \left[
	\left(k \sum_{j=1}^n \norm{\Delta_h \bff{u}_h^j}{\bb{L}^2}^2 \right)^{2^{p-1}} \right]
	+
	{\bb{E} \left[
	\left(k \sum_{j=1}^n \norm{\bff{H}_h^j}{\bb{L}^2}^2 \right)^{2^{p-1}} \right]}
	\leq
	C.
\end{align}
where $C$ depends on $T$, $p$, $R$, and $\norm{\bff{u}_0}{\bb{L}^2}$, but is independent of $n$, $h$, and $k$.
\end{lemma}

\begin{proof}
We begin the proof by showing the inequality for $p=1$. Setting $\bff{\chi}_h= \bff{u}_h^n$, we have
\begin{align}\label{equ:uhn min uhn L2}
    \frac12 \left(\norm{\bff{u}_h^n}{\bb{L}^2}^2 - \norm{\bff{u}_h^{n-1}}{\bb{L}^2}^2 \right)
    +
    \frac12 \norm{\bff{u}_h^n- \bff{u}_h^{n-1}}{\bb{L}^2}^2
    &=
    k\inpro{\bff{H}_h^n}{\bff{u}_h^n}
    +
    k\inpro{\nabla \bff{H}_h^n}{\nabla\bff{u}_h^n}
    +
    k\inpro{\mathcal{C}(\bff{u}_h^n)}{\bff{u}_h^n}
    \nonumber \\
    &\quad
    +
    k\inpro{\mathcal{M}_R(\bff{u}_h^{n-1})}{\bff{u}_h^n}
    +
    \inpro{G(\bff{u}_h^{n-1})}{\bff{u}_h^n} \overline{\Delta}W^n.
\end{align}
Successively taking $\bff{\phi}_h= \bff{u}_h^n$ and $\bff{\phi}_h= -\Delta_h \bff{u}_h^n$, then substituting the results to~\eqref{equ:uhn min uhn L2}, we obtain
\begin{align}\label{equ:ineq uhn L2}
    &\frac12 \left(\norm{\bff{u}_h^n}{\bb{L}^2}^2 - \norm{\bff{u}_h^{n-1}}{\bb{L}^2}^2 \right)
    +
    \frac12 \norm{\bff{u}_h^n- \bff{u}_h^{n-1}}{\bb{L}^2}^2
    +
    k\norm{\nabla \bff{u}_h^n}{\bb{L}^2}^2
    +
    k\norm{\Delta_h \bff{u}_h^n}{\bb{L}^2}^2
    \nonumber \\
    &=
    k\inpro{f_R(\bff{u}_h^{n-1})}{\bff{u}_h^n}
    -
    k\inpro{f_R(\bff{u}_h^{n-1})}{\Delta_h \bff{u}_h^n}
    +
    k\inpro{\mathcal{C}(\bff{u}_h^n)}{\bff{u}_h^n}
    +
    k\inpro{\mathcal{M}_R(\bff{u}_h^{n-1})}{\bff{u}_h^n}
    \nonumber\\
    &\quad
    +
    \inpro{G(\bff{u}_h^{n-1})}{\bff{u}_h^{n-1}} \overline{\Delta}W^n 
    + 
    \inpro{G(\bff{u}_h^{n-1})}{\bff{u}_h^n-\bff{u}_h^{n-1}} \overline{\Delta} W^n
    \nonumber\\
    &=: J_1+J_2+\cdots+J_6.
\end{align}
We will estimate each term on the last line. By Lipschitz continuity of $f_R$, $G$, and $\mathcal{M}_R$, \eqref{equ:Rv w}, and Young's inequality we have
\begin{align*}
    \abs{J_1}
    &\leq
    C_R k \norm{\bff{u}_h^{n-1}}{\bb{L}^2}^2 + \frac{k}{2} \norm{\bff{u}_h^n}{\bb{L}^2}^2,
    \\
    \abs{J_2}
    &\leq
    C_R k \norm{\bff{u}_h^{n-1}}{\bb{L}^2}^2
    +
    \frac{k}{2} \norm{\Delta_h \bff{u}_h^n}{\bb{L}^2}^2,
    \\
    \abs{J_3}+\abs{J_4}
    &\leq
    C_R k \norm{\bff{u}_h^{n-1}}{\bb{L}^2}^2
    + \frac18 \norm{\bff{u}_h^n-\bff{u}_h^{n-1}}{\bb{L}^2}^2
    + \frac{k}{2} \norm{\nabla \bff{u}_h^n}{\bb{L}^2}^2,
\end{align*}
We aim to estimate the moments of the last two terms. To this end, note that by Young's inequality,
\begin{align}\label{equ:g h L2}
    \abs{J_6}
    &\leq
    \frac14 \norm{\bff{u}_h^n-\bff{u}_h^{n-1}}{\bb{L}^2}^2
    +
    \norm{G(\bff{u}_h^{n-1})}{\bb{L}^2}^2 \abs{\overline{\Delta} W^n}^2.
\end{align}
By the tower property of the conditional expectation and the independence of the Wiener increment, we infer that
\begin{align}\label{equ:E Delta W L2 1}
    \bb{E}\left[ \max_{l\leq n} \sum_{j=1}^l \norm{G(\bff{u}_h^{j-1})}{\bb{L}^2}^2  \abs{\overline{\Delta} W^j}^2\right]
    &=
    \sum_{j=1}^n \bb{E}\left[\bb{E}\Big[\norm{G(\bff{u}_h^{j-1})}{\bb{L}^2}^2 \abs{\overline{\Delta} W^j}^2 \Big| \mathcal{F}_{t_{j-1}}\Big] \right]
    \nonumber\\
    &=
    \sum_{j=1}^n \bb{E}\left[ \norm{G(\bff{u}_h^{j-1})}{\bb{L}^2}^2 \bb{E}\Big[\abs{\overline{\Delta} W^j}^2 \Big| \mathcal{F}_{t_{j-1}}\Big] \right]
    \nonumber\\
    &\leq
    Ck \sum_{j=1}^n  \bb{E} \left[1+\norm{\bff{u}_h^{j-1}}{\bb{L}^2}^{2} \right].
\end{align}
Noting the assumptions on $G$, we also have by the Burkholder--Davis--Gundy inequality,
\begin{align}\label{equ:E Delta W L2 2}
    \bb{E} \left[ \max_{l\leq n} \sum_{j=1}^l \inpro{G(\bff{u}_h^{j-1})}{\bff{u}_h^{j-1}} \overline{\Delta}W^j \right]
    &\leq
    C\bb{E} \left[\left(k \sum_{j=1}^n \left(1+\norm{\bff{u}_h^{j-1}}{\bb{L}^2}^2\right) \norm{ \bff{u}_h^{j-1}}{\bb{L}^2}^2 \right)^{\frac12} \right]
    \nonumber\\
    &\leq
    C\bb{E} \left[ \max_{j\leq n} \left(1+\norm{\bff{u}_h^{j-1}}{\bb{L}^2}^2\right)^\frac12 \left(k \sum_{j=1}^n \norm{\bff{u}_h^{j-1}}{\bb{L}^2}^2 \right)^{\frac12} \right] 
    \nonumber\\
    &\leq
    C+
    \frac14 \bb{E}\left[\max_{l\leq n} \norm{\bff{u}_h^l}{\bb{L}^2}^2\right]
    +
    C \bb{E}\left[ \sum_{j=1}^n k\norm{\bff{u}_h^{j-1}}{\bb{L}^2}^2 \right].
\end{align}
As such, summing \eqref{equ:ineq uhn L2} over $j\in \{1,2,\ldots, l\}$, taking the maximum over $l$, applying the expected value, and absorbing the second term in~\eqref{equ:E Delta W L2 2} to the right-hand side, we infer from~\eqref{equ:g h L2}, \eqref{equ:E Delta W L2 1}, and~\eqref{equ:E Delta W L2 2} that
\begin{align*}
    &\bb{E} \left[ \max_{l\leq n} \norm{\bff{u}_h^l}{\bb{L}^2}^{2} \right] 
    +
    \bb{E}\left[ \sum_{j=1}^n \norm{\bff{u}_h^j-\bff{u}_h^{j-1}}{\bb{L}^2}^2\right]
    +
    \bb{E} \left[ \sum_{j=1}^n k \left(\norm{\nabla \bff{u}_h^n}{\bb{L}^2}^2 + \norm{\Delta_h \bff{u}_h^n}{\bb{L}^2}^2 \right)  \right]
    \\
    &\quad \leq
    \norm{\bff{u}_h^0}{\bb{L}^2}^{2}
    +
    C\left(1 +k \sum_{j=1}^n \bb{E} \left[ \max_{l\leq j} \norm{\bff{u}_h^{l-1}}{\bb{L}^2}^{2} \right] \right),
\end{align*}
where $C$ depends on $T$ and $R$. The first inequality then follows by the discrete Gronwall lemma.

Next, we aim to prove the second inequality. We will show the case $p=2$ in detail. Multiplying~\eqref{equ:ineq uhn L2} by $\norm{\bff{u}_h^n}{\bb{L}^2}^2$, we obtain
\begin{align}\label{equ:2est I1 to I5}
	&\frac14 \left(\norm{\bff{u}_h^n}{\bb{L}^2}^4 - \norm{\bff{u}_h^{n-1}}{\bb{L}^2}^4 + \left(\norm{\bff{u}_h^n}{\bb{L}^2}^2- \norm{\bff{u}_h^{n-1}}{\bb{L}^2}^2\right)^2 \right)
	+
	\frac12 \norm{\bff{u}_h^n- \bff{u}_h^{n-1}}{\bb{L}^2}^2 \norm{\bff{u}_h^n}{\bb{L}^2}^2
	\nonumber \\
	&
	+
	k\norm{\nabla \bff{u}_h^n}{\bb{L}^2}^2 \norm{\bff{u}_h^n}{\bb{L}^2}^2
	+
	k\norm{\Delta_h \bff{u}_h^n}{\bb{L}^2}^2 \norm{\bff{u}_h^n}{\bb{L}^2}^2
	\nonumber \\
	&=
	k\inpro{f_R(\bff{u}_h^{n-1})}{\bff{u}_h^n} \norm{\bff{u}_h^n}{\bb{L}^2}^2
	-
	k\inpro{f_R(\bff{u}_h^{n-1})}{\Delta_h \bff{u}_h^n} \norm{\bff{u}_h^n}{\bb{L}^2}^2
	+
	k\inpro{\mathcal{C}(\bff{u}_h^n)}{\bff{u}_h^n} \norm{\bff{u}_h^n}{\bb{L}^2}^2
    \nonumber\\
    &\quad
    +
    k\inpro{\mathcal{M}_R(\bff{u}_h^{n-1})}{\bff{u}_h^n} \norm{\bff{u}_h^n}{\bb{L}^2}^2
	+
	\inpro{G(\bff{u}_h^{n-1}) \overline{\Delta}W^n}{\bff{u}_h^{n-1}}  \norm{\bff{u}_h^n}{\bb{L}^2}^2
	\nonumber\\
	&\quad
	+ 
	\inpro{G(\bff{u}_h^{n-1}) \overline{\Delta} W^n}{\bff{u}_h^n-\bff{u}_h^{n-1}}  \norm{\bff{u}_h^n}{\bb{L}^2}^2
	\nonumber\\
	&= I_1+I_2+\cdots+I_6.
\end{align}
From the corresponding estimates for $J_1$ to $J_4$ in \eqref{equ:ineq uhn L2}, the first four terms can be estimated as:
\begin{align*}
	\abs{I_1}
	&\leq
	\frac{k}{16} \norm{\bff{u}_h^n-\bff{u}_h^{n-1}}{\bb{L}^2}^2 \norm{\bff{u}_h^n}{\bb{L}^2}^2
	+
	Ck\norm{\bff{u}_h^n}{\bb{L}^2}^4,
	\\
	\abs{I_2}
	&\leq
	\frac{k}{16} \norm{\Delta_h \bff{u}_h^n}{\bb{L}^2}^2 \norm{\bff{u}_h^n}{\bb{L}^2}^2
	+
	Ck\norm{\bff{u}_h^n}{\bb{L}^2}^4 + Ck\norm{\bff{u}_h^{n-1}}{\bb{L}^2}^4,
	\\
	\abs{I_3}+\abs{I_4}
	&\leq
	\frac{k}{16} \norm{\nabla \bff{u}_h^n}{\bb{L}^2}^2 \norm{\bff{u}_h^n}{\bb{L}^2}^2
    +
    \frac{k}{16} \norm{\bff{u}_h^n-\bff{u}_h^{n-1}}{\bb{L}^2}^2 \norm{\bff{u}_h^n}{\bb{L}^2}^2
	+
	Ck\norm{\bff{u}_h^n}{\bb{L}^2}^4.
\end{align*}
For the term $I_5$, by Young's inequality we have
\begin{align}\label{equ:I4 sto}
	I_5
	&=
	\inpro{G(\bff{u}_h^{n-1}) \overline{\Delta}W^n}{\bff{u}_h^{n-1}} \left[\left(\norm{\bff{u}_h^n}{\bb{L}^2}^2- \norm{\bff{u}_h^{n-1}}{\bb{L}^2}^2\right) + \norm{\bff{u}_h^{n-1}}{\bb{L}^2}^2 \right]
	\nonumber\\
	&\leq
	\frac{1}{16} \left(\norm{\bff{u}_h^n}{\bb{L}^2}^2- \norm{\bff{u}_h^{n-1}}{\bb{L}^2}^2\right)^2
	+
	C \norm{\bff{u}_h^{n-1}}{\bb{L}^2}^4 \abs{\overline{\Delta} W^n}^2
	+
	\inpro{G(\bff{u}_h^{n-1}) \overline{\Delta}W^n}{\bff{u}_h^{n-1}}  \norm{\bff{u}_h^{n-1}}{\bb{L}^2}^2.
\end{align}
Similarly, for the term $I_6$ we infer that
\begin{align}\label{equ:I5 sto}
	I_6
	&\leq
	C \norm{\bff{u}_h^{n-1}}{\bb{L}^2}^2 |\overline{\Delta} W^n|^2 \norm{\bff{u}_h^n}{\bb{L}^2}^2
	+
	\frac{1}{16} \norm{\bff{u}_h^n-\bff{u}_h^{n-1}}{\bb{L}^2}^2 \norm{\bff{u}_h^n}{\bb{L}^2}^2
	\nonumber\\
	&\leq
	\frac{1}{16} \left(\norm{\bff{u}_h^n}{\bb{L}^2}^2- \norm{\bff{u}_h^{n-1}}{\bb{L}^2}^2\right)^2
	+
	C\norm{\bff{u}_h^{n-1}}{\bb{L}^2}^4  \left(|\overline{\Delta} W^n|^4 + |\overline{\Delta} W^n|^2 \right)
	\nonumber\\
	&\quad
	+
	\frac{1}{16} \norm{\bff{u}_h^n-\bff{u}_h^{n-1}}{\bb{L}^2}^2 \norm{\bff{u}_h^n}{\bb{L}^2}^2.
\end{align}
Altogether, for sufficiently small $k$, we obtain
\begin{align}\label{equ:14 un 4}
	&\frac14 \left(\norm{\bff{u}_h^n}{\bb{L}^2}^4 - \norm{\bff{u}_h^{n-1}}{\bb{L}^2}^4\right) 
	+ 
	\frac18 \left(\norm{\bff{u}_h^n}{\bb{L}^2}^2- \norm{\bff{u}_h^{n-1}}{\bb{L}^2}^2\right)^2
	+
	\frac14 \norm{\bff{u}_h^n- \bff{u}_h^{n-1}}{\bb{L}^2}^2 \norm{\bff{u}_h^n}{\bb{L}^2}^2
	\nonumber \\
	&
	+
	\frac{k}{4} \norm{\nabla \bff{u}_h^n}{\bb{L}^2}^2 \norm{\bff{u}_h^n}{\bb{L}^2}^2
	+
	\frac{k}{4} \norm{\Delta_h \bff{u}_h^n}{\bb{L}^2}^2 \norm{\bff{u}_h^n}{\bb{L}^2}^2
	\nonumber \\
	&\leq
	Ck\norm{\bff{u}_h^n}{\bb{L}^2}^4 + Ck\norm{\bff{u}_h^{n-1}}{\bb{L}^2}^4
	+
	C\norm{\bff{u}_h^{n-1}}{\bb{L}^2}^4  \left(|\overline{\Delta} W^n|^4 + |\overline{\Delta} W^n|^2 \right)
	\nonumber\\
	&\quad
	+
	\inpro{G(\bff{u}_h^{n-1}) \overline{\Delta}W^n}{\bff{u}_h^{n-1}}  \norm{\bff{u}_h^{n-1}}{\bb{L}^2}^2.
\end{align}
We need to estimate the moment of the right-hand side of the last inequality. To this end, note that by the Burkholder--Davis--Gundy inequality, similarly to~\eqref{equ:E Delta W L2 2} we have
\begin{align}\label{equ:E sto}
	\bb{E} \left[ \max_{l\leq n} \sum_{j=1}^l \inpro{G(\bff{u}_h^{j-1}) \overline{\Delta}W^j}{\bff{u}_h^{j-1}}  \norm{\bff{u}_h^{j-1}}{\bb{L}^2}^2 \right]
	&\leq
	C\bb{E} \left[\left(k \sum_{j=1}^n \left(1+\norm{\bff{u}_h^{j-1}}{\bb{L}^2}^2\right) \norm{ \bff{u}_h^{j-1}}{\bb{L}^2}^6 \right)^{\frac12} \right]
	\nonumber\\
	&\leq
	C+
	\frac14 \bb{E}\left[\max_{l\leq n} \norm{\bff{u}_h^l}{\bb{L}^2}^4\right]
	+
	C \bb{E}\left[ \sum_{j=1}^n k\norm{\bff{u}_h^{j-1}}{\bb{L}^2}^4 \right].
\end{align}
With this estimate, we can continue from~\eqref{equ:14 un 4}. Summing~\eqref{equ:14 un 4} over $j\in \{1,2,\ldots,l\}$, taking the maximum over $l$, applying the expected value as before, we deduce the required inequality for $p=2$ by the discrete Gronwall lemma, except for the last two terms on the left-hand side. For general $p\geq 2$, the inductive step is as follows: once we obtain inequality of the form
\begin{align}\label{equ:ineq 2p}
	&\frac{1}{2^{p}} \left(\norm{\bff{u}_h^n}{\bb{L}^2}^{2^{p}} - \norm{\bff{u}_h^{n-1}}{\bb{L}^2}^{2^{p}} \right) 
	+ 
	\frac{1}{2^{p+1}} \left(\norm{\bff{u}_h^n}{\bb{L}^2}^{2^{p-1}}- \norm{\bff{u}_h^{n-1}}{\bb{L}^2}^{2^{p-1}} \right)^2 
	+
	\frac{1}{2^{p}} \norm{\bff{u}_h^n- \bff{u}_h^{n-1}}{\bb{L}^2}^2 \norm{\bff{u}_h^n}{\bb{L}^2}^{2^{p}-2}
	\nonumber \\
	&
	+
	k\norm{\nabla \bff{u}_h^n}{\bb{L}^2}^2 \norm{\bff{u}_h^n}{\bb{L}^2}^{2^{p}-2}
	+
	k\norm{\Delta_h \bff{u}_h^n}{\bb{L}^2}^2 \norm{\bff{u}_h^n}{\bb{L}^2}^{2^{p}-2}
	\nonumber\\
	&\leq
	Ck \norm{\bff{u}_h^n}{\bb{L}^2}^{2^{p}}
	+
	Ck \norm{\bff{u}_h^{n-1}}{\bb{L}^2}^{2^{p}}
	+
	C \norm{\bff{u}_h^{n-1}}{\bb{L}^2}^{2^p} \left(|\overline{\Delta} W^n|^{2^p} + |\overline{\Delta} W^n|^{2^{p-1}} \right)
	\nonumber\\
	&\quad
	+
	\inpro{G(\bff{u}_h^{n-1}) \overline{\Delta}W^n}{\bff{u}_h^{n-1}}  \norm{\bff{u}_h^{n-1}}{\bb{L}^2}^{2^p-2},
\end{align}
we multiply it by $\norm{\bff{u}_h^n}{\bb{L}^2}^{2^p}$. Note that the above inequality for $p=2$ is~\eqref{equ:14 un 4}. In this manner, we obtain
\begin{align*}
	&\frac{1}{2^{p+1}} \left(\norm{\bff{u}_h^n}{\bb{L}^2}^{2^{p+1}} - \norm{\bff{u}_h^{n-1}}{\bb{L}^2}^{2^{p+1}} \right) 
	+ 
	\frac{1}{2^{p+2}} \left(\norm{\bff{u}_h^n}{\bb{L}^2}^{2^{p}}- \norm{\bff{u}_h^{n-1}}{\bb{L}^2}^{2^{p}} \right)^2
	+
	\frac{1}{2^{p+1}} \norm{\bff{u}_h^n- \bff{u}_h^{n-1}}{\bb{L}^2}^2 \norm{\bff{u}_h^n}{\bb{L}^2}^{2^{p+1}-2}
	\nonumber \\
	&
	+
	k\norm{\nabla \bff{u}_h^n}{\bb{L}^2}^2 \norm{\bff{u}_h^n}{\bb{L}^2}^{2^{p+1}-2}
	+
	k\norm{\Delta_h \bff{u}_h^n}{\bb{L}^2}^2 \norm{\bff{u}_h^n}{\bb{L}^2}^{2^{p+1}-2}
	\nonumber\\
	&\leq
	Ck \norm{\bff{u}_h^n}{\bb{L}^2}^{2^{p+1}}
	+
	Ck \norm{\bff{u}_h^{n-1}}{\bb{L}^2}^{2^{p+1}}
	+
	C \norm{\bff{u}_h^{n-1}}{\bb{L}^2}^{2^p}  \left(|\overline{\Delta} W^n|^{2^p} + |\overline{\Delta} W^n|^{2^{p-1}} \right) \norm{\bff{u}_h^n}{\bb{L}^2}^{2^p}
	\nonumber\\
	&\quad
	+
	\inpro{G(\bff{u}_h^{n-1}) \overline{\Delta}W^n}{\bff{u}_h^{n-1}}  \norm{\bff{u}_h^{n-1}}{\bb{L}^2}^{2^{p+1}-2}
	\\
	&=: Ck \norm{\bff{u}_h^n}{\bb{L}^2}^{2^{p+1}}
	+
	Ck \norm{\bff{u}_h^{n-1}}{\bb{L}^2}^{2^{p+1}}
	+ S_1+S_2.
\end{align*}
For the term $S_1$, we add and subtract $\norm{\bff{u}_h^{n-1}}{\bb{L}^2}^{2^p}$, and apply Young's inequality to obtain
\begin{align*}
	S_1
	&\leq
	\frac{1}{2^{p+3}} \left(\norm{\bff{u}_h^n}{\bb{L}^2}^{2^{p-1}}- \norm{\bff{u}_h^{n-1}}{\bb{L}^2}^{2^{p-1}} \right)^2 
	+
	C\norm{\bff{u}_h^{n-1}}{\bb{L}^2}^{2^{p+1}} \left(|\overline{\Delta} W^n|^{2^{p+1}} + |\overline{\Delta} W^n|^{2^p} \right),
\end{align*}
and thus after rearranging, we obtain inequality of the form \eqref{equ:ineq 2p} with $p$ replaced by $p+1$. Now, for the term $S_2$, we can estimate its moment by the same argument as in~\eqref{equ:E sto}:
\begin{align*}
	&\bb{E} \left[ \max_{l\leq n} \sum_{j=1}^l \inpro{G(\bff{u}_h^{j-1}) \overline{\Delta}W^j}{\bff{u}_h^{j-1}}  \norm{\bff{u}_h^{j-1}}{\bb{L}^2}^{2^{p+1}-2} \right]
	\\
	&\leq
	C\bb{E} \left[\left(k \sum_{j=1}^n \left(1+\norm{\bff{u}_h^{j-1}}{\bb{L}^2}^4\right) \norm{ \bff{u}_h^{j-1}}{\bb{L}^2}^{2^{p+2}-4} \right)^{\frac12} \right]
	\nonumber\\
	&\leq
	C+
	\frac{1}{2^{p+2}} \bb{E}\left[\max_{l\leq n} \norm{\bff{u}_h^l}{\bb{L}^2}^{2^{p+1}} \right]
	+
	C \bb{E}\left[ \sum_{j=1}^n k\norm{\bff{u}_h^{j-1}}{\bb{L}^2}^{2^{p+1}} \right].
\end{align*}
Summing~\eqref{equ:ineq 2p} over $j\in \{1,2,\ldots, l\}$, taking the maximum over $l$ and the expected value, we obtain \eqref{equ:E u 2p general} for general $p$, except for the last two terms on the left-hand side. In particular, we have shown for any $q\geq 1$,
\begin{align}\label{equ:E max un q}
    \bb{E} \left[ \max_{l\leq n} \norm{\bff{u}_h^l}{\bb{L}^2}^{q} \right] \leq C,
\end{align}

Finally, we sum~\eqref{equ:ineq uhn L2} over $j\in \{1,2,\ldots, n\}$ and raise it to the $2^{p-1}$-th power. Noting~\eqref{equ:g h L2} and applying similar argument as before yield
\begin{align}\label{equ:E uh 2 4}
	&\norm{\bff{u}_h^n}{\bb{L}^2}^{2^p}
	+
	\left(k \sum_{j=1}^n \norm{\nabla \bff{u}_h^j}{\bb{L}^2}^2 \right)^{2^{p-1}}
	+
	\left(k \sum_{j=1}^n \norm{\Delta_h \bff{u}_h^j}{\bb{L}^2}^2 \right)^{2^{p-1}}
	\nonumber\\
	&\leq 
	C\left(k \sum_{j=1}^n \norm{\bff{u}_h^{j-1}}{\bb{L}^2}^2\right)^{2^{p-1}}
    +
	\left(\sum_{j=1}^n \norm{G(\bff{u}_h^{j-1}) \overline{\Delta} W^j}{\bb{L}^2}^2 \right)^{2^{p-1}}
	+
	\left(\sum_{j=1}^n \inpro{G(\bff{u}_h^{j-1})}{\bff{u}_h^{j-1}} \overline{\Delta} W^n\right)^{2^{p-1}}
	\nonumber\\
	&=:
	R_1+R_2+R_3.
\end{align}
The expected value of $R_1$ is clearly bounded by \eqref{equ:E max un q}. For $R_2$, we have by Jensen's inequality, \eqref{equ:E max un q}, and the same argument as in~\eqref{equ:E Delta W L2 1},
\begin{align*}
    \bb{E}[R_2] \leq 
    Cn^{2^{p-1}-1} k^{2^{p-1}} \sum_{j=1}^n \bb{E}\left[1+\norm{\bff{u}_h^{j-1}}{\bb{L}^2}^{2^p} \right]
    \leq C.
\end{align*}
For the term $R_3$, by the Burkholder--Davis--Gundy and the Jensen inequalities, noting the assumption on $G$, we obtain
\begin{align*}
    \bb{E}[R_3]
    &\leq
    C_p\, \bb{E}\left[\left(\sum_{j=1}^n k \norm{G(\bff{u}_h^{j-1})}{\bb{L}^2}^2 \norm{\bff{u}_h^{j-1}}{\bb{L}^2}^2\right)^{2^{p-2}} \right]
    \\
    &\leq
    C_p \, \bb{E}\left[\max_{j\leq n} \norm{\bff{u}_h^{j-1}}{\bb{L}^2}^{2^{p-1}} \left(k\sum_{j=1}^n \Big(1+\norm{\bff{u}_h^{j-1}}{\bb{L}^2}^2\Big) \right)^{2^{p-2}} \right]
    \\
    &\leq
    \frac14 \bb{E}\left[\max_{j\leq n} \norm{\bff{u}_h^{j-1}}{\bb{L}^2}^{2^{p}} \right]
    +
    Ck^{2^{p-1}} \bb{E}\left[ \left(\sum_{j=1}^n \Big(1+\norm{\bff{u}_h^{j-1}}{\bb{L}^2}^2\Big) \right)^{2^{p-1}} \right]
    \\
    &\leq
    \frac14 \bb{E}\left[\max_{j\leq n} \norm{\bff{u}_h^{j-1}}{\bb{L}^2}^{2^{p}} \right]
    +
    CT
    +
    C_T  \bb{E}\left[\max_{j\leq n} \norm{\bff{u}_h^{j-1}}{\bb{L}^2}^{2^{p}} \right]
    \leq C.
\end{align*}
This completes the proof of inequality~\eqref{equ:E u 2p general}, {upon noting the definition of $\bff{H}_h^n$ in the second equation of~\eqref{equ:euler}.}
\end{proof}

\begin{lemma}\label{lem:stab H1}
Suppose that $\{(\bff{u}_h^n, \bff{H}_h^n)\}$ satisfies~\eqref{equ:euler}. There exists a positive constant $C$ such that
\begin{align*}
	\bb{E} \left[ \max_{l\leq n} \norm{\bff{u}_h^l}{\bb{H}^1}^{2} \right] 
    +
    \bb{E}\left[\sum_{j=1}^n \norm{\bff{u}_h^j-\bff{u}_h^{j-1}}{\bb{H}^1}^2 \right]
    +
    \bb{E} \left[ \sum_{j=1}^n k \norm{\bff{H}_h^j}{\bb{H}^1}^2  \right]
	\leq 
	C,
\end{align*}
where $C$ depends on $T$, $R$, and $\norm{\bff{u}_0}{\bb{H}^1}$, but is independent of $n$, $h$, and $k$.
\end{lemma}

\begin{proof}
We set $\bff{\chi}_h= \bff{H}_h^n$ and $\bff{\phi}_h= \bff{u}_h^n-\bff{u}_h^{n-1}$ in \eqref{equ:euler} to obtain
\begin{align}
    \label{equ:un un1 Hn}
	\inpro{\bff{u}_h^n-\bff{u}_h^{n-1}}{\bff{H}_h^n}
	&=
	k \norm{\bff{H}_h^n}{\bb{L}^2}^2
	+
	k \norm{\nabla \bff{H}_h^n}{\bb{L}^2}^2
	+
	k\inpro{\mathcal{C}(\bff{u}_h^n)}{\bff{H}_h^n}
    +
    k\inpro{\mathcal{M}_R(\bff{u}_h^{n-1})}{\bff{H}_h^n}
    \nonumber\\
    &\quad
	+
	\inpro{G(\bff{u}_h^{n-1})}{\bff{H}_h^n} \overline{\Delta}W^n,
	\\
    \label{equ:Hn un un1}
	\inpro{\bff{H}_h^n}{\bff{u}_h^n-\bff{u}_h^{n-1}}
	&=
	-
	\frac{1}{2} \left(\norm{\nabla \bff{u}_h^{n}}{\bb{L}^2}^2 - \norm{\nabla \bff{u}_h^{n-1}}{\bb{L}^2}^2 \right)
	-
	\frac{1}{2}\norm{\nabla \bff{u}_h^{n}- \nabla \bff{u}_h^{n-1}}{\bb{L}^2}^2
    \nonumber\\
    &\quad
    +	
    \inpro{f_R(\bff{u}_h^{n-1})}{\bff{u}_h^n- \bff{u}_h^{n-1}}.
\end{align}
Next, putting $\bff{\phi}_h= \Pi_h G(\bff{u}_h^{n-1})$ yields
\begin{align}\label{equ:Gun Hn Wn}
    \inpro{G(\bff{u}_h^{n-1})}{\bff{H}_h^n} \overline{\Delta}W^n
    &=
    \inpro{\nabla\Pi_h G(\bff{u}_h^{n-1})}{\nabla\bff{u}_h^{n-1}} \overline{\Delta}W^n
    -
    \inpro{\nabla \Pi_h G(\bff{u}_h^{n-1})}{\nabla \bff{u}_h^n-\nabla \bff{u}_h^{n-1}} \overline{\Delta}W^n
    \nonumber\\
    &\quad
    +
    \inpro{\Pi_h G(\bff{u}_h^{n-1})}{f_R(\bff{u}_h^{n-1})} \overline{\Delta}W^n.
\end{align}
Furthermore, taking $\bff{\chi}_h= \Pi_h f_R(\bff{u}_h^{n-1})$ gives
\begin{align}\label{equ:fun un1}
    \inpro{f_R(\bff{u}_h^{n-1})}{\bff{u}_h^n- \bff{u}_h^{n-1}}
    &=
    k\inpro{\bff{H}_h^n}{f_R(\bff{u}_h^{n-1})}
    +
    k\inpro{\nabla \bff{H}_h^n}{\nabla\Pi_h f_R(\bff{u}_h^{n-1})}
    \nonumber\\
    &\quad
    -
    k\inpro{\bff{u}_h^n\times \bff{H}_h^n}{\Pi_h f_R(\bff{u}_h^{n-1})}
    +
    k\inpro{\mathcal{C}(\bff{u}_h^n)}{\Pi_h f_R(\bff{u}_h^{n-1})}
    \nonumber\\
    &\quad
    +
    k\inpro{\mathcal{M}_R(\bff{u}_h^{n-1})}{\Pi_h f_R(\bff{u}_h^{n-1})}
    +
    \inpro{G(\bff{u}_h^{n-1})}{\Pi_h f_R(\bff{u}_h^{n-1})} \overline{\Delta} W^n.
\end{align}
Subtracting \eqref{equ:Hn un un1} from \eqref{equ:un un1 Hn}, then adding the resulting expression with \eqref{equ:Gun Hn Wn} and \eqref{equ:fun un1}, we obtain
\begin{align}\label{equ:nab un nab un1}
	&\frac{1}{2} \left(\norm{\nabla \bff{u}_h^{n}}{\bb{L}^2}^2 - \norm{\nabla \bff{u}_h^{n-1}}{\bb{L}^2}^2 \right)
	+
	\frac{1}{2}\norm{\nabla \bff{u}_h^{n}- \nabla \bff{u}_h^{n-1}}{\bb{L}^2}^2
	+
	k\norm{\bff{H}_h^n}{\bb{L}^2}^2
	+
	k\norm{\nabla \bff{H}_h^n}{\bb{L}^2}^2
	\nonumber\\
	&=
	- k\inpro{\mathcal{C}(\bff{u}_h^n)}{\bff{H}_h^n}
    - k\inpro{\mathcal{M}_R(\bff{u}_h^{n-1})}{\bff{H}_h^n}
    +
    k\inpro{\bff{H}_h^n}{f_R(\bff{u}_h^{n-1})}
    +
    k\inpro{\nabla \bff{H}_h^n}{\nabla\Pi_h f_R(\bff{u}_h^{n-1})}
    \nonumber\\
    &\quad
    -
    k\inpro{\bff{u}_h^n\times \bff{H}_h^n}{\Pi_h f_R(\bff{u}_h^{n-1})}
    +
    k\inpro{\mathcal{C}(\bff{u}_h^n)}{\Pi_h f_R(\bff{u}_h^{n-1})}
    +
    k\inpro{\mathcal{M}_R(\bff{u}_h^{n-1})}{\Pi_h f_R(\bff{u}_h^{n-1})}
    \nonumber\\
    &\quad
    -
	\inpro{\nabla\Pi_h G(\bff{u}_h^{n-1})}{\nabla\bff{u}_h^{n-1}} \overline{\Delta}W^n
    -
    \inpro{\nabla \Pi_h G(\bff{u}_h^{n-1})}{\nabla \bff{u}_h^n-\nabla \bff{u}_h^{n-1}} \overline{\Delta}W^n
    \nonumber\\
    &=: I_1+I_2+\cdots+I_9.
\end{align}
We need to bound each term on the last line. 
For the first two terms, by \eqref{equ:Rv w}, the Gagliardo--Nirenberg and Young inequalities, it is clear that
\begin{align}\label{equ:est 1}
    \abs{I_1}+\abs{I_2}
    &\leq
    Ck\norm{\bff{u}_h^{n-1}}{\bb{L}^2}^2
    +
    Ck \left(1+ \norm{\bff{u}_h^n}{\bb{L}^2}^2 \right) \norm{\nabla \bff{u}_h^n}{\bb{L}^4}^2
    +
    \frac{k}{8} \norm{\bff{H}_h^n}{\bb{L}^2}^2
    +
    \frac{k}{8} \norm{\nabla \bff{H}_h^n}{\bb{L}^2}^2.
\end{align}
For the terms $I_3$ and $I_4$, by the assumptions on $f_R$, Young's inequality and~\eqref{equ:fR v L2} we have
\begin{align}
    \label{equ:est 2}
    \abs{I_3}
    +
    \abs{I_4}
    &\leq 
    C_R k \norm{\bff{u}_h^{n-1}}{\bb{L}^2}^2 + C_R k \norm{\nabla \bff{u}_h^{n-1}}{\bb{L}^2}^2 + \frac{k}{4} \norm{\bff{H}_h^n}{\bb{L}^2}^2 + \frac{k}{4} \norm{\nabla \bff{H}_h^n}{\bb{L}^2}^2.
\end{align}
For the term $I_5$, by Young's inequality, the Sobolev embedding $\bb{H}^1\hookrightarrow \bb{L}^4$, and the stability of $\Pi_h$, we infer
\begin{align}\label{equ:est 4}
    \abs{I_5}
    &\leq
    C_R k \norm{\bff{u}_h^n}{\bb{L}^4} \norm{\bff{H}_h^n}{\bb{L}^4} \norm{\Pi_h f_R(\bff{u}_h^{n-1})}{\bb{L}^2}
    \nonumber\\
    &\leq
    Ck \norm{\bff{u}_h^n}{\bb{L}^4}^2 \norm{\bff{u}_h^{n-1}}{\bb{L}^2}^2
    +
    \frac{k}{4} \norm{\bff{H}_h^n}{\bb{L}^2}^2
    +
    \frac{k}{4} \norm{\nabla \bff{H}_h^n}{\bb{L}^2}^2.
\end{align}
For the terms $I_6$ and $I_7$, we again apply~\eqref{equ:Rv w} and the Lipschitz continuity of $\mathcal{M}_R$ to obtain
\begin{align}\label{equ:est 5}
    \abs{I_6} + \abs{I_7}
    &\leq
    Ck\norm{\bff{u}_h^{n-1}}{\bb{L}^2}^2
    +
    Ck \left(1+\norm{\bff{u}_h^{n-1}}{\bb{L}^\infty}^2\right) \norm{\bff{u}_h^n}{\bb{L}^2}^2
    +
    Ck \norm{\nabla \bff{u}_h^n}{\bb{L}^2}^2
    \nonumber \\
    &\leq
    Ck\norm{\bff{u}_h^{n-1}}{\bb{L}^2}^2
    +
    Ck \left(1+\norm{\bff{u}_h^{n-1}}{\bb{L}^2}^2 + \norm{\Delta_h \bff{u}_h^{n-1}}{\bb{L}^2}^2 \right) \norm{\bff{u}_h^n}{\bb{L}^2}^2
    +
    Ck \norm{\nabla \bff{u}_h^n}{\bb{L}^2}^2,
\end{align}
where in the last step we also used~\eqref{equ:disc lapl L infty} and Young's inequality. 
For the term $I_9$, we have
\begin{align}\label{equ:est 7} 
    \abs{I_9}
    &\leq
    \frac{1}{4} \norm{\nabla \bff{u}_h^n-\nabla \bff{u}_h^{n-1}}{\bb{L}^2}^2
    +
    4 \norm{\nabla \Pi_h G(\bff{u}_h^{n-1})}{\bb{L}^2}^2 \abs{\overline{\Delta} W^n}^2
    \nonumber\\
    &\leq
    \frac{1}{4} \norm{\nabla \bff{u}_h^n-\nabla \bff{u}_h^{n-1}}{\bb{L}^2}^2
    +
    C \norm{\bff{u}_h^{n-1}}{\bb{H}^1}^2 \abs{\overline{\Delta} W^n}^2,
\end{align}
where in the last step we used the $\bb{H}^1$-stability of the $\bb{L}^2$-projection~\cite{BanYse14} and the definition of $G$.
The term $I_8$ in~\eqref{equ:nab un nab un1} remains as is for now.

We now collect all the estimates~\eqref{equ:est 1}, \eqref{equ:est 2}, 
 \eqref{equ:est 4}, \eqref{equ:est 5}, \eqref{equ:est 7}, and continue from~\eqref{equ:nab un nab un1}, taking care to absorb appropriate terms to the left-hand side. Summing over $j\in \{1,2,\ldots,l\}$, taking the maximum over $l$, and applying the expected value, we obtain
\begin{align}\label{equ:E nab uhl}
    &\bb{E} \left[ \max_{l\leq n} \norm{\nabla \bff{u}_h^l}{\bb{L}^2}^2 \right] 
    +
    \bb{E}\left[ \sum_{j=1}^n \norm{\nabla \bff{u}_h^j- \nabla \bff{u}_h^{j-1}}{\bb{L}^2}^2  \right]
    +
    \bb{E} \left[ \sum_{j=1}^n k \norm{\bff{H}_h^j}{\bb{H}^1}^2 \right]
    \nonumber\\
    &\quad 
    \leq
    \norm{\bff{u}_h^0}{\bb{H}^1}^{2}
    +
    C\, \bb{E} \left[ \max_{l\leq n} \norm{\bff{u}_h^l}{\bb{L}^2}^{2} \right]
    +
    C\, \bb{E}\left[ \sum_{j=1}^n k \left(1+\norm{\bff{u}_h^j}{\bb{L}^2}^2 \right) \norm{\nabla \bff{u}_h^j}{\bb{L}^4}^2 \right]  
    \nonumber\\
    &\quad
    +
    C\, \bb{E}\left[ \sum_{j=1}^n k \norm{\bff{u}_h^j}{\bb{L}^4}^2 \norm{\bff{u}_h^{j-1}}{\bb{L}^2}^2  \right]
    +
    C\, \bb{E} \left[ \sum_{j=1}^n k \left(1+\norm{\bff{u}_h^{j-1}}{\bb{L}^2}^2 + \norm{\Delta_h \bff{u}_h^{j-1}}{\bb{L}^2}^2 \right) \norm{\bff{u}_h^j}{\bb{L}^2}^2  \right]
    \nonumber\\
    &\quad
    +
    \bb{E} \left[ \max_{l\leq n} \sum_{j=1}^l \inpro{G(\bff{u}_h^{j-1})}{\Delta_h \bff{u}_h^{j-1}} \overline{\Delta} W^j \right]
    +
    C\, \bb{E} \left[  \sum_{j=1}^n \norm{\bff{u}_h^{j-1}}{\bb{H}^1}^2 \abs{\overline{\Delta} W^j}^2 \right]
    \nonumber\\
    &=:J_1+J_2+\cdots+J_7.
\end{align}

It remains to estimate the last five terms on the right-hand side. Firstly, the term $J_3$ is bounded by \eqref{equ:disc lapl L6} and Lemma~\ref{lem:stab L2}.
Next, we have
\begin{align*}
    J_4
    &\leq
    C\, \bb{E} \left[ \left(\max_{l\leq n} \norm{\bff{u}_h^{l-1}}{\bb{L}^2}^2\right) \left(\sum_{j=1}^n k \norm{\bff{u}_h^j}{\bb{L}^4}^2\right) \right]
    \leq
    C
    +
    C\, \bb{E}\left[\left( \sum_{j=1}^n k \norm{\bff{u}_h^j}{\bb{H}^1}^2 \right)^{2}\right] \leq C.
\end{align*}
The term $J_5$ can be similarly bounded. 
For the term $J_6$, by the Burkholder--Davis--Gundy inequality and the $\bb{H}^1$-stability of $\Pi_h$, we obtain
\begin{align*}
    \bb{E} \left[ \max_{l\leq n} \sum_{j=1}^l \inpro{G(\bff{u}_h^{j-1})}{\Delta_h \bff{u}_h^{j-1}} \overline{\Delta}W^j \right]
    &\leq
    C\bb{E} \left[\left(k \sum_{j=1}^n \left(1+\norm{\bff{u}_h^{j-1}}{\bb{L}^2}^2\right) \norm{\Delta_h \bff{u}_h^{j-1}}{\bb{L}^2}^2 \right)^{\frac12} \right]
    \\
    &\leq
    C\bb{E} \left[ \max_{j\leq n} \left(1+\norm{\bff{u}_h^{j-1}}{\bb{L}^2}^2\right)^\frac12 \left(k \sum_{j=1}^n \norm{\Delta_h \bff{u}_h^{j-1}}{\bb{L}^2}^2 \right)^{\frac12} \right] 
    \\
    &\leq
    C+
    C \bb{E}\left[ \sum_{j=1}^n k\norm{\Delta_h \bff{u}_h^{j-1}}{\bb{L}^2}^2 \right] \leq C,
\end{align*}
where in the last step we used Lemma~\ref{lem:stab L2}.
Similarly for the last term, we infer from the independence of the Wiener increment that
\begin{align*}
    J_7
    &\leq
    C \bb{E} \left[ \sum_{j=1}^n k \norm{\bff{u}_h^{j-1}}{\bb{H}^1}^{2} \right] \leq C.
\end{align*}
Substituting these estimates back into~\eqref{equ:E nab uhl}, we deduce the required inequality.
\end{proof}

In the following, we assume that $\beta_2=0$ in \eqref{equ:torque}. This estimate will be used only in Theorem~\ref{the:uhn strong}.

\begin{lemma}\label{lem:stab H1 high}
Suppose that $\beta_2=0$ in \eqref{equ:torque} and let $p\in [1,\infty)$ be a natural number. There exists a positive constant $C$ such that
\begin{align}\label{equ:H1 high}
	\bb{E} \left[ \max_{l\leq n} \norm{\bff{u}_h^l}{\bb{H}^1}^{2^p} \right] 
    +
    \bb{E} \left[ \sum_{j=1}^n k \norm{\bff{H}_h^j}{\bb{H}^1}^2 \norm{\nabla \bff{u}_h^j}{\bb{L}^2}^{2^{p-1}}  \right]
    +
    \bb{E}\left[ \left(k \sum_{j=1}^n \norm{\bff{H}_h^j}{\bb{H}^1}^2 \right)^{2^{p-1}} \right] 
	\leq 
	C,
\end{align}
where $C$ depends on $T, R$, and $\norm{\bff{u}_0}{\bb{H}^1}$, but is independent of $n$, $h$, and $k$.
\end{lemma}

\begin{proof}
As before, we prove the case $p=2$ in detail. Similarly to the proof of Lemma~\ref{lem:stab L2}, we multiply \eqref{equ:nab un nab un1} by $\norm{\nabla \bff{u}_n}{\bb{L}^2}^2$ to obtain
\begin{align}\label{equ:nab un I1 I7}
	&\frac{1}{4} \left(\norm{\nabla \bff{u}_h^{n}}{\bb{L}^2}^4 - \norm{\nabla \bff{u}_h^{n-1}}{\bb{L}^2}^4 + \left(\norm{\nabla \bff{u}_h^n}{\bb{L}^2}^2 - \norm{\nabla\bff{u}_h^{n-1}}{\bb{L}^2}^2\right)^2 \right)
	+
	\frac{1}{2}\norm{\nabla \bff{u}_h^{n}- \nabla \bff{u}_h^{n-1}}{\bb{L}^2}^2  \norm{\nabla\bff{u}_h^n}{\bb{L}^2}^2
	\nonumber \\
	&\quad
	+
	k\norm{\bff{H}_h^n}{\bb{L}^2}^2 \norm{\nabla\bff{u}_h^n}{\bb{L}^2}^2
	+
	k\norm{\nabla \bff{H}_h^n}{\bb{L}^2}^2 \norm{\nabla\bff{u}_h^n}{\bb{L}^2}^2
	\nonumber\\
	&=
	- k\inpro{\mathcal{C}(\bff{u}_h^n)}{\bff{H}_h^n} \norm{\nabla\bff{u}_h^n}{\bb{L}^2}^2
    - k\inpro{\mathcal{M}_R(\bff{u}_h^{n-1})}{\bff{H}_h^n} \norm{\nabla\bff{u}_h^n}{\bb{L}^2}^2
	+
	k\inpro{\bff{H}_h^n}{f_R(\bff{u}_h^{n-1})} \norm{\nabla\bff{u}_h^n}{\bb{L}^2}^2
    \nonumber\\
	&\quad
	+
	k\inpro{\nabla \bff{H}_h^n}{\nabla\Pi_h f_R(\bff{u}_h^{n-1})} \norm{\nabla\bff{u}_h^n}{\bb{L}^2}^2
	-
	k\inpro{\bff{u}_h^n\times \bff{H}_h^n}{\Pi_h f_R(\bff{u}_h^{n-1})} \norm{\nabla\bff{u}_h^n}{\bb{L}^2}^2
    \nonumber\\
	&\quad
    +
	k\inpro{\mathcal{C}(\bff{u}_h^n)}{\Pi_h f_R(\bff{u}_h^{n-1})} \norm{\nabla\bff{u}_h^n}{\bb{L}^2}^2
    +
	k\inpro{\mathcal{M}_R(\bff{u}_h^{n-1})}{\Pi_h f_R(\bff{u}_h^{n-1})} \norm{\nabla\bff{u}_h^n}{\bb{L}^2}^2
	\nonumber\\
	&\quad
	-
	\inpro{\nabla\Pi_h G(\bff{u}_h^{n-1})}{\nabla\bff{u}_h^{n-1}} \overline{\Delta}W^n  \norm{\nabla\bff{u}_h^n}{\bb{L}^2}^2
	-
	\inpro{\nabla \Pi_h G(\bff{u}_h^{n-1})}{\nabla \bff{u}_h^n-\nabla \bff{u}_h^{n-1}} \overline{\Delta}W^n  \norm{\nabla\bff{u}_h^n}{\bb{L}^2}^2
	\nonumber\\
	&=: I_1+I_2+\cdots+I_9.
\end{align}
Noting $\beta_2=0$, we can estimate the first five terms following the corresponding bounds in~\eqref{equ:est 1} to \eqref{equ:est 5}:
\begin{align*}
	\abs{I_1}
	&\leq
	Ck\norm{\bff{u}_h^n}{\bb{L}^2}^2 \norm{\nabla \bff{u}_h^n}{\bb{L}^2}^2
	+
	Ck \norm{\nabla \bff{u}_h^n}{\bb{L}^2}^4
	+
	\frac{k}{16} \norm{\bff{H}_h^n}{\bb{L}^2}^2  \norm{\nabla \bff{u}_h^n}{\bb{L}^2}^2,
	\\
	\abs{I_2}
	&\leq
	Ck \norm{\bff{u}_h^{n-1}}{\bb{L}^2}^4
    +
    Ck \norm{\bff{u}_h^n}{\bb{L}^2}^4
	+
	Ck \norm{\nabla \bff{u}_h^n}{\bb{L}^2}^4
	+
	\frac{k}{16} \norm{\bff{H}_h^n}{\bb{L}^2}^2  \norm{\nabla \bff{u}_h^n}{\bb{L}^2}^2,
    \\
	\abs{I_3}
	&\leq 
	C k \norm{\bff{u}_h^{n-1}}{\bb{L}^2}^4 +
	\frac{k}{16} \norm{\nabla \bff{u}_h^n}{\bb{L}^2}^4
	+ \frac{k}{16} \norm{\bff{H}_h^n}{\bb{L}^2}^2 \norm{\nabla \bff{u}_h^n}{\bb{L}^2}^2,
	\\
	\abs{I_4} 
	&\leq
	C k \norm{\nabla \bff{u}_h^{n-1}}{\bb{L}^2}^4
	+ \frac{k}{16} \norm{\nabla \bff{u}_h^n}{\bb{L}^2}^4 
	+ \frac{k}{16} \norm{\nabla \bff{H}_h^n}{\bb{L}^2}^2 \norm{\nabla \bff{u}_h^n}{\bb{L}^2}^2,
	\\
	\abs{I_5}
	&\leq
	Ck \norm{\bff{u}_h^n}{\bb{L}^2}^2 \norm{\bff{u}_h^{n-1}}{\bb{L}^2}^2 \left(\norm{\nabla \bff{u}_h^n}{\bb{L}^2}^2 + \norm{\Delta_h \bff{u}_h^n}{\bb{L}^2}^2\right)
	+
	\frac{k}{16} \norm{\bff{H}_h^n}{\bb{H}^1}^2 \norm{\nabla \bff{u}_h^n}{\bb{L}^2}^2,
    \\
	\abs{I_6}+\abs{I_7}
	&\leq
        Ck \norm{\bff{u}_h^n}{\bb{L}^2}^2 \norm{\bff{u}_h^{n-1}}{\bb{L}^2}^2 \left(\norm{\nabla \bff{u}_h^n}{\bb{L}^2}^2 + \norm{\Delta_h \bff{u}_h^n}{\bb{L}^2}^2\right)
        +
        Ck \norm{\bff{u}_h^{n-1}}{\bb{L}^2}^4
        +
        Ck \norm{\nabla \bff{u}_h^n}{\bb{L}^2}^4,
\end{align*}
where for the terms $I_5$ and $I_6$ we also used \eqref{equ:interp disc Lap nab L2}. The expected values of the terms $I_8$ and $I_9$ can be estimated as \eqref{equ:I4 sto}, \eqref{equ:I5 sto}, and \eqref{equ:E sto}. Substituting these estimates into \eqref{equ:nab un I1 I7}, summing, taking the maximum and the expected value, and applying the results of Lemma~\ref{lem:stab L2} as done previously shows~\eqref{equ:H1 high} for the first two terms. To establish the inequality for the last term, we sum \eqref{equ:nab un I1 I7} over $j\in\{1,2,\ldots,n\}$, square the result, and follow the same argument as in \eqref{equ:E uh 2 4}. The general case follows by induction as in the proof of Lemma~\ref{lem:stab L2}. We omit further details for brevity.
\end{proof}

To facilitate the proof of the error estimate, we decompose the error of the numerical method at time $t_n$, $n=0,1,\ldots,N$, as:
\begin{align}
    \label{equ:split u}
    \bff{u}(t_n)- \bff{u}_h^n
    &= 
    \left(\bff{u}(t_n)-\R_h \bff{u}(t_n)\right)
    +
    \left(\R_h \bff{u}(t_n)- \bff{u}_h^n\right)
    =:
    \bff{\rho}^n+ \bff{\theta}^n,
    \\
    \label{equ:split H}
    \bff{H}(t_n)-\bff{H}_h^n
    &=
    \left(\bff{H}(t_n) - \R_h \bff{H}(t_n)\right)
    +
    \left(\R_h \bff{H}(t_n)- \bff{H}_h^n\right)
    =:
    \bff{\eta}^n+ \bff{\xi}^n.
\end{align}
As such by the definition of the Ritz projection~\eqref{equ:Ritz},
\begin{equation}\label{equ:Ritz zero}
\inpro{\nabla \bff{\rho}^n}{\nabla \bff{\chi}_h}= \inpro{\nabla \bff{\eta}^n}{\nabla\bff{\chi}_h}=0, \quad \forall \bff{\chi}_h\in \bb{V}_h.
\end{equation}
Furthermore, define a sequence of subsets of $\Omega$ which depend on $\kappa$ and $m$:
\begin{align}\label{equ:Omega k m}
    \Omega_{\kappa,m}:= \left\{\omega\in \Omega: \max_{t\leq t_{m} \wedge T} \norm{\bff{u}(t)}{\bb{H}^2}^2 
    + \max_{t\leq t_{m} \wedge T} \norm{\bff{H}(t)}{\bb{L}^2}^2 
    + \max_{n\leq {m}} \norm{\bff{u}_h^n}{\bb{H}^1}^2 \leq \kappa \right\},
\end{align}
where $\kappa>0$ is to be specified. It is clear that for any $\kappa>0$ and $m\in \bb{N}$, we have $\Omega_{\kappa,m} \supset \Omega_{\kappa,m+1}$. Thus, for any time-discrete random variable $\bff{v}^n$,
\begin{align}\label{equ:1 vn vn1}
    &\bb{E}\left[\max_{m\leq n} \sum_{\ell=1}^m \one_{\Omega_{\kappa,\ell-1}} \inpro{\bff{v}^\ell-\bff{v}^{\ell-1}}{\bff{v}^\ell}\right]
    \nonumber\\
    &=
    \frac12 \bb{E}\left[\max_{m\leq n} \left( \one_{\Omega_{\kappa,m-1}} \norm{\bff{v}^m}{\bb{L}^2}^2 - \one_{\Omega_{\kappa,0}} \norm{\bff{v}^0}{\bb{L}^2}^2 
    +
    \sum_{\ell=2}^m \left( \one_{\Omega_{\kappa,m-2}}- \one_{\Omega_{\kappa,m-1}}\right) \norm{\bff{v}^{m-1}}{\bb{L}^2}^2 \right)
    \right]
    \nonumber\\
    &\quad
    +
    \frac12  \sum_{\ell=1}^n \bb{E} \left[\one_{\Omega_{\kappa,\ell-1}} \norm{\bff{v}^\ell-\bff{v}^{\ell-1}}{\bb{L}^2}^2 \right]
    \nonumber\\
    &\geq
    \frac12 \bb{E}\left[\max_{m\leq n} \left( \one_{\Omega_{\kappa,m-1}} \norm{\bff{v}^m}{\bb{L}^2}^2 \right)- \one_{\Omega_{\kappa,0}} \norm{\bff{v}^0}{\bb{L}^2}^2  \right] 
	+
    \frac12  \sum_{\ell=1}^n \bb{E} \left[\one_{\Omega_{\kappa,\ell-1}} \norm{\bff{v}^\ell-\bff{v}^{\ell-1}}{\bb{L}^2}^2 \right].
\end{align}

The following technical lemmas will be needed in the analysis.

\begin{lemma}\label{lem:Holder D A beta}
Let $(\bff{u},\bff{H})$ be the solution of~\eqref{equ:weakform} with initial data $\bff{u}_0\in \mathrm{D}(A^{\frac12})$. Suppose that $\{t_\ell\}_{\ell=0}^N$ is a uniform partition of $[0,T]$ with $k=T/N$. Then for any $\beta\in [\frac12,1)$ and $\alpha\in (0,1-\beta)$, there exists a constant $C$ such that for any $n\in \{1,2,\ldots,N\}$,
\begin{align*}
    \bb{E}\left[\sum_{\ell=1}^n \one_{\Omega_{\kappa,\ell-1}} \int_{t_{\ell-1}}^{t_\ell} \norm{\bff{u}(t_{\ell})-\bff{u}(s)}{\mathrm{D}(A^\beta)}^2 \ds \right] 
    &\leq
    C k^{2\alpha}.
\end{align*}
The constant $C$ depends on $\alpha$, $\beta$, and $T$, but is independent of $n$ and $k$.
\end{lemma}

\begin{proof}
From the proof of~\cite[Lemma~4.10]{GolSoeTra24b}, for $\bff{u}_0\in \mathrm{D}(A^{\frac12})$, we have the following estimate
\begin{align}\label{equ:Holder est}
    \bb{E}\left[\norm{\bff{u}(t)-\bff{u}(s)}{\mathrm{D}(A^\beta)}^p\right] \leq C(t-s)^{\alpha p} s^{-p(\alpha+\beta-\frac12)}, \quad \forall t\geq s>0.
\end{align}
Therefore,
\begin{align*}
    \bb{E}\left[\sum_{\ell=1}^n \one_{\Omega_{\kappa,\ell-1}} \int_{t_{\ell-1}}^{t_\ell} \norm{\bff{u}(t_\ell)-\bff{u}(s)}{\mathrm{D}(A^\beta)}^2 \ds \right] 
    &\leq
    C \sum_{\ell=1}^n \int_{t_{\ell-1}}^{t_\ell} (t_\ell-s)^{2\alpha} s^{-2(\alpha+\beta-\frac12)} \ds 
    \\
    &\leq
    Ck^{2\alpha} \int_0^T s^{-2(\alpha+\beta-\frac12)} \ds 
    \leq Ck^{2\alpha},
\end{align*}
as required.
\end{proof}

\begin{lemma}
Let $(\bff{u},\bff{H})$ be the solution of~\eqref{equ:weakform} with initial data $\bff{u}_0\in \mathrm{D}(A^{\frac12})$. Suppose that $\{t_\ell\}_{\ell=0}^N$ is a uniform partition of $[0,T]$ with $k=T/N$. Suppose that $X$ is a Banach space and $\{\bff{\phi}^\ell\}_{\ell=0}^N$ is a sequence of functions in $X$ such that
\[
    \bb{E}\left[\max_{\ell\leq N} \norm{\bff{\phi}^\ell}{X}^4\right] < \infty.
\]
Then for any $\alpha\in (0,\frac12)$, there exists a constant $C$ such that for any $n\in \{1,2,\ldots,N\}$,
\begin{align}\label{equ:E us utl phi}
    \bb{E}\left[\sum_{\ell=1}^n \one_{\Omega_{\kappa,\ell-1}} \int_{t_{\ell-1}}^{t_{\ell}} \norm{\bff{u}(s)-\bff{u}(t_{\ell-1})}{\mathrm{D}(A^\frac12)}^2 \norm{\bff{\phi}^\ell}{X}^2 \ds \right] 
    &\leq
    Ck^{2\alpha}.
\end{align}
Consequently, we have for any $\alpha\in (0,\frac12)$,
\begin{align}\label{equ:E us phi}
    \bb{E}\left[\sum_{\ell=1}^n \one_{\Omega_{\kappa,\ell-1}} \int_{t_{\ell-1}}^{t_{\ell}} \norm{\bff{u}(s)}{\mathrm{D}(A^\frac12)}^2 \norm{\bff{\phi}^\ell}{X}^2 \ds \right] 
    &\leq
    Ck^{2\alpha} + C\kappa k \bb{E} \left[\sum_{\ell=1}^n \one_{\Omega_{\kappa,\ell-1}} \norm{\bff{\phi}^\ell}{X}^2 \right].
\end{align}
The constant $C$ depends on $\alpha$ and $T$, but is independent of $n$ and $k$.
\end{lemma}

\begin{proof}
We have by H\"older's inequality,
\begin{align*}
    &\bb{E}\left[\sum_{\ell=1}^n \one_{\Omega_{\kappa,\ell-1}} \int_{t_{\ell-1}}^{t_{\ell}} \norm{\bff{u}(s)-\bff{u}(t_{\ell-1})}{\mathrm{D}(A^\frac12)}^2 \norm{\bff{\phi}^\ell}{X}^2 \ds \right]
    \\
    &\leq
    \left(\bb{E}\left[\max_{\ell\leq N} \norm{\bff{\phi}^\ell}{X}^4 \right]\right)^{\frac12}
    \left(\sum_{\ell=1}^n \int_{t_{\ell-1}}^{t_{\ell}} \left(\bb{E} \norm{\bff{u}(s)-\bff{u}(t_{\ell-1})}{\mathrm{D}(A^\frac12)}^4 \right)^{\frac12} \ds \right)
    \\
    &\leq
    C \int_0^{t_1} \left(\bb{E}\norm{\bff{u}(s)-\bff{u}_0}{\mathrm{D}(A^\frac12)}^4 \right)^{\frac12} \ds 
    +
    C \left(\sum_{\ell=2}^n \int_{t_{\ell-1}}^{t_{\ell}} \left(\bb{E} \norm{\bff{u}(s)-\bff{u}(t_{\ell-1})}{\mathrm{D}(A^\frac12)}^4 \right)^{\frac12} \ds \right)
    \\
    &=: I_1+I_2.
\end{align*}
For $I_1$, we use the fact that $\bff{u}\in L^p\big(\Omega; \mathcal{C}([0,T]; \text{D}(A^\frac12))\big)$ to obtain $I_1 \leq Ck$.
For the term $I_2$, we use \eqref{equ:Holder est} to infer for any $\alpha\in (0,\frac12)$,
\begin{align*}
    I_2 &\leq
    C \sum_{\ell=2}^n \int_{t_{\ell-1}}^{t_\ell} (s-t_{\ell-1})^{2\alpha}  t_{\ell-1}^{-2\alpha} \ds 
    \leq
    Ck^{2\alpha} \sum_{\ell=2}^n k t_{\ell-1}^{-2\alpha} 
    \leq
    Ck^{2\alpha} \int_0^T s^{-2\alpha} \,\ds 
    \leq Ck^{2\alpha},
\end{align*}
thus proving \eqref{equ:E us utl phi}. To show \eqref{equ:E us phi}, we write $\bff{u}(s)= \bff{u}(s)-\bff{u}(t_{\ell-1})+ \bff{u}(t_{\ell-1})$, and employ \eqref{equ:Omega k m} to note that
\begin{align*}
    \bb{E}\left[\sum_{\ell=1}^n \one_{\Omega_{\kappa,\ell-1}} \int_{t_{\ell-1}}^{t_{\ell}} \norm{\bff{u}(t_{\ell-1})}{\mathrm{D}(A^\frac12)}^2 \norm{\bff{\phi}^\ell}{X}^2 \ds \right] \leq 
    \kappa k \bb{E} \left[ \sum_{\ell=1}^n \one_{\Omega_{\kappa,\ell-1}} \norm{\bff{\phi}^\ell}{X}^2 \right].
\end{align*}
Inequality \eqref{equ:E us phi} then follows by using \eqref{equ:E us utl phi} and the triangle inequality.
This completes the proof of the lemma.
\end{proof}

We are now ready to prove an auxiliary error estimate. Similar to the assumptions in the following proposition, a technical mesh constraint condition $h=O(k)$ is also implicitly assumed in~\cite{GolJiaLe24}.

\begin{proposition}\label{pro:E theta n L2}
Let $(\bff{u},\bff{H})$ be the solution of~\eqref{equ:weakform} with initial data $\bff{u}_0\in \mathrm{D}(A^{\frac12})$, and let $(\bff{u}_h^n,\bff{H}_h^n)$ be the solution to~\eqref{equ:euler}. Let $\Omega_{\kappa,m}$ be as defined in~\eqref{equ:Omega k m}. Let $\bff{\theta}^n$ and $\bff{\xi}^n$ be as defined in~\eqref{equ:split u} and \eqref{equ:split H}, respectively. Suppose that $h=O(k)$ and $n\in \{1,2,\ldots, N\}$. Then for any $\delta>0$,
\begin{align*}
    \bb{E}\left[\max_{m\leq n} \left( \one_{\Omega_{\kappa,m-1}} \norm{\bff{\theta}^m}{\bb{L}^2}^2 \right) \right]  
    +
    k  \sum_{\ell=1}^n \bb{E} \left[ \one_{\Omega_{\kappa,\ell-1}} \left(\norm{\nabla \bff{\theta}^\ell}{\bb{L}^2}^2 +\norm{\bff{\xi}^\ell}{\bb{L}^2}^2\right) \right]
    &\leq
    e^{C \kappa} \left(h^2+k^{\frac12-\delta}\right),
\end{align*}
where $C$ is a constant depending on $R$, $T$, \blue{and $\delta$}, but is independent of $h$ and $k$.
\end{proposition}

\begin{proof}
Subtracting~\eqref{equ:weakform} from~\eqref{equ:euler}, rewriting the indices, and noting~\eqref{equ:split u} and~\eqref{equ:split H}, we have for any $\bff{\chi}_h,\bff{\phi}_h\in \bb{V}_h$,
\begin{align}
\label{equ:theta n theta n1}
    \inpro{\bff{\theta}^\ell-\bff{\theta}^{\ell-1}}{\bff{\chi}_h}
    &=
    -
    \inpro{\bff{\rho}^\ell -\bff{\rho}^{\ell-1}}{\bff{\chi}_h}
    +
    \int_{t_{\ell-1}}^{t_\ell} \inpro{\bff{H}(s)-\bff{H}_h^\ell}{\bff{\chi}_h} \ds
    +
    \int_{t_{\ell-1}}^{t_\ell} \inpro{\nabla \bff{H}(s)-\nabla \bff{H}_h^\ell}{\nabla \bff{\chi}_h} \ds 
    \nonumber\\
    &\quad
    -
    \int_{t_{\ell-1}}^{t_\ell} \inpro{\big(\bff{u}(s)-\bff{u}_h^\ell \big) \times \bff{H}_h^\ell}{\bff{\chi}_h} \ds
    -
    \int_{t_{\ell-1}}^{t_\ell} \inpro{\bff{u}(s) \times \big(\bff{H}(s)- \bff{H}_h^\ell\big)}{\bff{\chi}_h} \ds 
    \nonumber\\
    &\quad
    +
    \int_{t_{\ell-1}}^{t_\ell} \inpro{\bff{\nu} \cdot \nabla\big(\bff{u}(s)- \bff{u}_h^\ell\big)}{\bff{\chi}_h} \ds 
    +
    \int_{t_{\ell-1}}^{t_\ell} \inpro{ \big(\bff{u}(s)- \bff{u}_h^\ell\big) \times (\bff{\nu}\cdot \nabla)\bff{u}_h^\ell}{\bff{\chi}_h} \ds
    \nonumber\\
    &\quad
    +
    \int_{t_{\ell-1}}^{t_\ell} \inpro{\bff{u}(s) \times (\bff{\nu}\cdot \nabla) \big(\bff{u}(s)-\bff{u}_h^\ell \big)}{\bff{\chi}_h} \ds
    +
    \int_{t_{\ell-1}}^{t_\ell} \inpro{\mathcal{M}_R(\bff{u}(s))- \mathcal{M}_R(\bff{u}_h^{\ell-1})}{\bff{\chi}_h} \ds 
    \nonumber\\
    &\quad
    +
    \int_{t_{\ell-1}}^{t_\ell} \inpro{G(\bff{u}(s))- G(\bff{u}_h^{\ell-1})}{\bff{\chi}_h} \mathrm{d}W(s),
    \\
\label{equ:Hhn Htn}
    \inpro{\bff{\xi}^\ell}{\bff{\phi}_h}
    &=
    -
    \inpro{\bff{\eta}^\ell}{\bff{\phi}_h}
    -
    \inpro{\nabla \bff{\theta}^\ell+ \nabla \bff{\rho}^\ell}{\nabla \bff{\phi}_h}
    +
    \inpro{f_R(\bff{u}(t_\ell))- f_R(\bff{u}_h^{\ell-1})}{\bff{\phi}_h}.
\end{align}
We now put $\bff{\chi}_h=\bff{\theta}^\ell$ in~\eqref{equ:theta n theta n1} and $\bff{\phi}_h=k \bff{\theta}^\ell$ in~\eqref{equ:Hhn Htn}, then multiply the resulting equations by $\one_{\Omega_{\kappa,\ell-1}}$, where the set $\Omega_{\kappa,n}$ was defined in~\eqref{equ:Omega k m}. We then sum the resulting expression over $\ell\in \{1,2,\ldots,m\}$, take the maximum over $m\leq n$, and apply the expectation value. Noting~\eqref{equ:split u}, \eqref{equ:split H}, \eqref{equ:Ritz zero}, and~\eqref{equ:1 vn vn1}, we obtain
\begin{align}\label{equ:12 theta n}
    &\frac12 \bb{E}\left[\max_{m\leq n} \left( \one_{\Omega_{\kappa,m-1}} \norm{\bff{\theta}^m}{\bb{L}^2}^2 \right) \right] 
    +
    \frac12  \sum_{\ell=1}^n \bb{E} \left[\one_{\Omega_{\kappa,\ell-1}} \norm{\bff{\theta}^\ell-\bff{\theta}^{\ell-1}}{\bb{L}^2}^2 \right]
    \nonumber\\
    &\leq
    \frac12 \bb{E}\left[\norm{\bff{\theta}^0}{\bb{L}^2}^2 \right]
    -
    \bb{E}\left[\max_{m\leq n} \sum_{\ell=1}^m \one_{\Omega_{\kappa,\ell-1}} \inpro{\bff{\rho}^\ell-\bff{\rho}^{\ell-1}}{\bff{\theta}^\ell}\right]
    \nonumber\\
    &\quad
    +
    \bb{E}\left[\max_{m\leq n} \sum_{\ell=1}^m \one_{\Omega_{\kappa,\ell-1}} \int_{t_{\ell-1}}^{t_\ell} \inpro{\bff{\eta}^\ell + \bff{\xi}^\ell+ \bff{H}(s)-\bff{H}(t_\ell)}{\bff{\theta}^\ell} \ds \right]
    \nonumber\\
    &\quad
    +
    \bb{E}\left[\max_{m\leq n} \sum_{\ell=1}^m \one_{\Omega_{\kappa,\ell-1}} \int_{t_{\ell-1}}^{t_\ell} \inpro{\nabla \bff{\xi}^\ell+ \nabla \bff{H}(s)-\nabla \bff{H}(t_\ell)}{\nabla \bff{\theta}^\ell} \ds\right]
    \nonumber\\
    &\quad
    -
    \bb{E}\left[\max_{m\leq n} \sum_{\ell=1}^m \one_{\Omega_{\kappa,\ell-1}} \int_{t_{\ell-1}}^{t_\ell} \inpro{(\bff{\rho}^\ell+ \bff{\theta}^\ell + \bff{u}(s)-\bff{u}(t_\ell)) \times \bff{H}_h^\ell}{\bff{\theta}^\ell} \ds \right]
    \nonumber\\
    &\quad
    -
    \bb{E}\left[\max_{m\leq n} \sum_{\ell=1}^m \one_{\Omega_{\kappa,\ell-1}} \int_{t_{\ell-1}}^{t_\ell} \inpro{\bff{u}(s) \times (\bff{\eta}^\ell + \bff{\xi}^\ell + \bff{H}(s)-\bff{H}(t_\ell))}{\bff{\theta}^\ell} \ds \right]
    \nonumber\\
    &\quad
    +
   	\bb{E}\left[\max_{m\leq n} \sum_{\ell=1}^m \one_{\Omega_{\kappa,\ell-1}} \int_{t_{\ell-1}}^{t_\ell} \inpro{\bff{\nu}\cdot \nabla (\bff{\rho}^\ell+ \bff{\theta}^\ell + \bff{u}(s)-\bff{u}(t_\ell))}{\bff{\theta}^\ell} \ds \right]
    \nonumber\\
    &\quad
    +
    \bb{E}\left[\max_{m\leq n} \sum_{\ell=1}^m \one_{\Omega_{\kappa,\ell-1}} \int_{t_{\ell-1}}^{t_\ell} \inpro{(\bff{\rho}^\ell + \bff{\theta}^\ell + \bff{u}(s)-\bff{u}(t_\ell)) \times (\bff{\nu}\cdot \nabla) \bff{u}_h^\ell}{\bff{\theta}^\ell} \ds \right]
    \nonumber\\
    &\quad
    +
    \bb{E}\left[\max_{m\leq n} \sum_{\ell=1}^m \one_{\Omega_{\kappa,\ell-1}} \int_{t_{\ell-1}}^{t_\ell} \inpro{\bff{u}(s) \times (\bff{\nu}\cdot \nabla)(\bff{\rho}^\ell+ \bff{\theta}^\ell+\bff{u}(s)-\bff{u}(t_\ell))}{\bff{\theta}^\ell} \ds \right]
    \nonumber\\
    &\quad
     +
    \bb{E}\left[\max_{m\leq n} \sum_{\ell=1}^m \one_{\Omega_{\kappa,\ell-1}} \int_{t_{\ell-1}}^{t_\ell} \inpro{\mathcal{M}_R(\bff{u}(s))- \mathcal{M}_R(\bff{u}_h^{\ell-1})}{\bff{\theta}^\ell} \ds \right]
    \nonumber\\
    &\quad
    +
    \bb{E}\left[\max_{m\leq n} \sum_{\ell=1}^m \one_{\Omega_{\kappa,\ell-1}} \int_{t_{\ell-1}}^{t_\ell} \inpro{G(\bff{u}(s))- G(\bff{u}_h^{\ell-1})}{\bff{\theta}^\ell} \mathrm{d}W(s) \right],
\end{align}
and
\begin{align}\label{equ:int xi n theta n}
    k \bb{E}\left[\max_{m\leq n} \sum_{\ell=1}^m \one_{\Omega_{\kappa,\ell-1}} \inpro{\bff{\xi}^\ell}{\bff{\theta}^\ell} \right]
    &=
    -k
    \bb{E}\left[\max_{m\leq n} \sum_{\ell=1}^m \one_{\Omega_{\kappa,\ell-1}} \norm{\nabla \bff{\theta}^\ell}{\bb{L}^2}^2 \right]
    -k
   	\bb{E}\left[\max_{m\leq n} \sum_{\ell=1}^m \one_{\Omega_{\kappa,\ell-1}}  \inpro{\bff{\eta}^\ell}{\bff{\theta}^\ell} \right]
    \nonumber\\
    &\quad
    +
    k\bb{E}\left[\max_{m\leq n} \sum_{\ell=1}^m \one_{\Omega_{\kappa,\ell-1}}  \inpro{f_R(\bff{u}(t_\ell))- f_R(\bff{u}_h^{\ell-1})}{\bff{\theta}^\ell} \right].
\end{align}
Next, we put $\bff{\phi}_h= k\bff{\xi}^n$ to obtain
\begin{align}\label{equ:k xi n L2}
    k \bb{E}\left[\max_{m\leq n} \sum_{\ell=1}^m \one_{\Omega_{\kappa,\ell-1}} \norm{\bff{\xi}^\ell}{\bb{L}^2}^2 \right]
    &=
    -
    k \bb{E}\left[\max_{m\leq n} \sum_{\ell=1}^m \one_{\Omega_{\kappa,\ell-1}} \inpro{\bff{\eta}^\ell}{\bff{\xi}^\ell}  \right]
    -
    k \bb{E}\left[\max_{m\leq n} \sum_{\ell=1}^m \one_{\Omega_{\kappa,\ell-1}} \inpro{\nabla \bff{\theta}^\ell}{\nabla \bff{\xi}^\ell} \right]
    \nonumber\\
    &\quad
    +
    k \bb{E}\left[\max_{m\leq n} \sum_{\ell=1}^m \one_{\Omega_{\kappa,\ell-1}} \inpro{f_R(\bff{u}(t_\ell))- f_R(\bff{u}_h^{\ell-1})}{\bff{\xi}^\ell} \right]. 
\end{align}
Adding~\eqref{equ:12 theta n}, \eqref{equ:int xi n theta n}, and~\eqref{equ:k xi n L2}, we have
\begin{align}\label{equ:theta n L2 pre}
    &\frac12 \bb{E}\left[\max_{m\leq n} \left( \one_{\Omega_{\kappa,m-1}} \norm{\bff{\theta}^m}{\bb{L}^2}^2 \right) \right] 
    +
    \frac12  \sum_{\ell=1}^n \bb{E} \left[\one_{\Omega_{\kappa,\ell-1}} \norm{\bff{\theta}^\ell-\bff{\theta}^{\ell-1}}{\bb{L}^2}^2 \right]
    \nonumber\\
    &\quad
    +
    k
    \bb{E}\left[\sum_{\ell=1}^n \one_{\Omega_{\kappa,\ell-1}} \norm{\nabla \bff{\theta}^\ell}{\bb{L}^2}^2 \right]
    +
   	k \bb{E}\left[\sum_{\ell=1}^n \one_{\Omega_{\kappa,\ell-1}} \norm{\bff{\xi}^\ell}{\bb{L}^2}^2 \right]
   	\nonumber\\
   	&\leq
   	\frac12 \bb{E}\left[\norm{\bff{\theta}^0}{\bb{L}^2}^2\right]
   	-
   	\bb{E}\left[\max_{m\leq n} \sum_{\ell=1}^m \one_{\Omega_{\kappa,\ell-1}} \inpro{\bff{\rho}^\ell-\bff{\rho}^{\ell-1}}{\bff{\theta}^\ell}\right]
   	\nonumber\\
   	&\quad
   	+
   	\bb{E}\left[\max_{m\leq n} \sum_{\ell=1}^m \one_{\Omega_{\kappa,\ell-1}} \int_{t_{\ell-1}}^{t_\ell} \inpro{\bff{\eta}^\ell + \bff{\xi}^\ell+ \bff{H}(s)-\bff{H}(t_\ell)}{\bff{\theta}^\ell} \ds \right]
   	\nonumber\\
   	&\quad
   	+
   	\bb{E}\left[\max_{m\leq n} \sum_{\ell=1}^m \one_{\Omega_{\kappa,\ell-1}} \int_{t_{\ell-1}}^{t_\ell} \inpro{ \nabla \bff{H}(s)-\nabla \bff{H}(t_\ell)}{\nabla \bff{\theta}^\ell} \ds\right]
   	\nonumber\\
   	&\quad
   	-
   	{\bb{E}\left[\max_{m\leq n} \sum_{\ell=1}^m \one_{\Omega_{\kappa,\ell-1}} \int_{t_{\ell-1}}^{t_\ell} \inpro{(\bff{\rho}^\ell+ \bff{\theta}^\ell + \bff{u}(s)-\bff{u}(t_\ell)) \times \bff{H}_h^\ell}{\bff{\theta}^\ell} \ds \right] }
   	\nonumber\\
   	&\quad
   	-
   	{\bb{E}\left[\max_{m\leq n} \sum_{\ell=1}^m \one_{\Omega_{\kappa,\ell-1}} \int_{t_{\ell-1}}^{t_\ell} \inpro{\bff{u}(s) \times (\bff{\eta}^\ell + \bff{\xi}^\ell + \bff{H}(s)-\bff{H}(t_\ell))}{\bff{\theta}^\ell} \ds \right] }
   	\nonumber\\
   	&\quad
   	+
   	\bb{E}\left[\max_{m\leq n} \sum_{\ell=1}^m \one_{\Omega_{\kappa,\ell-1}} \int_{t_{\ell-1}}^{t_\ell} \inpro{\bff{\nu}\cdot \nabla (\bff{\rho}^\ell+ \bff{\theta}^\ell + \bff{u}(s)-\bff{u}(t_\ell))}{\bff{\theta}^\ell} \ds \right]
   	\nonumber\\
   	&\quad
   	+
   	\bb{E}\left[\max_{m\leq n} \sum_{\ell=1}^m \one_{\Omega_{\kappa,\ell-1}} \int_{t_{\ell-1}}^{t_\ell} \inpro{(\bff{\rho}^\ell + \bff{\theta}^\ell + \bff{u}(s)-\bff{u}(t_\ell)) \times (\bff{\nu}\cdot \nabla) \bff{u}_h^\ell}{\bff{\theta}^\ell} \ds \right]
   	\nonumber\\
   	&\quad
   	+
   	\bb{E}\left[\max_{m\leq n} \sum_{\ell=1}^m \one_{\Omega_{\kappa,\ell-1}} \int_{t_{\ell-1}}^{t_\ell} \inpro{\bff{u}(s) \times (\bff{\nu}\cdot \nabla)(\bff{\rho}^\ell+ \bff{\theta}^\ell+\bff{u}(s)-\bff{u}(t_\ell))}{\bff{\theta}^\ell} \ds \right]
   	\nonumber\\
   	&\quad
   	+
   	\bb{E}\left[\max_{m\leq n} \sum_{\ell=1}^m \one_{\Omega_{\kappa,\ell-1}} \int_{t_{\ell-1}}^{t_\ell} \inpro{\mathcal{M}_R(\bff{u}(s))- \mathcal{M}_R(\bff{u}_h^{\ell-1})}{\bff{\theta}^\ell} \ds \right]
   	\nonumber\\
   	&\quad
   	+
   	k\bb{E}\left[\max_{m\leq n} \sum_{\ell=1}^m \one_{\Omega_{\kappa,\ell-1}}  \inpro{f_R(\bff{u}(t_\ell))- f_R(\bff{u}_h^{\ell-1})}{\bff{\theta}^\ell} \right]
   	\nonumber\\
   	&\quad
   	-k
   	\bb{E}\left[\max_{m\leq n} \sum_{\ell=1}^m \one_{\Omega_{\kappa,\ell-1}}  \inpro{\bff{\eta}^\ell}{\bff{\theta}^\ell} \right]
   	-k
   	\bb{E}\left[\max_{m\leq n} \sum_{\ell=1}^m \one_{\Omega_{\kappa,\ell-1}}  \inpro{\bff{\eta}^\ell}{\bff{\xi}^\ell} \right]
   	\nonumber\\
   	&\quad
   	+
   	\bb{E}\left[\max_{m\leq n} \sum_{\ell=1}^m \one_{\Omega_{\kappa,\ell-1}} \int_{t_{\ell-1}}^{t_\ell} \inpro{\mathcal{M}_R(\bff{u}(s))- \mathcal{M}_R(\bff{u}_h^{\ell-1})}{\bff{\xi}^\ell} \ds \right]
   	\nonumber\\
   	&\quad
   	+
   	\bb{E}\left[\max_{m\leq n} \sum_{\ell=1}^m \one_{\Omega_{\kappa,\ell-1}} \int_{t_{\ell-1}}^{t_\ell} \inpro{G(\bff{u}(s))- G(\bff{u}_h^{\ell-1})}{\bff{\theta}^\ell} \mathrm{d}W(s) \right]
   	\nonumber\\
    &=: \frac12 \bb{E}\left[\norm{\bff{\theta}^0}{\bb{L}^2}^2\right] +I_1+I_2 + \cdots + I_{14}.
\end{align}
We will estimate each term on the last line, noting the regularity of the solution in Proposition~\ref{pro:Holder u}. Let $\epsilon>0$ be a constant to be determined later. For the first term, by~\eqref{equ:Ritz approx} and Young's inequality,
\begin{align*}
    \abs{I_1}
    &\leq
    Ck^{-1} \bb{E}\left[\sum_{\ell=1}^n \norm{\bff{\rho}^\ell -\bff{\rho}^{\ell-1}}{\bb{L}^2}^2\right] + Ck \bb{E}\left[\sum_{\ell=1}^n \one_{\Omega_{\kappa,\ell-1}}  \norm{\bff{\theta}^\ell}{\bb{L}^2}^2 \right]
    \\
    &\leq
    Cnh^4 k^{-1}
    +
    Ck \bb{E} \left[\sum_{\ell=1}^n \one_{\Omega_{\kappa,\ell-1}}  \norm{\bff{\theta}^\ell}{\bb{L}^2}^2 \right].
\end{align*}
Let $\delta>0$ be arbitrary. For the second term, by Young's inequality, \eqref{equ:Ritz approx}, and the H\"older continuity in time of the solution given by Proposition~\ref{pro:Holder u},
\begin{align*}
    \abs{I_2}
    &\leq
    Ch^4 + \epsilon k \bb{E} \left[\sum_{\ell=1}^n \one_{\Omega_{\kappa,\ell-1}}  \norm{\bff{\xi}^\ell}{\bb{L}^2}^2 \right]
    +
    C k \bb{E} \left[\sum_{\ell=1}^n \one_{\Omega_{\kappa,\ell-1}} \norm{\bff{\theta}^\ell}{\bb{L}^2}^2 \right]
    \\
    &\quad
    +
    C \bb{E} \left[\sum_{\ell=1}^n \one_{\Omega_{\kappa,\ell-1}} \int_{t_{\ell-1}}^{t_\ell} \norm{\bff{H}(s)-\bff{H}(t_\ell)}{\bb{L}^2}^2 \ds \right]
    \\
    &\leq
    Ch^4+ Ck^{1-\delta} + \epsilon k \bb{E} \left[\sum_{\ell=1}^n \one_{\Omega_{\kappa,\ell-1}}  \norm{\bff{\xi}^\ell}{\bb{L}^2}^2 \right]
    +
    C k \bb{E} \left[\sum_{\ell=1}^n \one_{\Omega_{\kappa,\ell-1}} \norm{\bff{\theta}^\ell}{\bb{L}^2}^2 \right].
\end{align*}

For the term $I_3$, by Young's inequality and Lemma~\ref{lem:Holder D A beta} (essentially with $\beta=3/4$), we have
\begin{align}\label{equ:12 Holder a}
    \abs{I_3}
    &\leq
    \epsilon k \bb{E} \left[\sum_{\ell=1}^n \one_{\Omega_{\kappa,\ell-1}} \norm{\nabla \bff{\theta}^\ell}{\bb{L}^2}^2 \right]
    +
    C \bb{E} \left[\sum_{\ell=2}^n \one_{\Omega_{\kappa,\ell-1}} \int_{t_{\ell-1}}^{t_\ell} \norm{\nabla\bff{H}(s)-\nabla\bff{H}(t_\ell)}{\bb{L}^2}^2 \ds \right]
    \nonumber\\
    &\leq
    \epsilon k \bb{E} \left[\sum_{\ell=1}^n \one_{\Omega_{\kappa,\ell-1}} \norm{\nabla \bff{\theta}^\ell}{\bb{L}^2}^2 \right]
    + Ck^{\frac12 -\delta},
\end{align}

Next, we estimate $I_4$. By the H\"older, Young, and Gagliardo--Nirenberg inequalities, noting the regularity of the solution, we obtain
\begin{align}\label{equ:12 Holder n}
    \abs{I_4}
    &\leq
    C \bb{E} \left[\sum_{\ell=1}^n \one_{\Omega_{\kappa,\ell-1}} \int_{t_{\ell-1}}^{t_\ell} \left( \norm{\bff{\rho}^\ell}{\bb{L}^4} + \norm{\bff{u}(s)-\bff{u}(t_\ell)}{\bb{L}^4} \right) \norm{\bff{H}_h^{\ell}}{\bb{L}^2} \norm{\bff{\theta}^\ell}{\bb{L}^4} \right]
    \nonumber\\
    &\leq
    Ckh^2 \bb{E}\left[\sum_{\ell=1}^n \norm{\bff{u}}{L^\infty_T(\bb{W}^{1,4})}^2 \norm{\bff{H}_h^\ell}{\bb{L}^2}^2 \right] 
    + Ck \bb{E} \left[\sum_{\ell=1}^n \one_{\Omega_{\kappa,\ell-1}} \norm{\bff{\theta}^\ell}{\bb{L}^2}^2 \right]
    \nonumber\\
    &\quad
    +
    Ck \bb{E} \left[\sum_{\ell=1}^n \one_{\Omega_{\kappa,\ell-1}} k^\frac12 \norm{\bff{u}}{\mathcal{C}^{1/4}_T(\bb{H}^1)}^2 \norm{\bff{H}_h^\ell}{\bb{L}^2}^2 \right]
    + \epsilon k \bb{E} \left[\sum_{\ell=1}^n \one_{\Omega_{\kappa,\ell-1}} \norm{\nabla \bff{\theta}^\ell}{\bb{L}^2}^2 \right]
    \nonumber\\
    &\leq
    Ch^2 \bb{E}\left[\norm{\bff{u}}{L^\infty_T(\bb{H}^2)}^4 + \left( k \sum_{\ell=1}^n  \norm{\bff{H}_h^\ell}{\bb{L}^2}^2 \right)^2\right]  + Ck \bb{E} \left[\sum_{\ell=1}^n \one_{\Omega_{\kappa,\ell-1}} \norm{\bff{\theta}^\ell}{\bb{L}^2}^2 \right]
     \nonumber\\
     &\quad
    +
    C k^{\frac12} \bb{E} \left[ \norm{\bff{u}}{\mathcal{C}^{1/4}_T(\bb{H}^1)}^4 + \left( k \sum_{\ell=1}^n  \norm{\bff{H}_h^\ell}{\bb{L}^2}^2 \right)^2 \right]
    + 
    \epsilon k \bb{E} \left[\sum_{\ell=1}^n \one_{\Omega_{\kappa,\ell-1}} \norm{\nabla \bff{\theta}^\ell}{\bb{L}^2}^2 \right]
    \nonumber\\
    &\leq
    Ch^2 + {Ck^{\frac12}} + Ck \bb{E} \left[\sum_{\ell=1}^n \one_{\Omega_{\kappa,\ell-1}} \norm{\bff{\theta}^\ell}{\bb{L}^2}^2 \right]
    + 
    \epsilon k \bb{E} \left[\sum_{\ell=1}^n \one_{\Omega_{\kappa,\ell-1}} \norm{\nabla \bff{\theta}^\ell}{\bb{L}^2}^2 \right],
\end{align}
where in the last step we also used \eqref{equ:E u 2p general}.

To estimate $I_5$, we utilise \eqref{equ:E us phi}, noting that $\bb{E}\left[\max_{\ell\leq N} \norm{\bff{\theta}^\ell}{\bb{L}^2}^4\right]$ is finite by \eqref{equ:split u}, \eqref{equ:E u 2p general}, and the boundedness of the Ritz projection.
By H\"older's and Young's inequalities, for any $\epsilon>0$,
\begin{align*}
    \abs{I_5}
    &\leq
    \bb{E} \left[\sum_{\ell=1}^n \one_{\Omega_{\kappa,\ell-1}} \int_{t_{\ell-1}}^{t_\ell} \norm{\bff{u}(s)}{\bb{L}^\infty} \left(\norm{\bff{\eta}^\ell}{\bb{L}^2}+ \norm{\bff{\xi}^\ell}{\bb{L}^2} + \norm{\bff{H}(s)-\bff{H}(t_\ell)}{\bb{L}^2} \right) \norm{\bff{\theta}^\ell}{\bb{L}^2} \ds \right]
    \\
    &\leq
    C \bb{E} \left[\sum_{\ell=1}^n \one_{\Omega_{\kappa,\ell-1}} \int_{t_{\ell-1}}^{t_\ell} \norm{\bff{u}(s)}{\bb{L}^\infty}^2  \norm{\bff{\theta}^\ell}{\bb{L}^2}^2 \ds \right]
    \\
    &\quad
    +
    \epsilon \bb{E} \left[\sum_{\ell=1}^n \one_{\Omega_{\kappa,\ell-1}} \int_{t_{\ell-1}}^{t_\ell}  \left(\norm{\bff{\eta}^\ell}{\bb{L}^2}^2 + \norm{\bff{\xi}^\ell}{\bb{L}^2}^2 + \norm{\bff{H}(s)-\bff{H}(t_\ell)}{\bb{L}^2}^2 \right) \ds \right]
    \\
    &\leq
    Ck^{1-\delta} + C\kappa k \bb{E}\left[\sum_{\ell=1}^n \one_{\Omega_{\kappa,\ell-1}} \norm{\bff{\theta}^\ell}{\bb{L}^2}^2 \right] 
    \\
    &\quad
    +
    C h^4 + \epsilon k  \bb{E} \left[\sum_{\ell=1}^n \one_{\Omega_{\kappa,\ell-1}} \norm{\bff{\xi}^\ell}{\bb{L}^2}^2 \right]
    +
    \epsilon \bb{E} \left[\sum_{\ell=1}^n \one_{\Omega_{\kappa,\ell-1}} \int_{t_{\ell-1}}^{t_\ell} \norm{\bff{H}(s)-\bff{H}(t_\ell)}{\bb{L}^2}^2 \ds \right]
    \\
    &\leq
    Ch^4+Ck^{1-\delta} +  C\kappa k \bb{E}\left[\sum_{\ell=1}^n \one_{\Omega_{\kappa,\ell-1}} \norm{\bff{\theta}^\ell}{\bb{L}^2}^2 \right] 
    + 
    \epsilon k  \bb{E} \left[\sum_{\ell=1}^n \one_{\Omega_{\kappa,\ell-1}} \norm{\bff{\xi}^\ell}{\bb{L}^2}^2 \right],
\end{align*}
for any $\delta>0$, where in the last step we also used  Lemma~\ref{lem:Holder D A beta} with $\beta=\frac12$.

For $I_6$, a straightforward application of Young's inequality and Lemma~\ref{lem:Holder D A beta} yields 
\begin{align*}
    \abs{I_6}
    &\leq
    Ck \bb{E}\left[\sum_{\ell=1}^n \one_{\Omega_{\kappa,\ell-1}} \norm{\bff{\theta}^\ell}{\bb{L}^2}^2 \right]
    +
    \epsilon h^2
    +
    \epsilon k \bb{E}\left[\sum_{\ell=1}^n \one_{\Omega_{\kappa,\ell-1}} \norm{\nabla \bff{\theta}^\ell}{\bb{L}^2}^2 \right]
    \\
    &\quad
    +
    \epsilon \bb{E}\left[\sum_{\ell=1}^n \one_{\Omega_{\kappa,\ell-1}} \int_{t_{\ell-1}}^{t_\ell} \norm{\nabla \bff{u}(s)-\nabla \bff{u}(t_\ell)}{\bb{L}^2}^2 \ds \right]
    \\
    &\leq
    Ck \bb{E}\left[\sum_{\ell=1}^n \one_{\Omega_{\kappa,\ell-1}} \norm{\bff{\theta}^\ell}{\bb{L}^2}^2 \right]
    +
    \epsilon k \bb{E}\left[\sum_{\ell=1}^n \one_{\Omega_{\kappa,\ell-1}} \norm{\nabla \bff{\theta}^\ell}{\bb{L}^2}^2 \right]
    +
    Ch^2 + Ck^{1-\delta}.
\end{align*}

Next, we employ Young’s inequality and Lemma~\ref{lem:stab L2}, and proceed as in the estimate of the term $I_4$ to obtain
\begin{align*}
    \abs{I_7}
    &\leq
   	\epsilon k \bb{E}\left[\sum_{\ell=1}^n \one_{\Omega_{\kappa,\ell-1}} \norm{\bff{\theta}^\ell}{\bb{L}^4}^2 \right]
    +
    C \bb{E}\left[\max_{j\leq n} \norm{\bff{\rho}^j}{\bb{L}^4}^4\right]^{\frac12} \bb{E} \left[\left(k\sum_{\ell=1}^n \norm{\nabla \bff{u}_h^\ell}{\bb{L}^2}^2\right)^2 \right]^{\frac12}
    \\
    &\quad
    +
    C k \bb{E}\left[\sum_{\ell=1}^n \one_{\Omega_{\kappa,\ell-1}} \int_{t_{\ell-1}}^{t_\ell}  \norm{\nabla \bff{u}_h^\ell}{\bb{L}^2}^2 \norm{\bff{u}(s)-\bff{u}(t_\ell)}{\bb{L}^4}^2 \ds \right]
    \\
    &\leq
   \epsilon k \bb{E}\left[\sum_{\ell=1}^n \one_{\Omega_{\kappa,\ell-1}} \norm{\bff{\theta}^\ell}{\bb{H}^1}^2 \right]
   +
    Ch^2 + Ck^{\frac12} \bb{E}\left[ \norm{\bff{u}}{\mathcal{C}^{\frac14}_T(\bb{H}^1)}^2 \sum_{\ell=1}^n \one_{\Omega_{\kappa,\ell-1}} k \norm{\nabla \bff{u}_h^\ell}{\bb{L}^2}^2 \right]
    \\
    &\leq
    \epsilon k \bb{E}\left[\sum_{\ell=1}^n \one_{\Omega_{\kappa,\ell-1}} \norm{\bff{\theta}^\ell}{\bb{H}^1}^2 \right]
    +
    Ch^2 + Ck^{\frac12} \bb{E}\left[ \norm{\bff{u}}{\mathcal{C}^{\frac14}_T(\bb{H}^1)}^4 + \left( \sum_{\ell=1}^n k \norm{\nabla \bff{u}_h^\ell}{\bb{L}^2}^2 \right)^2 \right]
    \\
    &\leq
    Ch^2 + Ck^{\frac12} + \epsilon k \bb{E}\left[\sum_{\ell=1}^n \one_{\Omega_{\kappa,\ell-1}} \norm{\bff{\theta}^\ell}{\bb{H}^1}^2 \right].
\end{align*}
For $I_8$, we proceed as in the estimate for $I_5$ to obtain
\begin{align*}
    \abs{I_8}
    &\leq
    {C\bb{E} \left[\sum_{\ell=1}^n \one_{\Omega_{\kappa,\ell-1}} \int_{t_{\ell-1}}^{t_\ell} \norm{\bff{u}(s)}{\bb{L}^\infty} \left(\norm{\nabla \bff{\rho}^\ell}{\bb{L}^2} + \norm{\nabla \bff{\theta}^\ell}{\bb{L}^2} + \norm{\bff{u}(s)-\bff{u}(t_\ell)}{\bb{H}^1} \right) \norm{\bff{\theta}^\ell}{\bb{L}^2} \ds \right] }
    \\
    &\leq
    Ch^2 + Ck^{1-\delta} 
    +
    C\kappa k \bb{E}\left[\sum_{\ell=1}^n \one_{\Omega_{\kappa,\ell-1}} \norm{\bff{\theta}^\ell}{\bb{L}^2}^2 \right]
    +
    \epsilon k \bb{E}\left[\sum_{\ell=1}^n \one_{\Omega_{\kappa,\ell-1}} \norm{\nabla \bff{\theta}^\ell}{\bb{L}^2}^2 \right].
\end{align*}
For the terms $I_9$, $I_{10}$, and $I_{13}$, we apply the Lipschitz continuity assumption on $\mathcal{M}_R$ and $f_R$ to obtain
\begin{align*}
    \abs{I_9}+ \abs{I_{10}} + \abs{I_{13}}
    &\leq
    Ch^4 + Ck^{1-\delta} + C k \bb{E}\left[\sum_{\ell=1}^n \one_{\Omega_{\kappa,\ell-1}} \norm{\bff{\theta}^\ell}{\bb{L}^2}^2 \right] +
    \epsilon k \bb{E}\left[\sum_{\ell=1}^n \one_{\Omega_{\kappa,\ell-1}} \norm{\bff{\xi}^\ell}{\bb{L}^2}^2 \right].
\end{align*}
The terms $I_{11}$ and $I_{12}$ can be estimated easily as
\begin{align*}
    \abs{I_{11}} + \abs{I_{12}} 
    &\leq
    C k \bb{E}\left[\sum_{\ell=1}^n \one_{\Omega_{\kappa,\ell-1}} \norm{\bff{\theta}^\ell}{\bb{L}^2}^2 \right]
    +
    Ch^4
    +
    \epsilon k \bb{E}\left[\sum_{\ell=1}^n \one_{\Omega_{\kappa,\ell-1}} \norm{\bff{\xi}^\ell}{\bb{L}^2}^2 \right].
\end{align*}
For the term $I_{14}$, we split the stochastic integral as
\begin{align*}
	I_{14}
	&=
	\bb{E} \left[\max_{m\leq n} \sum_{\ell=1}^m \one_{\Omega_{\kappa,\ell-1}} \int_{t_{\ell-1}}^{t_\ell}  \inpro{G(\bff{u}(s))-G(\bff{u}_h^{\ell-1})}{\bff{\theta}^{\ell-1}} \dW(s) \right]
	\\
	&\quad
	+
	\bb{E} \left[ \max_{m\leq n} \sum_{\ell=1}^m \one_{\Omega_{\kappa,\ell-1}} \int_{t_{\ell-1}}^{t_\ell}  \inpro{G(\bff{u}(s))-G(\bff{u}_h^{\ell-1})}{\bff{\theta}^\ell-\bff{\theta}^{\ell-1}} \dW(s) \right]
	=: I_{14a}+ I_{14b}.
\end{align*}
For the first term above, noting the assumptions on $G$ and the H\"older continuity of $\bff{u}$, we apply the Burkholder--Davis--Gundy and the Young inequalities to obtain
\begin{align*}
	I_{14a}
	&\leq
	C \bb{E}\left[ \left(\sum_{\ell=1}^n \one_{\Omega_{\kappa,\ell-1}} \int_{t_{\ell-1}}^{t_\ell} \norm{G(\bff{u}(s))-G(\bff{u}_h^{\ell-1})}{\bb{L}^2}^2 \norm{\bff{\theta}^{\ell-1}}{\bb{L}^2}^2 \ds \right)^{\frac12} \right]
	\\
	&\leq
	C \bb{E} \left[ \left(\max_{m\leq n} \one_{\Omega_{\kappa,m-1}} \norm{\bff{\theta}^{m-1}}{\bb{L}^2}\right) \left(\sum_{\ell=1}^n \one_{\Omega_{\kappa,\ell-1}} \int_{t_{\ell-1}}^{t_\ell} \norm{G(\bff{u}(s))-G(\bff{u}_h^{\ell-1})}{\bb{L}^2}^2 \right)^{\frac12} \right]
	\\
	&\leq
	\epsilon \bb{E} \left[\max_{m\leq n} \one_{\Omega_{\kappa,m-1}} \norm{\bff{\theta}^m}{\bb{L}^2}^2 \right]
	+
	Ch^4
	+
	Ck^{1-\delta}
	+
	C k \bb{E}\left[\sum_{\ell=1}^n \one_{\Omega_{\kappa,\ell-1}} \norm{\bff{\theta}^{\ell-1}}{\bb{L}^2}^2 \right].
\end{align*}
For the term $I_{14b}$, we use Young's inequality, It\^o's isometry, the assumptions on $G$, and the H\"older continuity of $\bff{u}$ to obtain
\begin{align*}
	I_{14b}
	&\leq
	C\bb{E} \left[\sum_{\ell=1}^n \norm{\one_{\Omega_{\kappa,\ell-1}} \int_{t_{\ell-1}}^{t_\ell} \left(G(\bff{u}(s))- G(\bff{u}_h^{\ell-1})\right) \dW(s)}{\bb{L}^2}^2 \right] 
	+
	\epsilon \bb{E}\left[\sum_{\ell=1}^n \one_{\Omega_{\kappa,\ell-1}} \norm{\bff{\theta}^\ell -\bff{\theta}^{\ell-1}}{\bb{L}^2}^2 \right]
	\\
	&=
	C\bb{E} \left[\sum_{\ell=1}^n \one_{\Omega_{\kappa,\ell-1}} \int_{t_{\ell-1}}^{t_\ell} \norm{G(\bff{u}(s))- G(\bff{u}_h^{\ell-1})}{\bb{L}^2}^2 \ds \right] 
	+
	\epsilon \bb{E}\left[\sum_{\ell=1}^n \one_{\Omega_{\kappa,\ell-1}} \norm{\bff{\theta}^\ell -\bff{\theta}^{\ell-1}}{\bb{L}^2}^2 \right]
	\\
	&\leq
	Ch^4
	+
	Ck^{1-\delta}
	+
	C k \bb{E}\left[\sum_{\ell=1}^n \one_{\Omega_{\kappa,\ell-1}} \norm{\bff{\theta}^{\ell-1}}{\bb{L}^2}^2 \right]
	+
	\epsilon \bb{E}\left[\sum_{\ell=1}^n \one_{\Omega_{\kappa,\ell-1}} \norm{\bff{\theta}^\ell -\bff{\theta}^{\ell-1}}{\bb{L}^2}^2 \right].
\end{align*}

We substitute all the above estimates into~\eqref{equ:theta n L2 pre}, set $\epsilon=1/16$, and rearrange the terms. Now, continuing from~\eqref{equ:theta n L2 pre}, noting the assumption $h=O(k)$ we obtain for any $\delta>0$,
\begin{align}\label{equ:E 1 theta n L2}
	&\bb{E}\left[\max_{m\leq n} \left( \one_{\Omega_{\kappa,m-1}} \norm{\bff{\theta}^m}{\bb{L}^2}^2 \right) \right] 
	+
	\sum_{\ell=1}^n \bb{E} \left[\one_{\Omega_{\kappa,\ell-1}} \norm{\bff{\theta}^\ell-\bff{\theta}^{\ell-1}}{\bb{L}^2}^2 \right]
	\nonumber\\
	&\quad
	+
	k
	\bb{E}\left[\sum_{\ell=1}^n \one_{\Omega_{\kappa,\ell-1}} \norm{\nabla \bff{\theta}^\ell}{\bb{L}^2}^2 \right]
	+
	k \bb{E}\left[\sum_{\ell=1}^n \one_{\Omega_{\kappa,\ell-1}} \norm{\bff{\xi}^\ell}{\bb{L}^2}^2 \right]
	\nonumber\\
	&\leq
	\bb{E}\left[\norm{\bff{\theta}^0}{\bb{L}^2}^2 \right] 
	+
	Ch^2 + C(1+\kappa)k^{\frac12-\delta}
	+
	C(1+\kappa ) k \sum_{\ell=1}^n \bb{E} \left[\one_{\Omega_{\kappa,\ell-1}} \norm{\bff{\theta}^\ell}{\bb{L}^2}^2 \right].
\end{align}
By choosing $\bff{u}_h^0$ such that $\bb{E}\left[\norm{\bff{\theta}^0}{\bb{L}^2}^2 \right] \leq Ch^2$, say $\bff{u}_h^0=\Pi_h\bff{u}_0$, we infer the required result for sufficiently small $k$ by the discrete Gronwall lemma.
\end{proof}

We also deduce the following estimate in stronger norms.

\begin{proposition}\label{pro:E theta n H1}
Assume that the hypotheses of Proposition~\ref{pro:E theta n L2} hold and $n\in \{1,2,\ldots, N\}$. For any $\delta>0$, we have
\begin{align*}
    \bb{E}\left[\max_{m\leq n} \left( \one_{\Omega_{\kappa,m-1}} \norm{\nabla \bff{\theta}^m}{\bb{L}^2}^2 \right) \right]  
    +
    k  \sum_{\ell=1}^n \bb{E} \left[ \one_{\Omega_{\kappa,\ell-1}} \left(\norm{\nabla \Delta_h \bff{\theta}^\ell}{\bb{L}^2}^2 +\norm{\nabla \bff{\xi}^\ell}{\bb{L}^2}^2\right) \right]
    &\leq
    Ce^{C \kappa^2} \left(h^2+k^{\frac12-\delta}\right),
\end{align*}
where $C$ is a constant depending on $R$, $T$, $\delta$, and the coefficients of the equation, but is independent of $h$ and $k$.
\end{proposition}

\begin{proof}
We put $\bff{\chi}_h= -\Delta_h \bff{\theta}^\ell$ in~\eqref{equ:theta n theta n1}, then multiply the resulting equations by $\one_{\Omega_{\kappa,\ell-1}}$, sum the resulting expression over $\ell\in \{1,2,\ldots,m\}$, take the expectation value, and argue similarly as in~\eqref{equ:12 theta n} and \eqref{equ:theta n L2 pre} to obtain
\begin{align}\label{equ:12 nabla theta n}
    &\frac12 \bb{E}\left[\max_{m\leq n} \left(\one_{\Omega_{\kappa,m-1}} \norm{\nabla \bff{\theta}^m}{\bb{L}^2}^2\right) \right] 
    +
    \frac12 \bb{E}\left[ \sum_{\ell=1}^n \one_{\Omega_{\kappa,\ell-1}} \norm{\nabla \bff{\theta}^\ell- \nabla \bff{\theta}^{\ell-1}}{\bb{L}^2}^2 \right]
    \nonumber\\
    &\leq
    \frac12 \bb{E}\left[\norm{\nabla \bff{\theta}^0}{\bb{L}^2}^2\right]
    +
    \bb{E}\left[\sum_{\ell=1}^n \one_{\Omega_{\kappa,\ell-1}} \inpro{\bff{\rho}^\ell-\bff{\rho}^{\ell-1}}{\Delta_h \bff{\theta}^\ell}\right]
    \nonumber\\
    &\quad
    -
    \bb{E}\left[\sum_{\ell=1}^n \one_{\Omega_{\kappa,\ell-1}} \int_{t_{\ell-1}}^{t_\ell} \inpro{\bff{\eta}^\ell + \bff{\xi}^\ell + \bff{H}(s)-\bff{H}(t_\ell)}{\Delta_h \bff{\theta}^\ell} \ds \right]
    \nonumber\\
    &\quad
    -
    \bb{E}\left[\sum_{\ell=1}^n \one_{\Omega_{\kappa,\ell-1}} \int_{t_{\ell-1}}^{t_\ell} \inpro{\nabla \bff{\xi}^\ell+ \nabla \bff{H}(s)-\nabla \bff{H}(t_\ell)}{\nabla \Delta_h \bff{\theta}^\ell} \ds\right]
    \nonumber\\
    &\quad
    +
    \bb{E}\left[\sum_{\ell=1}^n \one_{\Omega_{\kappa,\ell-1}} \int_{t_{\ell-1}}^{t_\ell} \inpro{(\bff{\rho}^\ell+ \bff{\theta}^\ell+ \bff{u}(s)-\bff{u}(t_\ell)) \times \bff{H}(s)}{\Delta_h \bff{\theta}^\ell} \ds \right]
    \nonumber\\
    &\quad
    +
    \bb{E}\left[\sum_{\ell=1}^n \one_{\Omega_{\kappa,\ell-1}} \int_{t_{\ell-1}}^{t_\ell} \inpro{\bff{u}_h^\ell \times (\bff{\eta}^\ell+ \bff{\xi}^\ell+ \bff{H}(s)-\bff{H}(t_\ell))}{\Delta_h \bff{\theta}^\ell} \ds \right]
    \nonumber\\
    &\quad
    -
    \bb{E}\left[\sum_{\ell=1}^n \one_{\Omega_{\kappa,\ell-1}} \int_{t_{\ell-1}}^{t_\ell} \inpro{\bff{\nu}\cdot \nabla (\bff{\rho}^\ell+ \bff{\theta}^\ell + \bff{u}(s)-\bff{u}(t_\ell))}{\Delta_h \bff{\theta}^\ell} \ds \right]
    \nonumber\\
    &\quad
    -
    \bb{E}\left[\sum_{\ell=1}^n \one_{\Omega_{\kappa,\ell-1}} \int_{t_{\ell-1}}^{t_\ell} \inpro{(\bff{\rho}^\ell+ \bff{\theta}^\ell + \bff{u}(s)-\bff{u}(t_\ell)) \times (\bff{\nu}\cdot \nabla) \bff{u}_h^\ell}{\Delta_h \bff{\theta}^\ell} \ds \right]
    \nonumber\\
    &\quad
    -
    \bb{E}\left[\sum_{\ell=1}^n \one_{\Omega_{\kappa,\ell-1}} \int_{t_{\ell-1}}^{t_\ell} \inpro{\bff{u}(s) \times (\bff{\nu}\cdot \nabla)(\bff{\rho}^\ell+ \bff{\theta}^\ell+\bff{u}(s)-\bff{u}(t_\ell))}{\Delta_h \bff{\theta}^\ell} \ds \right]
    \nonumber\\
    &\quad
    -
    \bb{E}\left[\sum_{\ell=1}^n \one_{\Omega_{\kappa,\ell-1}} \int_{t_{\ell-1}}^{t_\ell} \inpro{\mathcal{M}_R(\bff{u}(s))- \mathcal{M}_R(\bff{u}_h^{\ell-1})}{\Delta_h \bff{\theta}^\ell} \ds \right]
    \nonumber\\
    &\quad
    -
    \bb{E}\left[\sum_{\ell=1}^n \one_{\Omega_{\kappa,\ell-1}} \int_{t_{\ell-1}}^{t_\ell} \inpro{G(\bff{u}(s))- G(\bff{u}_h^{\ell-1})}{\Delta_h \bff{\theta}^\ell} \mathrm{d}W(s) \right].
\end{align}
Similarly, we take $\bff{\phi}_h=k \Delta_h^2 \bff{\theta}^\ell$ and rearrange the terms. Noting the definition of $\Delta_h$ in~\eqref{equ:disc laplacian}, we have
\begin{align}\label{equ:Delta int xi theta n}
    k\bb{E}\left[\sum_{\ell=1}^n \one_{\Omega_{\kappa,\ell-1}} \norm{\nabla \Delta_h \bff{\theta}^\ell}{\bb{L}^2}^2 \right]
    &=
    k\bb{E}\left[\sum_{\ell=1}^n \one_{\Omega_{\kappa,\ell-1}} \inpro{\nabla \bff{\xi}^\ell}{\nabla\Delta_h \bff{\theta}^\ell} \ds \right]
    \nonumber\\
    &\quad
    -
    k\bb{E}\left[\sum_{\ell=1}^n \one_{\Omega_{\kappa,\ell-1}} \inpro{\nabla\Pi_h \bff{\eta}^\ell}{\nabla\Delta_h \bff{\theta}^\ell} \right]
    \nonumber\\
    &\quad
    +
    k\bb{E}\left[\sum_{\ell=1}^n \one_{\Omega_{\kappa,\ell-1}} \inpro{\nabla\Pi_h f_R(\bff{u}(t_\ell))- \nabla\Pi_h f_R(\bff{u}_h^{\ell-1})}{\nabla\Delta_h \bff{\theta}^\ell} \right].
\end{align}
Adding~\eqref{equ:12 nabla theta n} and~\eqref{equ:Delta int xi theta n}, upon rearranging the terms we obtain
\begin{align}\label{equ:I1 to I13}
    &\frac12 \bb{E}\left[\max_{m\leq n} \left(\one_{\Omega_{\kappa,m-1}} \norm{\nabla \bff{\theta}^m}{\bb{L}^2}^2\right) \right] 
    +
    \frac12 \bb{E}\left[ \sum_{\ell=1}^n \one_{\Omega_{\kappa,\ell-1}} \norm{\nabla \bff{\theta}^\ell- \nabla \bff{\theta}^{\ell-1}}{\bb{L}^2}^2 \right]
    \nonumber\\
    &\quad
    +
    k
    \bb{E}\left[ \sum_{\ell=1}^n \one_{\Omega_{\kappa,\ell-1}} \norm{\Delta_h \bff{\theta}^\ell}{\bb{L}^2}^2 \right]
    +
    k
    \bb{E}\left[ \sum_{\ell=1}^n \one_{\Omega_{\kappa,\ell-1}} \norm{\nabla \Delta_h \bff{\theta}^\ell}{\bb{L}^2}^2 \right]
    \nonumber\\
    &\leq
     \frac12 \bb{E}\left[\norm{\nabla \bff{\theta}^0}{\bb{L}^2}^2\right]
     +
    k
    \bb{E}\left[ \sum_{\ell=1}^n \one_{\Omega_{\kappa,\ell-1}} \norm{\Delta_h \bff{\theta}^\ell}{\bb{L}^2}^2 \right]
    \nonumber\\
    &\quad
    +
    \bb{E}\left[ \sum_{\ell=1}^n \one_{\Omega_{\kappa,\ell-1}} \inpro{\bff{\rho}^\ell-\bff{\rho}^{\ell-1}}{\Delta_h \bff{\theta}^\ell}\right]
    \nonumber\\
    &\quad
    -
    \bb{E}\left[ \sum_{\ell=1}^n \one_{\Omega_{\kappa,\ell-1}} \int_{t_{\ell-1}}^{t_\ell} \inpro{\bff{\eta}^\ell + \bff{\xi}^\ell+ \bff{H}(s)-\bff{H}(t_\ell)}{\Delta_h \bff{\theta}^\ell} \ds \right]
    \nonumber\\
    &\quad
    -
     \bb{E}\left[ \sum_{\ell=1}^n \one_{\Omega_{\kappa,\ell-1}} \int_{t_{\ell-1}}^{t_\ell}  \inpro{\nabla \bff{H}(s)-\nabla \bff{H}(t_\ell)}{\nabla \Delta_h \bff{\theta}^\ell} \ds\right]
    \nonumber\\
    &\quad
    -
    k  \bb{E}\left[ \sum_{\ell=1}^n \one_{\Omega_{\kappa,\ell-1}} \inpro{\nabla\Pi_h \bff{\eta}^\ell}{\nabla\Delta_h \bff{\theta}^\ell} \right]
    \nonumber\\
    &\quad
    +
    k  \bb{E}\left[ \sum_{\ell=1}^n \one_{\Omega_{\kappa,\ell-1}}  \inpro{\nabla\Pi_h f_R(\bff{u}(t_\ell))- \nabla\Pi_h f_R(\bff{u}_h^{\ell-1})}{\nabla\Delta_h \bff{\theta}^\ell} \right]
    \nonumber\\
    &\quad
    +
     \bb{E}\left[ \sum_{\ell=1}^n \one_{\Omega_{\kappa,\ell-1}} \int_{t_{\ell-1}}^{t_\ell}  \inpro{(\bff{\rho}^\ell + \bff{\theta}^\ell + \bff{u}(s)-\bff{u}(t_\ell)) \times \bff{H}(s)}{\Delta_h \bff{\theta}^\ell} \ds \right]
    \nonumber\\
    &\quad
    +
    \bb{E}\left[ \sum_{\ell=1}^n \one_{\Omega_{\kappa,\ell-1}} \int_{t_{\ell-1}}^{t_\ell}  \inpro{\bff{u}_h^\ell \times (\bff{\eta}^\ell+ \bff{\xi}^\ell+ \bff{H}(s)-\bff{H}(t_\ell))}{\Delta_h \bff{\theta}^\ell} \ds \right]
    \nonumber\\
    &\quad
    -
     \bb{E}\left[ \sum_{\ell=1}^n \one_{\Omega_{\kappa,\ell-1}} \int_{t_{\ell-1}}^{t_\ell}  \inpro{\bff{\nu}\cdot \nabla (\bff{\rho}^\ell + \bff{\theta}^\ell + \bff{u}(s)-\bff{u}(t_\ell))}{\Delta_h \bff{\theta}^\ell} \ds \right]
    \nonumber\\
    &\quad
    -
     \bb{E}\left[ \sum_{\ell=1}^n \one_{\Omega_{\kappa,\ell-1}} \int_{t_{\ell-1}}^{t_\ell}  \inpro{(\bff{\rho}^\ell+ \bff{\theta}^\ell + \bff{u}(s)-\bff{u}(t_\ell)) \times (\bff{\nu}\cdot \nabla) \bff{u}_h^\ell}{\Delta_h \bff{\theta}^\ell} \ds \right]
    \nonumber\\
    &\quad
    -
     \bb{E}\left[ \sum_{\ell=1}^n \one_{\Omega_{\kappa,\ell-1}} \int_{t_{\ell-1}}^{t_\ell} \inpro{\bff{u}(s) \times (\bff{\nu}\cdot \nabla)(\bff{\rho}^\ell + \bff{\theta}^\ell+\bff{u}(s)-\bff{u}(t_\ell))}{\Delta_h \bff{\theta}^\ell} \ds \right]
    \nonumber\\
    &\quad
    -
     \bb{E}\left[ \sum_{\ell=1}^n \one_{\Omega_{\kappa,\ell-1}} \int_{t_{\ell-1}}^{t_\ell}  \inpro{\mathcal{M}_R(\bff{u}(s))- \mathcal{M}_R(\bff{u}_h^{\ell-1})}{\Delta_h \bff{\theta}^\ell} \ds \right]
    \nonumber\\
    &\quad
    -
     \bb{E}\left[ \sum_{\ell=1}^n \one_{\Omega_{\kappa,\ell-1}} \int_{t_{\ell-1}}^{t_\ell}  \inpro{G(\bff{u}(s))- G(\bff{u}_h^{\ell-1})}{\Delta_h \bff{\theta}^\ell} \mathrm{d}W(s) \right]
    \nonumber\\
    &=:
    \frac12 \bb{E}\left[\norm{\nabla \bff{\theta}^0}{\bb{L}^2}^2\right]
    + I_1+I_2+\cdots+I_{13}.
\end{align}
We will estimate each term on the last line. In what follows, whenever appropriate, we invoke Proposition~\ref{pro:Holder u} and~\ref{pro:E theta n L2} without further elaboration. Let $\epsilon>0$. For the first term, by \eqref{equ:interp disc Lap Delta L2} and Young's inequality we have
\begin{align*}
    \abs{I_1}
    &\leq
    Ck
    \bb{E}\left[ \sum_{\ell=1}^n \one_{\Omega_{\kappa,\ell-1}} \norm{\nabla \bff{\theta}^\ell}{\bb{L}^2}^2 \right]
    +
    \epsilon k
    \bb{E}\left[ \sum_{\ell=1}^n \one_{\Omega_{\kappa,\ell-1}} \norm{\nabla \Delta_h \bff{\theta}^\ell}{\bb{L}^2}^2 \right]
    \\
    &\leq
    Ce^{C\kappa} \left(h^2+k^{\frac12-\delta}\right)
    +
    \epsilon k
    \bb{E}\left[ \sum_{\ell=1}^n \one_{\Omega_{\kappa,\ell-1}} \norm{\nabla \Delta_h \bff{\theta}^\ell}{\bb{L}^2}^2 \right].
\end{align*}
For the second term, by Young's inequality and \eqref{equ:Ritz approx},
\begin{align*}
    \abs{I_2}
    &\leq
    Ck^{-1} \bb{E}\left[\sum_{\ell=1}^n \norm{\bff{\rho}^\ell -\bff{\rho}^{\ell-1}}{\bb{L}^2}^2\right] + \epsilon k \bb{E}\left[\sum_{\ell=1}^n \one_{\Omega_{\kappa,\ell-1}}  \norm{\Delta_h \bff{\theta}^\ell}{\bb{L}^2}^2 \right]
    \\
    &\leq
    Cnh^4 k^{-1}
    +
    \epsilon k \bb{E}\left[ \sum_{\ell=1}^n \one_{\Omega_{\kappa,\ell-1}} \norm{\Delta_h \bff{\theta}^\ell}{\bb{L}^2}^2 \right]
    \\
    &\leq
    Ck + \epsilon k \bb{E}\left[ \sum_{\ell=1}^n \one_{\Omega_{\kappa,\ell-1}} \norm{\Delta_h \bff{\theta}^\ell}{\bb{L}^2}^2 \right],
\end{align*}
where in the last step we used the assumption $h=O(k)$.

Next, noting Lemma~\ref{lem:Holder D A beta} with $\beta=1/2$, by Young's inequality we have
\begin{align*}
    \abs{I_3}
    &\leq
    Ch^4 + C k \bb{E}\left[ \sum_{\ell=1}^n \one_{\Omega_{\kappa,\ell-1}}  \norm{\bff{\xi}^\ell}{\bb{L}^2}^2 \right]
    +
    \epsilon k \bb{E}\left[ \sum_{\ell=1}^n \one_{\Omega_{\kappa,\ell-1}} \norm{\Delta_h \bff{\theta}^\ell}{\bb{L}^2}^2 \right]
    \nonumber\\
    &\quad
    +
    C \bb{E}\left[ \sum_{\ell=1}^n \one_{\Omega_{\kappa,\ell-1}} \int_{t_{\ell-1}}^{t_\ell} \norm{\bff{H}(s)-\bff{H}(t_\ell)}{\bb{L}^2}^2 \ds \right]
    \nonumber\\
    &\leq
    Ce^{C\kappa} \left(h^2+k^{\frac12-\delta}\right)
    +
    \epsilon k \bb{E}\left[ \sum_{\ell=1}^n \one_{\Omega_{\kappa,\ell-1}} \norm{\Delta_h \bff{\theta}^\ell}{\bb{L}^2}^2 \right].
\end{align*}
By a similar argument as in \eqref{equ:12 Holder a}, we have
\begin{align}
    \label{equ:12 Holder b}
    \abs{I_4}
    &\leq
    \epsilon k \bb{E}\left[ \sum_{\ell=1}^n \one_{\Omega_{\kappa,\ell-1}} \norm{\nabla \Delta_h \bff{\theta}^\ell}{\bb{L}^2}^2 \right]
    +
    C \bb{E}\left[ \sum_{\ell=1}^n \one_{\Omega_{\kappa,\ell-1}} \int_{t_{\ell-1}}^{t_\ell} \norm{\nabla\bff{H}(s)-\nabla\bff{H}(t_\ell)}{\bb{L}^2}^2 \ds \right]
    \nonumber\\
    &\leq
    \epsilon k \bb{E}\left[ \sum_{\ell=1}^n \one_{\Omega_{\kappa,\ell-1}} \norm{\nabla \Delta_h \bff{\theta}^\ell}{\bb{L}^2}^2 \right]
    +
    Ck^{\frac12 -\delta}.
\end{align}
For the term $I_5$, we use~\eqref{equ:proj stab W1p} and Young's inequality to deduce
\begin{align*}
    \abs{I_5}
    &\leq
    Ch^2
    +
    \epsilon k \bb{E}\left[ \sum_{\ell=1}^n \one_{\Omega_{\kappa,\ell-1}} \norm{\nabla \Delta_h \bff{\theta}^\ell}{\bb{L}^2}^2 \right].
\end{align*}

To estimate $I_6$, we apply Young's inequality, \eqref{equ:fR v w H1}, and \eqref{equ:Ritz approx}. Noting Proposition~\ref{pro:Holder u} (temporal H\"older continuity of $\bff{u}$) and the Sobolev embeddings $\bb{H}^2\hookrightarrow \bb{W}^{1,4}\hookrightarrow \bb{L}^\infty$, we obtain for any $\delta>0$,
\begin{align}\label{equ:term I6}
    \abs{I_6}
    &\leq
    \epsilon k \bb{E}\left[ \sum_{\ell=1}^n \one_{\Omega_{\kappa,\ell-1}} \norm{\nabla \Delta_h \bff{\theta}^\ell}{\bb{L}^2}^2 \right]
    +
    Ck \bb{E}\left[ \sum_{\ell=1}^n \one_{\Omega_{\kappa,\ell-1}} \norm{\nabla f_R(\bff{u}(t_\ell))- \nabla f_R(\bff{u}(t_{\ell-1}))}{\bb{L}^2}^2 \right]
    \nonumber\\
    &\quad
    +
    Ck \bb{E}\left[ \sum_{\ell=1}^n \one_{\Omega_{\kappa,\ell-1}} \norm{\nabla f_R(\bff{u}(t_{\ell-1}))- \nabla f_R(\bff{u}_h^{\ell-1})}{\bb{L}^2}^2 \right]
    \nonumber\\
    &\leq
    \epsilon k \bb{E}\left[ \sum_{\ell=1}^n \one_{\Omega_{\kappa,\ell-1}}  \norm{\nabla \Delta_h \bff{\theta}^\ell}{\bb{L}^2}^2 \right]
    +
    C k \bb{E}\left[ \sum_{\ell=1}^n \one_{\Omega_{\kappa,\ell-1}} \left(1+\norm{\bff{u}(t_{\ell-1})}{\bb{L}^\infty}^6\right) \norm{\nabla \bff{u}(t_\ell)-\nabla \bff{u}(t_{\ell-1})}{\bb{L}^2}^2 \right]
    \nonumber\\
    &\quad
    +
    C k \bb{E}\left[ \sum_{\ell=1}^n \one_{\Omega_{\kappa,\ell-1}} \left(1+\norm{\bff{u}(t_{\ell-1})}{\bb{L}^\infty}^6\right) \norm{\nabla \bff{u}(t_{\ell-1})}{\bb{L}^4}^2 \norm{\bff{u}(t_\ell)-\bff{u}(t_{\ell-1})}{\bb{L}^4}^2 \right] 
    \nonumber\\
    &\quad
    +
    C k \bb{E}\left[ \sum_{\ell=1}^n \one_{\Omega_{\kappa,\ell-1}} \left(1+\norm{\bff{u}(t_{\ell-1})}{\bb{L}^\infty}^6\right) \norm{\nabla \bff{u}(t_{\ell-1})-\nabla \bff{u}_h^{\ell-1}}{\bb{L}^2}^2 \right]
    \nonumber\\
    &\quad
    +
    C k \bb{E}\left[ \sum_{\ell=1}^n \one_{\Omega_{\kappa,\ell-1}} \left(1+\norm{\bff{u}(t_{\ell-1})}{\bb{L}^\infty}^6\right) \norm{\nabla \bff{u}(t_{\ell-1})}{\bb{L}^4}^2 \norm{\bff{u}(t_{\ell-1})-\bff{u}_h^{\ell-1}}{\bb{L}^4}^2 \right] 
    \nonumber\\
    &\leq
    \epsilon k \bb{E}\left[ \sum_{\ell=1}^n \one_{\Omega_{\kappa,\ell-1}}  \norm{\nabla \Delta_h \bff{\theta}^\ell}{\bb{L}^2}^2 \right]
    +
    C(1+\kappa^4) k^{\frac12 -\delta}
    \nonumber\\
    &\quad
    +
    C(1+\kappa^4) k \bb{E}\left[ \sum_{\ell=1}^n \one_{\Omega_{\kappa,\ell-1}} \norm{\nabla \bff{\theta}^{\ell-1} +\nabla\bff{\rho}^{\ell-1}}{\bb{L}^2}^2 
    +  \norm{\bff{\theta}^{\ell-1} +\bff{\rho}^{\ell-1}}{\bb{L}^4}^2 \right]
    \nonumber\\
    &\leq
    \epsilon k \bb{E}\left[ \sum_{\ell=1}^n \one_{\Omega_{\kappa,\ell-1}}  \norm{\nabla \Delta_h \bff{\theta}^\ell}{\bb{L}^2}^2 \right]
    +
    C (1+\kappa^4) \left(k^{1-\delta}+h^2 + k \bb{E}\left[ \sum_{\ell=1}^n \one_{\Omega_{\kappa,\ell-1}} \norm{\bff{\theta}^{\ell-1}}{\bb{H}^1}^2 \right] \right)
    \nonumber\\
    &\leq
    \epsilon k \bb{E}\left[ \sum_{\ell=1}^n \one_{\Omega_{\kappa,\ell-1}}  \norm{\nabla \Delta_h \bff{\theta}^\ell}{\bb{L}^2}^2 \right]
    +
    C (1+\kappa^4) e^{C\kappa} \left(h^2+k^{\frac12-\delta}\right),
\end{align}
where in the last step we also used Proposition~\ref{pro:E theta n L2}. 

For the term $I_7$, we apply Young's inequality and \eqref{equ:Ritz approx}, then invoke \eqref{equ:E us utl phi} to obtain
\begin{align*}
    \abs{I_7}
    &\leq
    \bb{E}\left[ \sum_{\ell=1}^n \one_{\Omega_{\kappa,\ell-1}} \int_{t_{\ell-1}}^{t_\ell} \norm{\bff{\rho}^\ell +\bff{\theta}^\ell + \bff{u}(s)-\bff{u}(t_\ell)}{\bb{L}^4} \norm{\bff{H}(s)}{\bb{L}^2} \norm{\Delta_h \bff{\theta}^\ell}{\bb{L}^4} \ds \right]
    \\
    &\leq
    \epsilon k \bb{E}\left[ \sum_{\ell=1}^n \one_{\Omega_{\kappa,\ell-1}}  \norm{\Delta_h \bff{\theta}^\ell}{\bb{L}^4}^2 \right]
    +
    Ch^4 \bb{E}\left[\norm{\bff{H}}{L^\infty_T(\bb{L}^2)}^4 + \norm{\bff{u}}{L^\infty_T(\bb{H}^2)}^4 \right]
    \\
    &\quad
    +
    C \bb{E}\left[ \sum_{\ell=1}^n \one_{\Omega_{\kappa,\ell-1}} \int_{t_{\ell-1}}^{t_\ell}  \norm{\bff{u}(s)-\bff{u}(t_\ell)}{\bb{L}^4}^2 \norm{\bff{H}(s)}{\bb{L}^2}^2 \ds \right]
    \\
    &\leq
    \epsilon k \bb{E}\left[ \sum_{\ell=1}^n \one_{\Omega_{\kappa,\ell-1}}  \norm{\nabla \Delta_h \bff{\theta}^\ell}{\bb{L}^2}^2 \right]
    +
    C(1+\kappa)e^{C\kappa} \left(h^2+k^{\frac12 -\delta}\right),
\end{align*}
where in the last step we also used~\eqref{equ:interp disc Lap Delta L2}, \eqref{equ:v Lp interp}, and Proposition~\ref{pro:E theta n L2}. 

Next, we estimate $I_8$. To this end, noting \eqref{equ:split H}, \eqref{equ:trunc sllbar b}, and \eqref{equ:euler}, we write
\begin{align*}
	\bff{\eta}^\ell+ \bff{\xi}^\ell 
	&= \bff{H}(t_\ell)-\bff{H}_h^n
	\\
	&=
	\Delta \bff{u}(t_\ell)- \Delta_h \bff{u}_h^\ell + f_R(\bff{u}(t_\ell)) - \Pi_h f_R(\bff{u}_h^{\ell-1})
	\\
	&=
	(I-\Pi_h)\Delta \bff{u}(t_\ell)+ \Delta_h \bff{\theta}^\ell
	+
	(I-\Pi_h) f_R(\bff{u}(t_\ell))
	+
	\Pi_h \big( f_R(\bff{u}(t_\ell))- f_R(\bff{u}_h^{\ell-1}) \big),
\end{align*}
where in the last step we used the fact that $\Delta_h \mathcal{R}_h= \Pi_h \Delta$. It follows that
\begin{align*}
    \abs{I_8}
    &\leq
	\bb{E}\left[ \sum_{\ell=1}^n \one_{\Omega_{\kappa,\ell-1}} \int_{t_{\ell-1}}^{t_\ell}  \abs{\inpro{\bff{u}_h^\ell \times \left((I-\Pi_h)\Delta \bff{u}(t_\ell)\right)}{\Delta_h \bff{\theta}^\ell}} \ds \right]
	\\
	&\quad
	+
	\bb{E}\left[ \sum_{\ell=1}^n \one_{\Omega_{\kappa,\ell-1}} \int_{t_{\ell-1}}^{t_\ell}  \abs{\inpro{\bff{u}_h^\ell \times \left(
	(I-\Pi_h) f_R(\bff{u}(t_\ell))
	\right)}{\Delta_h \bff{\theta}^\ell}} \ds \right]
	\\
	&\quad
	+
	\bb{E}\left[ \sum_{\ell=1}^n \one_{\Omega_{\kappa,\ell-1}} \int_{t_{\ell-1}}^{t_\ell}  \abs{\inpro{\bff{u}_h^\ell \times \left(\Pi_h \big( f_R(\bff{u}(t_\ell))- f_R(\bff{u}_h^{\ell-1}) \big)\right)}{\Delta_h \bff{\theta}^\ell}} \ds \right]
	\\
	&\quad
	+
	\bb{E}\left[ \sum_{\ell=1}^n \one_{\Omega_{\kappa,\ell-1}} \int_{t_{\ell-1}}^{t_\ell}  \abs{\inpro{\bff{u}_h^\ell \times \left(\bff{H}(s)-\bff{H}(t_\ell)\right)}{\Delta_h \bff{\theta}^\ell}} \ds \right]
	\\
	&\leq
	I_{8a}+I_{8b}+I_{8c}+I_{8d}.
\end{align*}
For the terms $I_{8a}$ and $I_{8b}$, we use \eqref{equ:proj approx} and Young's inequality to obtain
\begin{align*}
	\abs{I_{8a}}+\abs{I_{8b}}
	&\leq
	Ch^2 \bb{E}\left[\norm{\bff{u}}{L^\infty(t_1,T;\bb{H}^3)}^2 \sum_{\ell=1}^n k\norm{\bff{u}_h^\ell}{\bb{L}^4}^2\right]
	\\
	&\quad
	+
	\epsilon k\bb{E}\left[ \sum_{\ell=1}^n \one_{\Omega_{\kappa,\ell-1}} \norm{\Delta_h \bff{\theta}^\ell}{\bb{L}^2}^2\right]
	+
	\epsilon k\bb{E}\left[ \sum_{\ell=1}^n \one_{\Omega_{\kappa,\ell-1}} \norm{\nabla \Delta_h \bff{\theta}^\ell}{\bb{L}^2}^2\right]
	\\
	&\leq
	Ch^2 \bb{E}\left[\norm{\bff{u}}{L^\infty(t_1,T;\bb{H}^3)}^4 + \left( \sum_{\ell=1}^n k\norm{\bff{u}_h^\ell}{\bb{H}^1}^2 \right)^2 \right]
	\\
	&\quad
	+
	\epsilon k\bb{E}\left[ \sum_{\ell=1}^n \one_{\Omega_{\kappa,\ell-1}} \norm{\Delta_h \bff{\theta}^\ell}{\bb{L}^2}^2\right]
	+
	\epsilon k\bb{E}\left[ \sum_{\ell=1}^n \one_{\Omega_{\kappa,\ell-1}} \norm{\nabla \Delta_h \bff{\theta}^\ell}{\bb{L}^2}^2\right]
	\\
	&\leq
	Ch^2 +
	\epsilon k\bb{E}\left[ \sum_{\ell=1}^n \one_{\Omega_{\kappa,\ell-1}} \norm{\Delta_h \bff{\theta}^\ell}{\bb{L}^2}^2\right]
	+
	\epsilon k\bb{E}\left[ \sum_{\ell=1}^n \one_{\Omega_{\kappa,\ell-1}} \norm{\nabla \Delta_h \bff{\theta}^\ell}{\bb{L}^2}^2\right].
\end{align*}
For $I_{8c}$, we use the stability of $\Pi_h$, Lipschitz continuity of $f_R$, and H\"older continuity of $\bff{u}$ to obtain
\begin{align*}
	\abs{I_{8c}}
	&\leq
	Ck \bb{E}\left[ \sum_{\ell=1}^n \one_{\Omega_{\kappa,\ell-1}} \norm{\bff{u}_h^\ell}{\bb{L}^4} \norm{\bff{u}(t_\ell)- \bff{u}(t_{\ell-1})+ \bff{\rho}^{\ell-1}}{\bb{L}^2} \norm{\Delta_h \bff{\theta}^\ell}{\bb{L}^4} \right]
	\\
	&\quad
	+
	Ck \bb{E}\left[ \sum_{\ell=1}^n \one_{\Omega_{\kappa,\ell-1}} \norm{\bff{u}_h^\ell}{\bb{L}^2} \norm{\bff{\theta}^{\ell-1}}{\bb{L}^4} \norm{\Delta_h \bff{\theta}^\ell}{\bb{L}^4} \right]
	\\
	&\leq
	C(h^4+ k^{1 -\delta})\, \bb{E}\left[\norm{\bff{u}}{\mathcal{C}^{1/2 -\delta}_T(\bb{L}^2)}^4 + \left(\sum_{\ell=1}^n k \norm{\bff{u}_h^\ell}{\bb{H}^1}^2 \right)^2 \right]
	+
	C\kappa k \bb{E}\left[ \sum_{\ell=1}^n \one_{\Omega_{\kappa,\ell-1}} \norm{\bff{\theta}^{\ell-1}}{\bb{H}^1}^2 \right]
	\\
	&\quad
	+
	\epsilon k\bb{E}\left[ \sum_{\ell=1}^n \one_{\Omega_{\kappa,\ell-1}} \norm{\Delta_h \bff{\theta}^\ell}{\bb{L}^2}^2\right]
	+
	\epsilon k\bb{E}\left[ \sum_{\ell=1}^n \one_{\Omega_{\kappa,\ell-1}} \norm{\nabla \Delta_h \bff{\theta}^\ell}{\bb{L}^2}^2\right]
	\\
	&\leq
	C(1+\kappa)e^{C\kappa} \left(h^2+k^{\frac12 -\delta}\right) 
	+
	\epsilon k\bb{E}\left[ \sum_{\ell=1}^n \one_{\Omega_{\kappa,\ell-1}} \norm{\Delta_h \bff{\theta}^\ell}{\bb{L}^2}^2\right]
	+
	\epsilon k\bb{E}\left[ \sum_{\ell=1}^n \one_{\Omega_{\kappa,\ell-1}} \norm{\nabla \Delta_h \bff{\theta}^\ell}{\bb{L}^2}^2\right],
\end{align*}
where in the last step we again used Proposition~\ref{pro:E theta n L2}.
For $I_{8d}$, by a similar argument as in Lemma~\ref{lem:Holder D A beta}, we have
\begin{align*}
	\abs{I_{8d}} 
	&\leq
	C \bb{E}\left[ \sum_{\ell=1}^n \one_{\Omega_{\kappa,\ell-1}} \int_{t_{\ell-1}}^{t_\ell} \norm{\bff{u}_h^\ell}{\bb{L}^2}^2 \norm{\bff{H}(t_\ell)-\bff{H}(s)}{\bb{L}^4}^2 \right]
    +
	\epsilon k\bb{E}\left[ \sum_{\ell=1}^n \one_{\Omega_{\kappa,\ell-1}} \norm{\Delta_h \bff{\theta}^\ell}{\bb{L}^4}^2\right]
    \\
    &\leq
    C \left(\bb{E} \left[\max_{j\leq n} \norm{\bff{u}_h^j}{\bb{L}^2}^4\right]\right)^{\frac12} 
    \left(\sum_{\ell=1}^n \int_{t_{\ell-1}}^{t_\ell} \left(\bb{E} \norm{\bff{H}(t_\ell)-\bff{H}(s)}{\bb{L}^4}^4 \right)^{\frac12} \ds \right)
    +
	\epsilon k\bb{E}\left[ \sum_{\ell=1}^n \one_{\Omega_{\kappa,\ell-1}} \norm{\Delta_h \bff{\theta}^\ell}{\bb{L}^4}^2\right]
	\\
	&\leq
	Ck^{\frac12 -\delta} + \epsilon k\bb{E}\left[ \sum_{\ell=1}^n \one_{\Omega_{\kappa,\ell-1}} \norm{\Delta_h \bff{\theta}^\ell}{\bb{L}^2}^2\right]
	+
	\epsilon k\bb{E}\left[ \sum_{\ell=1}^n \one_{\Omega_{\kappa,\ell-1}} \norm{\nabla \Delta_h \bff{\theta}^\ell}{\bb{L}^2}^2\right],
\end{align*}
where in the last step we essentially used \eqref{equ:Holder est} with $\beta=3/4$.
Hence, altogether we obtain for any $\delta>0$,
\begin{align*}
	\abs{I_8} 
	&\leq
	C(1+\kappa)e^{C\kappa} \left(h^2+k^{\frac12 -\delta}\right) 
	+
	\epsilon k\bb{E}\left[ \sum_{\ell=1}^n \one_{\Omega_{\kappa,\ell-1}} \norm{\Delta_h \bff{\theta}^\ell}{\bb{L}^2}^2\right]
	+
	\epsilon k\bb{E}\left[ \sum_{\ell=1}^n \one_{\Omega_{\kappa,\ell-1}} \norm{\nabla \Delta_h \bff{\theta}^\ell}{\bb{L}^2}^2\right].
\end{align*}
For the terms $I_{10}$ and $I_{11}$, we apply a similar argument as above to infer that
\begin{align*}
    \abs{I_{10}}
    &\leq
    \bb{E}\left[ \sum_{\ell=1}^n \one_{\Omega_{\kappa,\ell-1}} \int_{t_{\ell-1}}^{t_\ell} \left(\norm{\bff{\rho}^\ell}{\bb{L}^4}+ \norm{\bff{\theta}^\ell}{\bb{L}^4} + \norm{\bff{u}(s)-\bff{u}(t_\ell)}{\bb{L}^4} \right) \norm{\nabla \bff{u}_h^\ell}{\bb{L}^2}  \norm{\Delta_h \bff{\theta}^\ell}{\bb{L}^4} \right] \ds 
    \\
    &\leq
    C(1+\kappa)e^{C\kappa} \left(h^2+k^{\frac12 -\delta}\right)
    +
    \epsilon k\bb{E}\left[ \sum_{\ell=1}^n \one_{\Omega_{\kappa,\ell-1}} \norm{\nabla \Delta_h \bff{\theta}^\ell}{\bb{L}^2}^2\right]
    +
	\epsilon k\bb{E}\left[ \sum_{\ell=1}^n \one_{\Omega_{\kappa,\ell-1}} \norm{\Delta_h \bff{\theta}^\ell}{\bb{L}^2}^2\right],
    \\
    \abs{I_{11}}
    &\leq
    \bb{E}\left[ \sum_{\ell=1}^n \one_{\Omega_{\kappa,\ell-1}} \int_{t_{\ell-1}}^{t_\ell}\norm{\bff{u}(s)}{\bb{L}^\infty} \left(\norm{\nabla \bff{\rho}^\ell}{\bb{L}^2}+ \norm{\nabla \bff{\theta}^\ell}{\bb{L}^2} + \norm{\nabla \bff{u}(s)- \nabla\bff{u}(t_\ell)}{\bb{L}^2} \right)  \norm{\Delta_h \bff{\theta}^\ell}{\bb{L}^2} \right] \ds 
    \\
    &\leq
    C(1+\kappa)e^{C\kappa} \left(h^2+k^{\frac12 -\delta}\right)
    +
    \epsilon k\bb{E}\left[ \sum_{\ell=1}^n \one_{\Omega_{\kappa,\ell-1}} \norm{\Delta_h \bff{\theta}^\ell}{\bb{L}^2}^2\right].
\end{align*}
Next, by Young's inequality and \eqref{equ:Ritz approx}, it is easy to see that
\begin{align*}
    \abs{I_9}+\abs{I_{12}}
    &\leq
    \epsilon k\bb{E}\left[ \sum_{\ell=1}^n \one_{\Omega_{\kappa,\ell-1}} \norm{\nabla \Delta_h \bff{\theta}^\ell}{\bb{L}^2}^2\right]
    +
    \epsilon k\bb{E}\left[ \sum_{\ell=1}^n \one_{\Omega_{\kappa,\ell-1}} \norm{\Delta_h \bff{\theta}^\ell}{\bb{L}^2}^2\right]
    +
    Ce^{C\kappa} \left(h^2+k^{\frac12 -\delta}\right).
\end{align*}
Finally, we split the stochastic integral in $I_{13}$ as
\begin{align*}
	I_{13}
	&=
	\bb{E} \left[\max_{m\leq n} \sum_{\ell=1}^m \one_{\Omega_{\kappa,\ell-1}} \int_{t_{\ell-1}}^{t_\ell}  \inpro{\nabla\Pi_h G(\bff{u}(s))-\nabla\Pi_h G(\bff{u}_h^{\ell-1})}{\nabla \bff{\theta}^{\ell-1}} \dW(s) \right]
	\\
	&\quad
	+
	\bb{E} \left[ \max_{m\leq n} \sum_{\ell=1}^m \one_{\Omega_{\kappa,\ell-1}} \int_{t_{\ell-1}}^{t_\ell}  \inpro{\nabla\Pi_h G(\bff{u}(s))-\nabla\Pi_h G(\bff{u}_h^{\ell-1})}{\nabla \bff{\theta}^\ell- \nabla \bff{\theta}^{\ell-1}} \dW(s) \right]
	\\
	&=: I_{13a}+ I_{13b}.
\end{align*}
For the first term above, noting the assumptions on $G$, the $\bb{H}^1$ stability of $\Pi_h$, and Lemma~\ref{lem:Holder D A beta}, we apply the Burkholder--Davis--Gundy and the Young inequalities to obtain
\begin{align*}
	I_{13a}
	&\leq
	C \bb{E}\left[ \left(\sum_{\ell=1}^n \one_{\Omega_{\kappa,\ell-1}} \int_{t_{\ell-1}}^{t_\ell} \norm{\nabla G(\bff{u}(s))-\nabla G(\bff{u}_h^{\ell-1})}{\bb{L}^2}^2 \norm{\nabla \bff{\theta}^{\ell-1}}{\bb{L}^2}^2 \ds \right)^{\frac12} \right]
	\\
	&\leq
	C \bb{E} \left[ \left(\max_{m\leq n} \one_{\Omega_{\kappa,m-1}} \norm{\nabla \bff{\theta}^{m-1}}{\bb{L}^2} \right) \left(\sum_{\ell=1}^n \one_{\Omega_{\kappa,\ell-1}} \int_{t_{\ell-1}}^{t_\ell} \norm{\nabla \bff{u}(s)-\nabla \bff{u}_h^{\ell-1}}{\bb{L}^2}^2 \right)^{\frac12} \right]
	\\
	&\leq
	\epsilon \bb{E} \left[\max_{m\leq n} \one_{\Omega_{\kappa,m-1}} \norm{\nabla \bff{\theta}^m}{\bb{L}^2}^2 \right]
	+
	Ce^{C\kappa} \left(h^2+k^{\frac12 -\delta}\right).
\end{align*}
For the term $I_{13b}$, we use Young's inequality, It\^o's isometry, the assumptions on $G$, and Lemma~\ref{lem:Holder D A beta} to obtain
\begin{align*}
	I_{13b}
	&\leq
	C\bb{E} \left[\sum_{\ell=1}^n \norm{\one_{\Omega_{\kappa,\ell-1}} \int_{t_{\ell-1}}^{t_\ell} \left(\nabla\Pi_h G(\bff{u}(s))- \nabla\Pi_h G(\bff{u}_h^{\ell-1})\right) \dW(s)}{\bb{L}^2}^2 \right] 
    \\
    &\quad
	+
	\epsilon \bb{E}\left[\sum_{\ell=1}^n \one_{\Omega_{\kappa,\ell-1}} \norm{\nabla \bff{\theta}^\ell -\nabla\bff{\theta}^{\ell-1}}{\bb{L}^2}^2 \right]
	\\
	&=
	C\bb{E} \left[\sum_{\ell=1}^n \one_{\Omega_{\kappa,\ell-1}} \int_{t_{\ell-1}}^{t_\ell} \norm{\nabla\Pi_h G(\bff{u}(s))- \nabla\Pi_h G(\bff{u}_h^{\ell-1})}{\bb{L}^2}^2 \ds \right] 
	\\
	&\quad
	+
	\epsilon \bb{E}\left[\sum_{\ell=1}^n \one_{\Omega_{\kappa,\ell-1}} \norm{\nabla \bff{\theta}^\ell -\nabla \bff{\theta}^{\ell-1}}{\bb{L}^2}^2 \right]
	\\
	&\leq
	Ce^{C\kappa} \left(h^2+k^{\frac12 -\delta}\right)
	+
	\epsilon \bb{E}\left[\sum_{\ell=1}^n \one_{\Omega_{\kappa,\ell-1}} \norm{\nabla \bff{\theta}^\ell -\nabla \bff{\theta}^{\ell-1}}{\bb{L}^2}^2 \right].
\end{align*}

We now substitute all the above estimates into~\eqref{equ:I1 to I13}, set $\epsilon=1/16$, and rearrange the terms.
Altogether, continuing from \eqref{equ:I1 to I13} we deduce that
\begin{align}\label{equ:E nab delta theta}
    &\bb{E}\left[\max_{m\leq n} \left(\one_{\Omega_{\kappa,m-1}} \norm{\nabla \bff{\theta}^m}{\bb{L}^2}^2\right) \right] 
    +
    k
    \bb{E}\left[ \sum_{\ell=1}^n \one_{\Omega_{\kappa,\ell-1}} \norm{\Delta_h \bff{\theta}^\ell}{\bb{L}^2}^2 \right]
    +
    k
    \bb{E}\left[ \sum_{\ell=1}^n \one_{\Omega_{\kappa,\ell-1}} \norm{\nabla \Delta_h \bff{\theta}^\ell}{\bb{L}^2}^2 \right]
    \nonumber \\
    &\leq
    \bb{E}\left[\norm{\nabla \bff{\theta}^0}{\bb{L}^2}^2 \right] 
    +
    Ce^{C\kappa} \left(h^2+k^{\frac12 -\delta}\right).
\end{align}
Furthermore, note that by setting $\bff{\phi}_h= -k\Delta_h \bff{\xi}^n$ in \eqref{equ:Hhn Htn}, multiplying by $\one_{\Omega_{\kappa,n-1}}$ and taking expectation, then applying \eqref{equ:disc laplacian} and \eqref{equ:Ritz zero}, we obtain
\begin{align}\label{equ:kE nab xi}
    k \bb{E} \left[\one_{\Omega_{\kappa,n-1}} \norm{\nabla\bff{\xi}^n}{\bb{L}^2}^2 \right] 
    &=
    k \bb{E}\left[\one_{\Omega_{\kappa,n-1}} \inpro{\nabla\Pi_h \bff{\eta}^n}{\nabla\bff{\xi}^n} \right] 
    +
    k \bb{E}\left[\one_{\Omega_{\kappa,n-1}} \inpro{\nabla\Delta_h \bff{\theta}^n}{\nabla\bff{\xi}^n} \right]
    \nonumber\\
    &\quad
    +
    k\bb{E} \left[\one_{\Omega_{\kappa,n-1}} \inpro{\nabla\Pi_h f_R(\bff{u}(t_n)) - \nabla\Pi_h f_R(\bff{u}_h^{n-1})}{\nabla \bff{\xi}^n} \right]
    \nonumber\\
    &\leq
    Ckh^2 + \epsilon k \bb{E}\left[\one_{\Omega_{\kappa,n-1}} \norm{\nabla\bff{\xi}^n}{\bb{L}^2}^2 \right] 
    +
    Ck \bb{E}\left[\one_{\Omega_{\kappa,n-1}} \norm{\nabla\Delta_h \bff{\theta}^n}{\bb{L}^2}^2 \right] 
    \nonumber\\
    &\quad
    +
    Ck \bb{E} \left[\one_{\Omega_{\kappa,n-1}} \norm{\nabla\Pi_h f_R(\bff{u}(t_n)) - \nabla\Pi_h f_R(\bff{u}_h^{n-1})}{\bb{L}^2}^2 \right], 
\end{align}
where in the last step we used Young's inequality and \eqref{equ:Ritz approx}. The final term in \eqref{equ:kE nab xi} can be estimated using \eqref{equ:fR v w H1} as done in~\eqref{equ:term I6}. Summing~\eqref{equ:kE nab xi} over $\ell \in \{1,2,\ldots,n\}$ and applying~\eqref{equ:E nab delta theta}, we obtain
\begin{align*}
    k \bb{E}\left[ \sum_{\ell=1}^n \one_{\Omega_{\kappa,\ell-1}} \norm{\nabla \bff{\xi}^\ell}{\bb{L}^2}^2 \right]
    &\leq
    \bb{E}\left[\norm{\nabla \bff{\theta}^0}{\bb{L}^2}^2 \right] 
    +
    C e^{C\kappa} \left(h^2+k^{\frac12 -\delta}\right).
\end{align*}
Choosing $\bff{u}_h^0$ such that $\bb{E}\left[\norm{\nabla \bff{\theta}^0}{\bb{L}^2}^2 \right] \leq Ch^2$, say $\bff{u}_h^0=\Pi_h \bff{u}_0$, we deduce the required result from the last estimate and \eqref{equ:E nab delta theta}.
\end{proof}

The following error estimate, which holds over a sample space with large probability, now follows from the above propositions. Indeed, note that by Chebyshev's inequality,
\begin{align*}
    \bb{P}\left[\Omega_{\kappa,m}\right] 
    &\geq
    1- \frac{1}{\kappa} \left(\bb{E} \left[ \max_{t\leq t_m \wedge T} \norm{\bff{u}(t)}{\bb{H}^2}^2 \right] 
    +
    \bb{E} \left[ \max_{t\leq t_m \wedge T} \norm{\bff{H}(t)}{\bb{L}^2}^2 \right]
    + \bb{E} \left[\max_{n\leq m} \norm{\bff{u}_h^n}{\bb{H}^1}^2 \right] \right)
    \geq
    1-\frac{C_{R,T}}{\kappa}.
\end{align*}
Here, $C_{R,T}$ is a constant depending on $R$, $T$, and the coefficients of the equation, which is conferred by Lemma~\ref{lem:stab H1} and Proposition~\ref{pro:Holder u}. Therefore, $\bb{P}\left[\Omega_{\kappa,m}\right] \to 1$ as $\kappa\to\infty$.

\begin{theorem}\label{the:uhn un 1 L2}
Assume that the hypotheses of Proposition~\ref{pro:E theta n L2} hold and $n\in \{1,2,\ldots, N\}$. For any $\delta>0$, we have
\begin{align}\label{equ:error CT}
    &\bb{E} \left[\max_{m\leq n} \left(\one_{\Omega_{\kappa,m-1}} \norm{\bff{u}(t_m)-\bff{u}_h^m}{\bb{H}^1}^2 \right) \right] 
    +
    k  \sum_{\ell=1}^n \bb{E} \left[ \one_{\Omega_{\kappa,\ell-1}} \norm{\bff{H}(t_\ell)-\bff{H}_h^\ell}{\bb{H}^1}^2 \right]
    \leq
    \widetilde{C}e^{\widetilde{C}\kappa} \left(h^2+k^{\frac12-\delta}\right),
\end{align}
where $\widetilde{C}$ is a constant depending on $R$, $T$, and $\delta$, but is independent of $h$ and $k$.
\end{theorem}

\begin{proof}
This follows from Propositions~\ref{pro:E theta n L2} and~\ref{pro:E theta n H1}, equations~\eqref{equ:split u} and~\eqref{equ:split H}, estimate~\eqref{equ:Ritz approx}, and the triangle inequality.
\end{proof}

We now define the following quantities:
\begin{align}\label{equ:An exp}
    A_n&:= \max_{m\leq n} \left(\one_{\Omega_{\kappa,m-1}} \norm{\bff{u}(t_m)-\bff{u}_h^m}{\bb{H}^1}^2 \right)
    +
    k  \sum_{\ell=1}^n \one_{\Omega_{\kappa,\ell-1}} \norm{\bff{H}(t_\ell)-\bff{H}_h^\ell}{\bb{H}^1}^2,
    \\
    \label{equ:An comp}
    \widetilde{A_n} &:= \max_{m\leq n} \left(\one_{\Omega_{\kappa,m-1}^\complement} \norm{\bff{u}(t_m)-\bff{u}_h^m}{\bb{H}^1}^2 \right)
    +
    k  \sum_{\ell=1}^n \one_{\Omega_{\kappa,\ell-1}^\complement} \norm{\bff{H}(t_\ell)-\bff{H}_h^\ell}{\bb{H}^1}^2.
\end{align}
By selecting a suitable value of $\kappa$, we establish several convergence results. We begin by stating a theorem regarding the rate of convergence in probability.

\begin{theorem}\label{the:uhn prob}
Assume that the hypotheses of Proposition~\ref{pro:E theta n L2} hold and $n\in \{1,2,\ldots, N\}$. For any $\delta>0$, we have
\begin{align*}
    \lim_{h,k\to 0^+}
    \bb{P}\left[\max_{m\leq n} \norm{\bff{u}(t_m)-\bff{u}_h^m}{\bb{H}^1}^2 +
    k  \sum_{\ell=1}^n \norm{\bff{H}(t_\ell)-\bff{H}_h^\ell}{\bb{H}^1}^2 \geq  \alpha \left(h^{2(1-2\delta)}+k^{\frac12 (1-8\delta)}\right) \right] = 0
\end{align*}
for any $\alpha,\delta>0$.
\end{theorem}

\begin{proof}
By Chebyshev's inequality and Theorem~\ref{the:uhn un 1 L2} with $\kappa = O\left( \log (\log 1/h)\right)$, for any $\alpha,\delta>0$ we have
\begin{align*}
    &\bb{P}\left[\max_{m\leq n} \norm{\bff{u}(t_m)-\bff{u}_h^m}{\bb{H}^1}^2 +
    k  \sum_{\ell=1}^n \norm{\bff{H}(t_\ell)-\bff{H}_h^\ell}{\bb{H}^1}^2 \geq  \alpha (h^{2(1-2\delta)}+k^{\frac12 (1-8\delta)}) \right]
    \\
    &\leq
    \alpha^{-1}  \left(h^{2(1-2\delta)}+k^{\frac12 (1-8\delta)} \right)^{-1}
    \bb{E} \left[A_n\right] + \bb{P} \left[\Omega_{\kappa,n-1}^\complement\right]
    \\
    &\leq 
    C \alpha^{-1}  \left(h^{2(1-2\delta)}+k^{\frac12 (1-8\delta)} \right)^{-1}
    \left(h^{2(1-\delta)}+ k^{\frac12 (1-4\delta)} \right)
    +
    C_{R,T} \big(\log(\log(1/h))\big)^{-\frac12},
\end{align*}
which tends to $0$ as $h,k\to 0^+$.
\end{proof}

We now assume that $\beta_2=0$ to derive a strong order of convergence for the scheme.

\begin{theorem}\label{the:uhn strong}
Suppose that $\beta_2=0$. Assume that the hypotheses of Proposition~\ref{pro:E theta n L2} hold and $n\in \{1,2,\ldots, N\}$. For any $\delta>0$, we have
\begin{align}\label{equ:limit hk}
    \bb{E}\left[\max_{m\leq n} \norm{\bff{u}(t_m)-\bff{u}_h^m}{\bb{H}^1}^2 +
    k  \sum_{\ell=1}^n \norm{\bff{H}(t_\ell)-\bff{H}_h^\ell}{\bb{H}^1}^2 \right] \leq
    C \abs{\log \left(h^2+k^{\frac12-\delta}\right)}^{-r}, \quad \forall r\geq 1.
\end{align}
The constant $C$ depends on $R$, $T$, $r$, and $\delta$, but is independent of $h$ and $k$.
In particular, the right-hand side of~\eqref{equ:limit hk} tends to $0$ as $h,k\to 0^+$.
\end{theorem}

\begin{proof}
Note that by H\"older's inequality with exponents $2^{q-1}$ and $p=2^{q-1}/(2^{q-1}-1)$, where $q>1$, we have
\begin{align}\label{equ:max compl u}
    \bb{E} \left[\max_{m\leq n} \one_{\Omega_{\kappa,m-1}^\complement} \norm{\bff{u}(t_m)- \bff{u}_h^m}{\bb{H}^1}^2 \right]
    \leq
    C\left[\bb{P} \left(\Omega_{\kappa,n-1}^\complement\right)\right]^{\frac{1}{p}} 
    \left[\bb{E}\left(\max_{t\in [0,T]} \norm{\bff{u}(t)}{\bb{H}^1}^{2^q} + \max_{m\leq n} \norm{\bff{u}_h^m}{\bb{H}^1}^{2^q}\right) \right]^{\frac{1}{2^{q-1}}}.
\end{align}
Similarly,
\begin{align}\label{equ:max compl H}
    \bb{E} \left[k \sum_{\ell=1}^n \one_{\Omega_{\kappa,\ell-1}^\complement} \norm{\bff{H}(t_\ell)-\bff{H}_h^\ell}{\bb{H}^1}^2 \right]
    \leq
    C\left[\bb{P} \left(\Omega_{\kappa,n-1}^\complement\right)\right]^{\frac{1}{p}} 
    \left[\bb{E}\left(k \sum_{\ell=1}^n \norm{\bff{H}(t_\ell)}{\bb{H}^1}^2 + \norm{\bff{H}_h^\ell}{\bb{H}^1}^2\right)^{2^{q-1}} \right]^{\frac{1}{2^{q-1}}}.
\end{align}
The last terms on the right-hand side of \eqref{equ:max compl u} and \eqref{equ:max compl H} are bounded due to the assumed regularity in Proposition~\ref{pro:Holder u} and the stability estimate~\eqref{equ:H1 high}. Therefore, it remains to establish a bound for the probability of the `bad' set $\Omega_{\kappa,n-1}^\complement$. To this end, by Chebyshev's inequality and the definition of the set $\Omega_{\kappa,n-1}$,
\begin{align*}
    \bb{P} \left(\Omega_{\kappa,n-1}^\complement\right)
    \leq
    \kappa^{-2^{q-1}} \left[ \bb{E}\left( \max_{t\in [0,T]} \norm{\bff{u}(t)}{\bb{H}^2}^{2^q} + \max_{t\in [0,T]} \norm{\bff{H}(t)}{\bb{L}^2}^{2^q} + \max_{m\leq n} \norm{\bff{u}_h^m}{\bb{H}^1}^{2^q} \right) \right],
\end{align*}
which implies by the definition \eqref{equ:An comp},
\begin{align}\label{equ:E An tilde}
    \bb{E}\left[\widetilde{A_n}\right] \leq C_q \kappa^{-2^{q-1}}.
\end{align}
For sufficiently small $h$ and $k$, we set
$\kappa = \frac{1}{\widetilde{C}}\Big( \big|\log (h^2+k^{\frac12-\delta})\big| - (2^{q-1}-1) \log \big|\log (h^2+k^{\frac12-\delta})\big| \Big)$,
where $\widetilde{C}$ is the constant in~\eqref{equ:error CT}. With this choice of $\kappa$, noting \eqref{equ:E An tilde}, we have by~\eqref{equ:error CT}, \eqref{equ:max compl u}, and \eqref{equ:max compl H},
\begin{align*}
    \bb{E}\left[\max_{m\leq n} \norm{\bff{u}(t_m)-\bff{u}_h^m}{\bb{H}^1}^2 +
    k  \sum_{\ell=1}^n \norm{\bff{H}(t_\ell)-\bff{H}_h^\ell}{\bb{H}^1}^2 \right]
    &=
    \bb{E}\left[A_n\right] + \bb{E} \left[\widetilde{A_n}\right]
    \\
    &\leq
    \widetilde{C} e^{\widetilde{C}\kappa} \left(h^2+k^{\frac12-\delta}\right)
    +
    C_q \kappa^{-2^{q-1}}
    \\
    &\leq
    C_r \abs{\log \left(h^2+k^{\frac12-\delta}\right)}^{-r},
\end{align*}
for any $r\geq 1$. This completes the proof of the theorem.
\end{proof}

\begin{remark}
If the initial data $\bff{u}_0$ and the noise are more regular, say belonging to $\mathrm{D}(A^{\frac34})$, then by a similar argument as in~\cite{GolSoeTra24b}, one can show that the pathwise solution $\bff{u}$ of \eqref{equ:weakform} belongs to $L^p\big(\Omega; C^\alpha (0,T;\mathrm{D}(A^{\frac34}))\big) \cap L^p\big(\Omega; C^\alpha ([0,T];\mathrm{D}(A^{\frac14}))\big)$, where $\alpha\in (0,\frac12)$, for any $p\geq 1$. In that case, an $O(k^{1-\delta})$ bound can be obtained in~\eqref{equ:12 Holder a}, \eqref{equ:12 Holder n}, and \eqref{equ:12 Holder b}, leading to an $O(k^{1-\delta})$ bound in Proposition~\ref{pro:E theta n L2}, Proposition~\ref{pro:E theta n H1}, and Theorem~\ref{the:uhn un 1 L2} (instead of $O(k^{\frac12-\delta})$ as stated currently). Consequently, the right-hand side of \eqref{equ:limit hk} would read $C_r \abs{\log \left(h^2+k^{1-\delta}\right)}^{-r}$ in this case.
\end{remark}

\begin{remark}
The convergence rate established in Theorem~\ref{the:uhn strong} is likely suboptimal. This is due to the technique of estimating errors on the event $\Omega_{\kappa,m}$, where the error bounds depend exponentially on the truncation parameter $\kappa$. It accounts for the rare possibility of a certain `blow-up' events defined by $\Omega_{\kappa,m}^\complement$. Optimising the choice of $\kappa$ with respect to $h$ and $k$ to control the probability of the complement $\Omega_{\kappa,m}^\complement$ typically results in a reduced theoretical order. However, as illustrated in the numerical experiments (Section~\ref{sec:num exp}), the observed rate of convergence is significantly better, aligning more closely with the results of Theorems~\ref{the:uhn un 1 L2} and~\ref{the:uhn prob}.
\end{remark}

\section{Numerical experiments}\label{sec:num exp}

We present a set of numerical experiments to validate the theoretical convergence properties of the proposed finite element scheme for the sLLBar equation~\eqref{equ:sllbar}, with $G(\bff{u})=\lambda_1 \bff{g}- \gamma \bff{u}\times \bff{g}$ (where $\bff{g}$ is to be specified) and $\mathcal{M}(\bff{u})=\bff{0}$. All computations are carried out in the~\textsc{FEniCS} environment. The computational domains $\mathscr{D}$ are taken to be the unit interval and the unit square. The magnetisation vector field $\bff{u}$ and the effective field $\bff{H}$ are discretised in space using continuous piecewise linear finite elements on a family of quasi-uniform meshes.

To assess convergence, we compute a reference solution on a fine mesh and with a small time step. This reference solution serves as an approximation of the exact stochastic solution for each realisation of the Wiener process. For coarser discretisations, we use the same Brownian path so that the difference between the numerical solution and the reference is meaningful pathwise. The errors $\mathcal{E}_{s}^{\bff{u}}(h,k)$ and $\mathcal{E}_{s}^{\bff{H}}(h,k)$ at final time $T$ is then measured in $L^2(\Omega; \bb{H}^s)$-norm for $s=0$ or $1$, defined by
\[
\mathcal{E}_{s}^{\bff{u}}(h,k):= \left(\bb{E}\norm{\bff{u}_h^N- \bff{u}_{\text{ref}}(T)}{\bb{H}^s}^2\right)^{\frac12}
\;\text{ and }\;
\mathcal{E}_{s}^{\bff{H}}(h,k):= \left(\bb{E}\norm{\bff{H}_h^N- \bff{H}_{\text{ref}}(T)}{\bb{H}^s}^2\right)^{\frac12}.
\]
Here, $\big(\bff{u}_h^N, \bff{H}_h^n\big)$ is the numerical solution with mesh size $h$ at time $T=Nk$, and $\big(\bff{u}_{\text{ref}}(T), \bff{H}_{\text{ref}}(T)\big)$ denotes the reference solution at the same final time along the same realisation of the Brownian motion. In practice, the expectation is approximated by a Monte Carlo average over $M$ independent sample paths.

We vary separately the mesh size $h$ and the time step $k$ to test spatial and temporal convergence. To verify spatial convergence, we fix a sufficiently small time step and compare errors across a sequence of meshes ($h=2^{-j}$ for some consecutive values of $j$). For temporal convergence, we fix a fine spatial mesh and vary the time step size ($k=(25\times 2^j)^{-1}$ for some consecutive values of $j$). In both cases, errors are averaged over $M=25$ realisations. The experimental rate of convergence is obtained by fitting the errors against mesh size or time step in a log-log plot.

\subsection{Simulation 1 (thin wire)}

Set $\mathscr{D}=[0,1]$. We take the parameters to be $\lambda_1=0.02$, $\lambda_2=0.001$, $\gamma=6.0$, $\kappa=0.5$, $\mu=1.0$, $\beta_1=0.1$, $\beta_2=0.05$. The current density is $\nu=1.0$. The initial data is specified to be
\[
    \bff{u}_0(x)=\big(0.1, \cos(2\pi x), \sin(2\pi x)\big),
\]
and the vector field $\bff{g}$ is set to be
\[
    \bff{g}(x)= \big(2\sin (\pi x), \sin (\pi x), 2\cos (2\pi x)\big).
\]
We solve the sLLBar equation by employing the implicit scheme~\eqref{equ:euler}.
Snapshots of a sample path of the magnetisation vector field $\bff{u}$ and the effective field $\bff{H}$ with mesh-size $h=1/16$ at selected times are shown in Figures~\ref{fig:snapshots u 0} and~\ref{fig:snapshots H 0}, respectively. The colour indicates the relative value of the magnitude. 

Figure~\ref{fig:mass energy exp0} shows the energy of the system over 30 independent sample paths for $h=1/16$, $k=1/50$, and for $h=1/32$, $k=1/100$. We recall that the energy~\cite{Soe24} is defined as
\[
\text{Energy}(\bff{u}):= \frac{1}{2} \norm{\nabla \bff{u}}{\bb{L}^2}^2 + \frac{\kappa}{4} \norm{\abs{\bff{u}}^2 - \mu}{\bb{L}^2}^2.
\]
The energy shows minor pathwise fluctuations and, on average, decays over time.

Now, we set $T=0.05$ and fix a reference solution with $h=1/128$ and $k=1/3200$. To verify spatial convergence, we compare errors across a sequence of meshes ($h=2^{-j}$ for $j=2,3,4$). For temporal convergence, we vary the time step size ($k=(25\times 2^j)^{-1}$ for $j=2,3,4$).
Figures~\ref{fig:order u spatial 0} and \ref{fig:order u time 0} display the plots of $\mathcal{E}_s^{\bff{u}}$ against $1/h$ and $1/k$, respectively. Similar plots for $\mathcal{E}_s^{\bff{H}}$ against $1/h$ and $1/k$ are shown in Figures~\ref{fig:order H spatial 0} and~\ref{fig:order H time 0}. These results are consistent with Theorem~\ref{the:uhn un 1 L2} in the $\bb{H}^1$ norm, while the numerical simulation indicates an even higher convergence rate in the $\bb{L}^2$ norm. Such behavior is in line with what is traditionally expected in finite element analysis, though a rigorous proof in our setting remains an interesting open question.

\begin{figure}[!htb]
	\centering
	\begin{subfigure}[b]{0.21\textwidth}
		\centering
		\includegraphics[width=\textwidth]{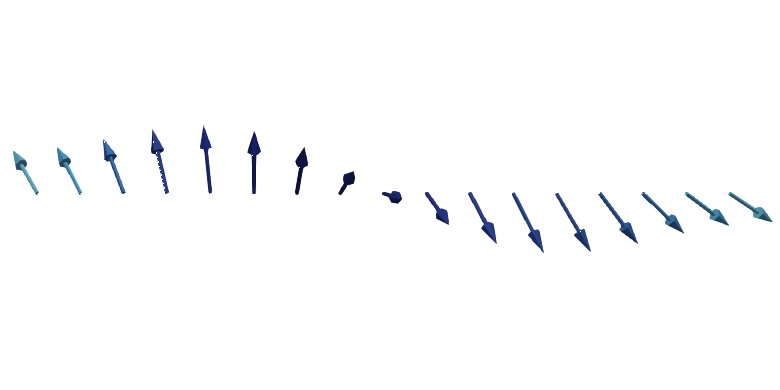}
		\caption{$t=0$}
	\end{subfigure}
	\begin{subfigure}[b]{0.21\textwidth}
		\centering
		\includegraphics[width=\textwidth]{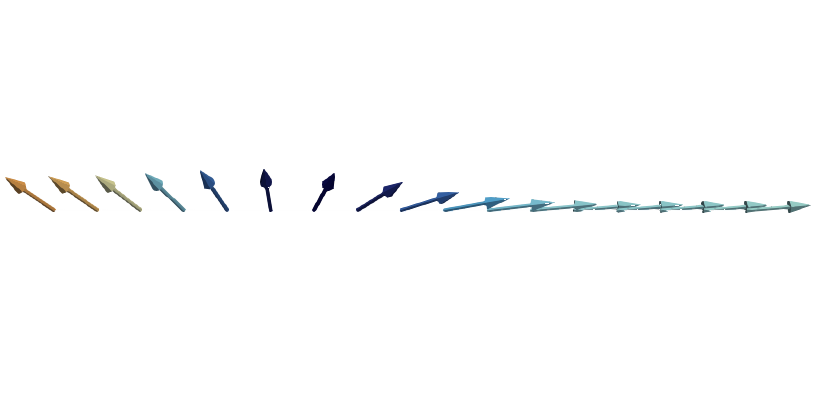}
		\caption{$t=0.02$}
	\end{subfigure}
	\begin{subfigure}[b]{0.21\textwidth}
		\centering
		\includegraphics[width=\textwidth]{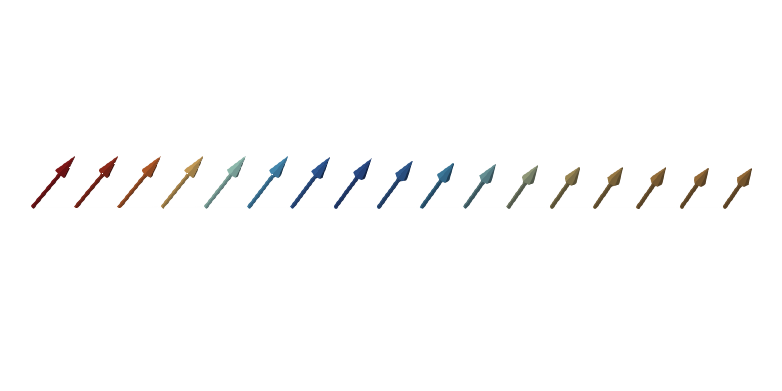}
		\caption{$t=0.1$}
	\end{subfigure}
	\begin{subfigure}[b]{0.21\textwidth}
		\centering
		\includegraphics[width=\textwidth]{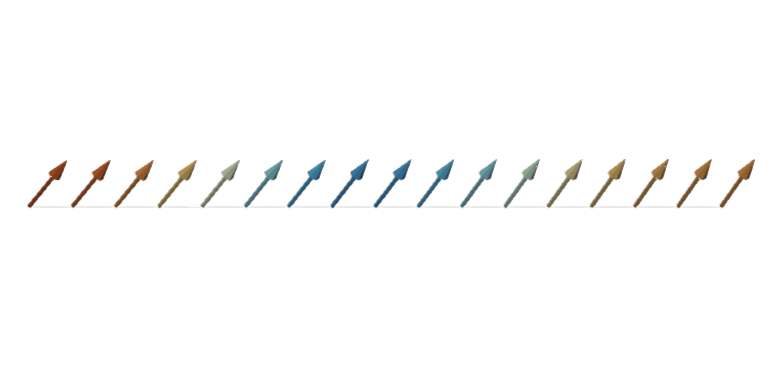}
		\caption{$t=0.2$}
	\end{subfigure}
	\begin{subfigure}[b]{0.08\textwidth}
		\centering
		\includegraphics[width=\textwidth]{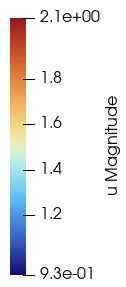}
	\end{subfigure}
	\caption{Snapshots of a sample path of the magnetisation $\bff{u}$ in simulation 1.}
	\label{fig:snapshots u 0}
\end{figure}

\begin{figure}[!htb]
\centering
\begin{subfigure}[b]{0.47\textwidth}
\centering
\includegraphics[width=\textwidth]{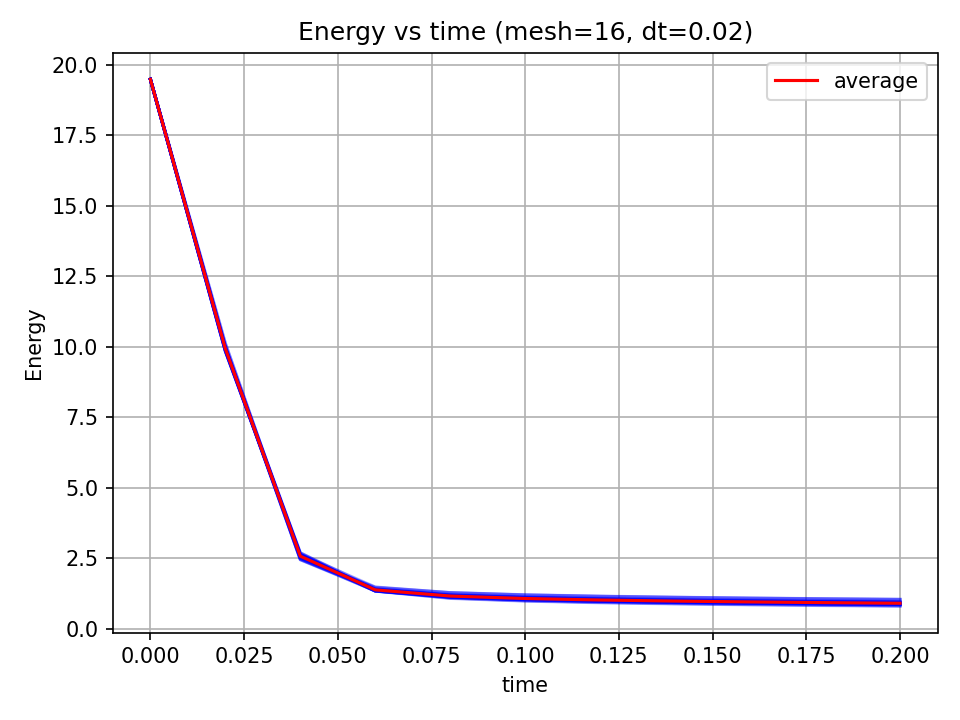}
\caption{Graph of energy vs time with $h=1/16$ and $k=1/50$ for 30 sample paths.}
\end{subfigure}
\hspace{1em}
\begin{subfigure}[b]{0.47\textwidth}
\centering
\includegraphics[width=\textwidth]{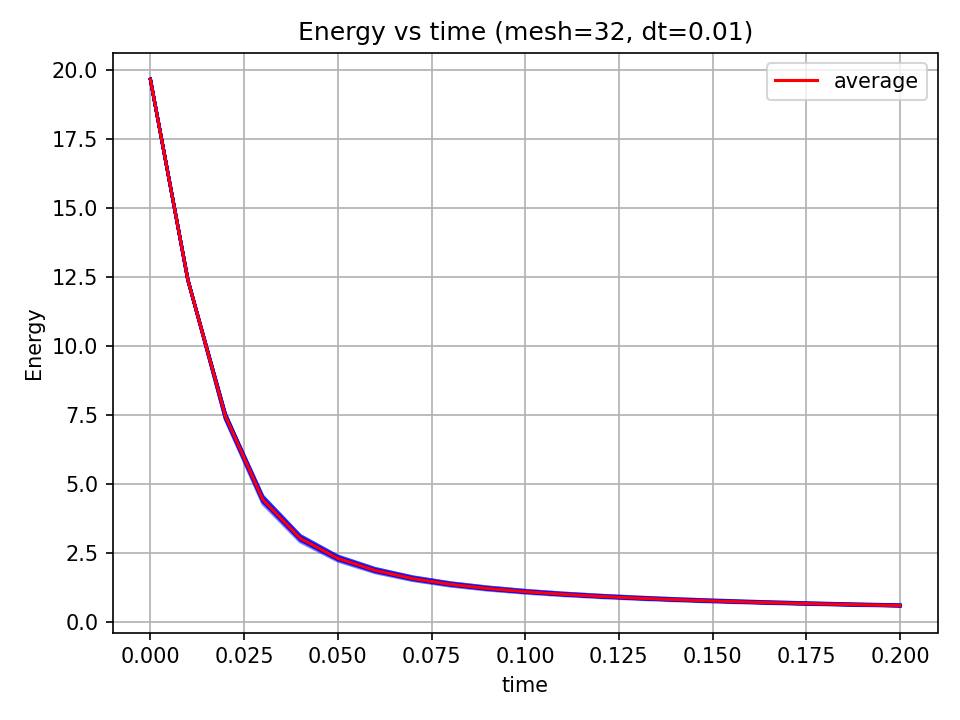}
\caption{Graph of energy vs time with $h=1/32$ and $k=1/100$ for 30 sample paths.}
\end{subfigure}
\caption{Energy evolution in simulation 1}
\label{fig:mass energy exp0}
\end{figure}

\begin{figure}[!htb]
	\begin{subfigure}[b]{0.45\textwidth}
		\centering
		\begin{tikzpicture}
			\begin{axis}[
				title=Plot of $\mathcal{E}_s^{\bff{u}}$ against $1/h$,
				height=1.1\textwidth,
				width=1\textwidth,
				xlabel= $1/h$,
				ylabel= $\mathcal{E}_s^{\bff{u}}$,
				xmode=log,
				ymode=log,
				legend pos=south west,
				legend cell align=left,
				]
				\addplot+[mark=*,red] coordinates {(4,1.12)(8,0.18)(16,0.051)};
			\addplot+[mark=*,blue] coordinates {(4,4.2)(8,1.04)(16,0.46)};
		\addplot+[dashed,no marks,blue,domain=8:16]{12.5/x};
		\addplot+[dashed,no marks,red,domain=8:16]{6.7/x^2};
		\legend{\scriptsize{$\mathcal{E}_0^{\bff{u}}(h)$}, \scriptsize{$\mathcal{E}_1^{\bff{u}}(h)$}, \scriptsize{order 1 line}, \scriptsize{order 2 line}}
	\end{axis}
\end{tikzpicture}
\caption{Spatial convergence order of $\bff{u}$.}
\label{fig:order u spatial 0}
\end{subfigure}
\hspace{1em}
\begin{subfigure}[b]{0.45\textwidth}
\centering
\begin{tikzpicture}
	\begin{axis}[
		title=Plot of $\mathcal{E}_s^{\bff{u}}$ against $1/k$,
		height=1.1\textwidth,
		width=1\textwidth,
		xlabel= $1/k$,
		ylabel= $\mathcal{E}_s^{\bff{u}}$,
		xmode=log,
		ymode=log,
		legend pos=south west,
		legend cell align=left,
		]
		\addplot+[mark=*,red] coordinates {(100,1.9)(200,1.6)(400,1.03)};
	\addplot+[mark=*,blue] coordinates {(100,5.5)(200,4.9)(400,3.27)};
\addplot+[dashed,no marks,blue,domain=180:400]{17/sqrt(x)};
\legend{\scriptsize{$\mathcal{E}_0^{\bff{u}}(k)$}, \scriptsize{$\mathcal{E}_1^{\bff{u}}(k)$}, \scriptsize{order 1/2 line}}
\end{axis}
\end{tikzpicture}
\caption{Temporal convergence order of $\bff{u}$.}
\label{fig:order u time 0}
\end{subfigure}
\caption{Convergence orders of magnetisation vector field $\bff{u}$ in simulation 1.}
\end{figure}

\begin{figure}[!htb]
\centering
\begin{subfigure}[b]{0.21\textwidth}
\centering
\includegraphics[width=\textwidth]{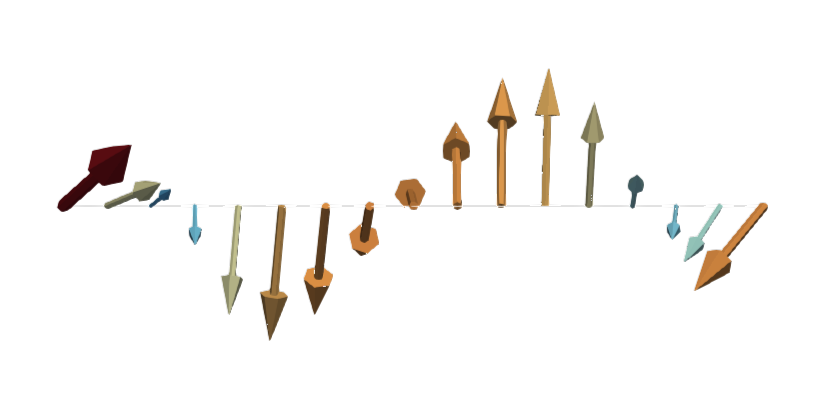}
\caption{$t=0$}
\end{subfigure}
\begin{subfigure}[b]{0.21\textwidth}
\centering
\includegraphics[width=\textwidth]{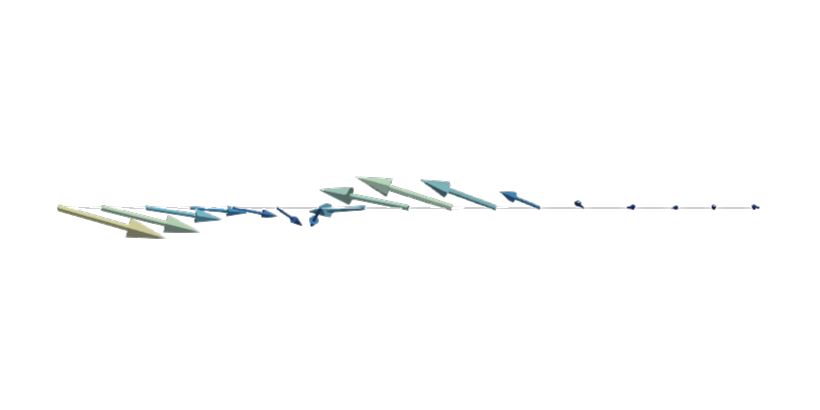}
\caption{$t=0.02$}
\end{subfigure}
\begin{subfigure}[b]{0.21\textwidth}
\centering
\includegraphics[width=\textwidth]{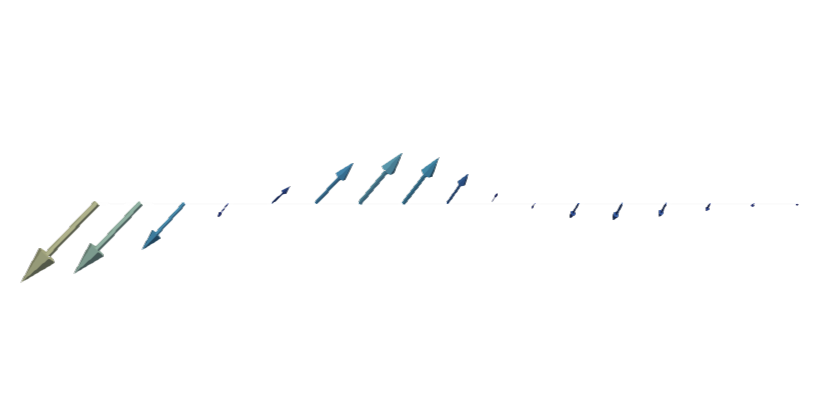}
\caption{$t=0.1$}
\end{subfigure}
\begin{subfigure}[b]{0.21\textwidth}
\centering
\includegraphics[width=\textwidth]{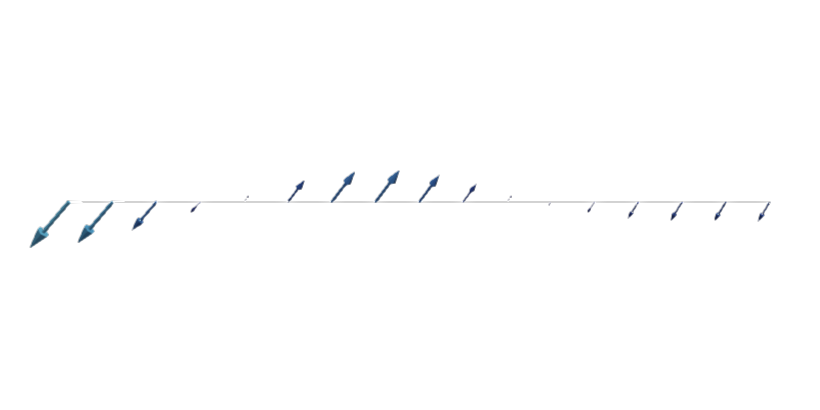}
\caption{$t=0.2$}
\end{subfigure}
\begin{subfigure}[b]{0.08\textwidth}
\centering
\includegraphics[width=\textwidth]{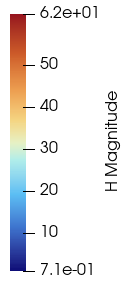}
\end{subfigure}
\caption{Snapshots of a sample path of the effective field $\bff{H}$ in simulation 1.}
\label{fig:snapshots H 0}
\end{figure}

\begin{figure}[!htb]
\begin{subfigure}[b]{0.45\textwidth}
\centering
\begin{tikzpicture}
\begin{axis}[
title=Plot of $\mathcal{E}_s^{\bff{H}}$ against $1/h$,
height=1.2\textwidth,
width=1\textwidth,
xlabel= $1/h$,
ylabel= $\mathcal{E}_s^{\bff{H}}$,
xmode=log,
ymode=log,
legend pos=south west,
legend cell align=left,
]
\addplot+[mark=*,red] coordinates {(4,23.3)(8,4.6)(16,1.3)};
\addplot+[mark=*,blue] coordinates {(4,244)(8,101)(16,49.5)}; 
\addplot+[dashed,no marks,blue,domain=7:16]{340/x};
\addplot+[dashed,no marks,red,domain=7:16]{150/x^2};
\legend{\scriptsize{$\mathcal{E}_0^{\bff{H}}(h)$}, \scriptsize{$\mathcal{E}_1^{\bff{H}}(h)$}, \scriptsize{order 1 line}, \scriptsize{order 2 line}}
\end{axis}
\end{tikzpicture}
\caption{Spatial convergence order of $\bff{H}$.}
\label{fig:order H spatial 0}
\end{subfigure}
\hspace{1em}
\begin{subfigure}[b]{0.45\textwidth}
\centering
\begin{tikzpicture}
\begin{axis}[
title=Plot of $\mathcal{E}_s^{\bff{H}}$ against $1/k$,
height=1.2\textwidth,
width=1\textwidth,
xlabel= $1/k$,
ylabel= $\mathcal{E}_s^{\bff{H}}$,
xmode=log,
ymode=log,
legend pos=south west,
legend cell align=left,
]
\addplot+[mark=*,red] coordinates {(100,28.3)(200,25)(400,18.7)};
\addplot+[mark=*,blue] coordinates {(100,228)(200,213)(400,179)};
\addplot+[dashed,no marks,blue,domain=200:400]{250/sqrt(x)};
\legend{\scriptsize{$\mathcal{E}_0^{\bff{H}}(k)$}, \scriptsize{$\mathcal{E}_1^{\bff{H}}(k)$}, \scriptsize{order 1/2 line}}
\end{axis}
\end{tikzpicture}
\caption{Temporal convergence order of $\bff{H}$.}
\label{fig:order H time 0}
\end{subfigure}
\caption{Convergence orders of effective field $\bff{H}$ in simulation 1.}
\end{figure}


\subsection{Simulation 2 (thin slab, small noise)}

Fix $\mathscr{D}=[0,1]^2$. In this simulation, we take the parameters to be $\lambda_1=0.2$, $\lambda_2=0.1$, $\gamma=5.0$, $\kappa=0.5$, $\mu=1.0$, $\beta_1=0.1$, and $\beta_2=0.05$. The current density is $\bff{\nu}=(1,0)^\top$. The initial data is specified to be
\[
\bff{u}_0(x)= \big(-y, x, 0 \big),
\]
and the vector field $\bff{g}$ is taken to be
\[
\bff{g}(x)= \big(0.8(1-x), 0.2, 0.5(1+x)\big).
\]
We solve the sLLBar equation with $T=0.2$.

Snapshots of a sample path of the magnetisation vector field $\bff{u}$ and the effective field $\bff{H}$ with mesh-size $h=1/16$ at selected times are shown in Figures~\ref{fig:snapshots u 1} and~\ref{fig:snapshots H 1}, respectively. The colour indicates the relative value of the magnitude. 

Figure~\ref{fig:mass energy exp1} shows the energy of the system over 30 independent sample paths for $h=1/16$, $k=0.01$, and for $h=1/32$, $k=0.005$. For relatively small noise intensity, the energy shows only minor pathwise fluctuations and, on average, decays over time.

Next, still with $T=0.2$, we fix a reference solution with $h=1/64$ and $k=1/800$. Figures~\ref{fig:order u spatial 1} and \ref{fig:order u time 1} display the plots of $\mathcal{E}_s^{\bff{u}}$ against $1/h$ and $1/k$, respectively. Similar plots for $\mathcal{E}_s^{\bff{H}}$ against $1/h$ and $1/k$ are shown in Figures~\ref{fig:order H spatial 1} and~\ref{fig:order H time 1}.

\begin{figure}[!htb]
	\centering
	\begin{subfigure}[b]{0.21\textwidth}
		\centering
		\includegraphics[width=\textwidth]{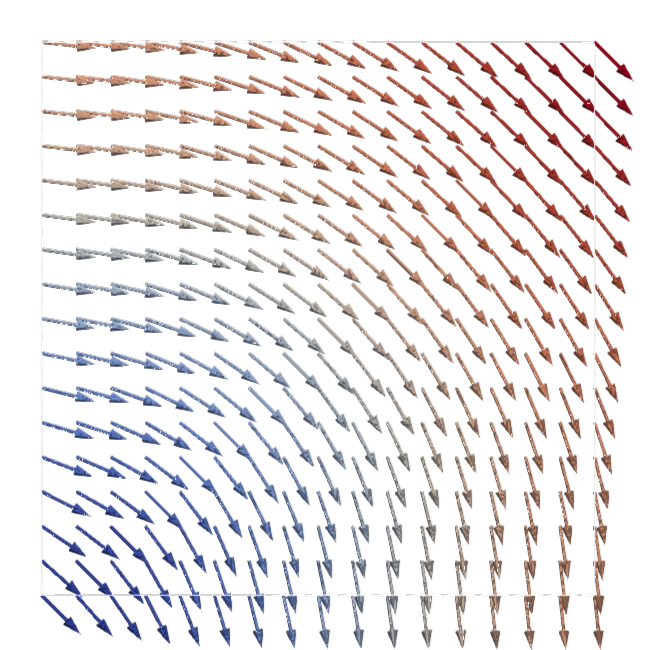}
		\caption{$t=0$}
	\end{subfigure}
	\begin{subfigure}[b]{0.21\textwidth}
		\centering
		\includegraphics[width=\textwidth]{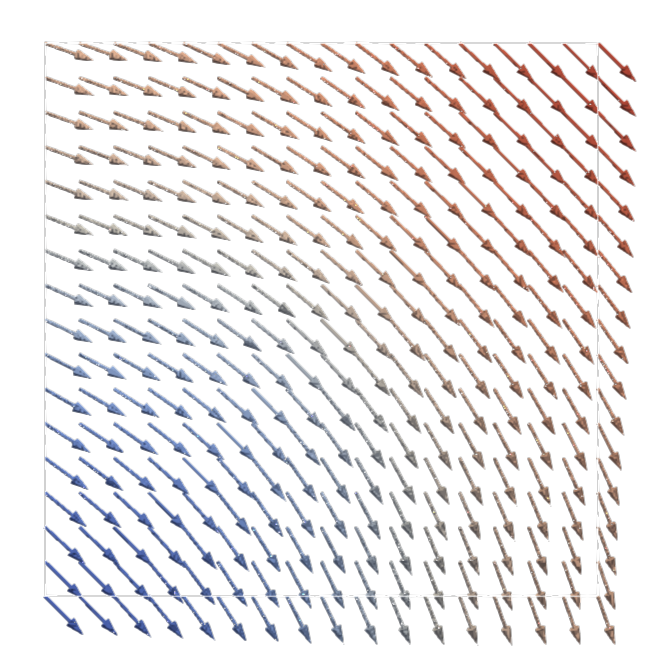}
		\caption{$t=0.02$}
	\end{subfigure}
	\begin{subfigure}[b]{0.21\textwidth}
		\centering
		\includegraphics[width=\textwidth]{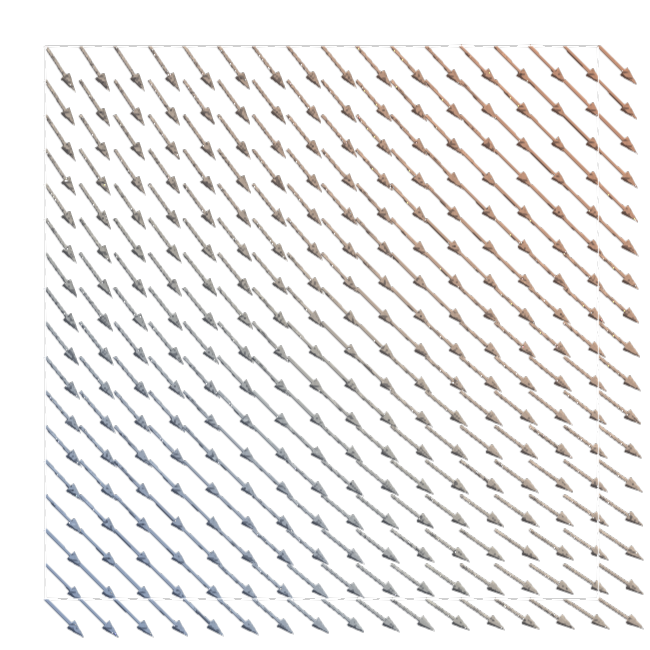}
		\caption{$t=0.1$}
	\end{subfigure}
	\begin{subfigure}[b]{0.21\textwidth}
		\centering
		\includegraphics[width=\textwidth]{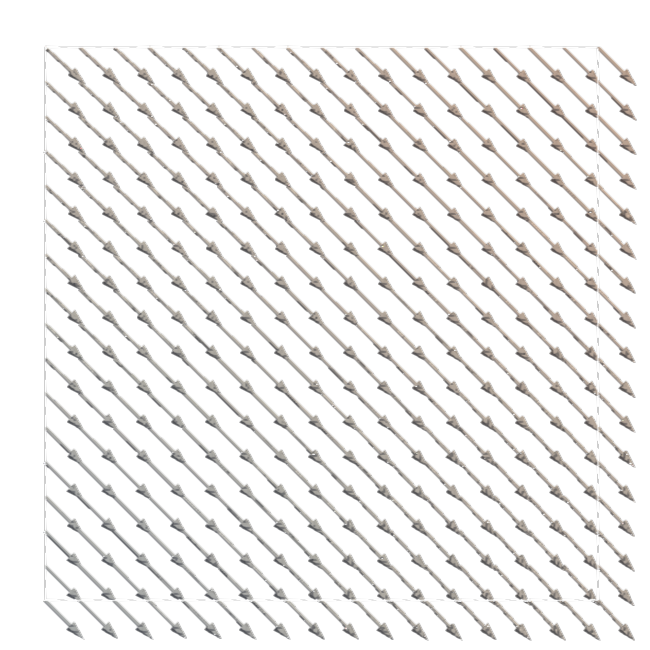}
		\caption{$t=0.2$}
	\end{subfigure}
	\begin{subfigure}[b]{0.08\textwidth}
		\centering
		\includegraphics[width=\textwidth]{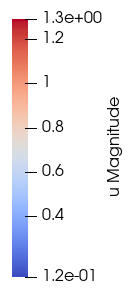}
	\end{subfigure}
	\caption{Snapshots of a sample path of the magnetisation $\bff{u}$ in simulation 2.}
	\label{fig:snapshots u 1}
\end{figure}

\begin{figure}[!htb]
\centering
\begin{subfigure}[b]{0.47\textwidth}
\centering
\includegraphics[width=\textwidth]{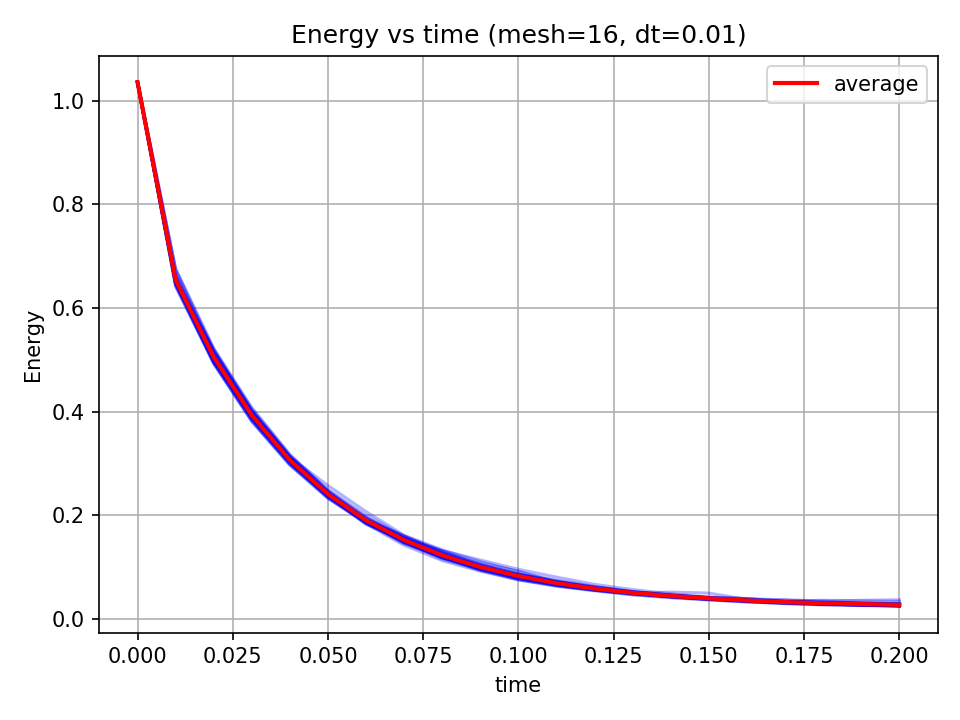}
\caption{Graph of energy vs time with $h=1/16$ and $k=1/100$ for 30 sample paths.}
\end{subfigure}
\hspace{1em}
\begin{subfigure}[b]{0.47\textwidth}
\centering
\includegraphics[width=\textwidth]{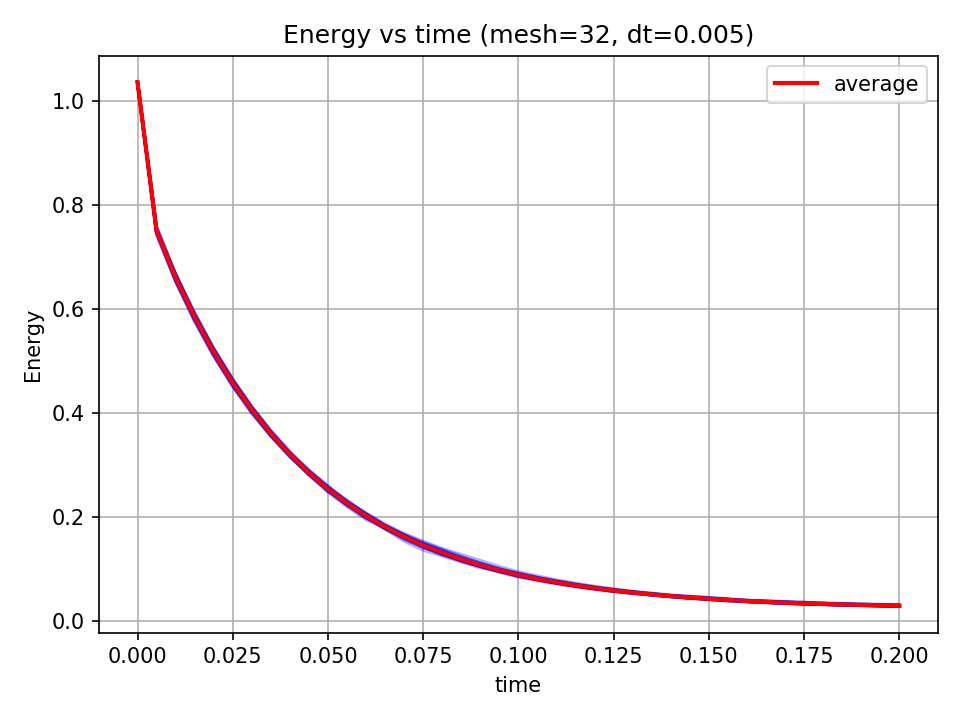}
\caption{Graph of energy vs time with $h=1/32$ and $k=1/200$ for 30 sample paths.}
\end{subfigure}
\caption{Energy evolution in simulation 2}
\label{fig:mass energy exp1}
\end{figure}

\begin{figure}[!htb]
	\begin{subfigure}[b]{0.45\textwidth}
		\centering
		\begin{tikzpicture}
			\begin{axis}[
				title=Plot of $\mathcal{E}_s^{\bff{u}}$ against $1/h$,
				height=1.1\textwidth,
				width=1\textwidth,
				xlabel= $1/h$,
				ylabel= $\mathcal{E}_s^{\bff{u}}$,
				xmode=log,
				ymode=log,
				legend pos=south west,
				legend cell align=left,
				]
				\addplot+[mark=*,red] coordinates {(4,0.014)(8,0.0036)(16,0.00088)}; 
			\addplot+[mark=*,blue] coordinates {(4,0.035)(8,0.014)(16,0.0064)}; 
		\addplot+[dashed,no marks,blue,domain=7:16]{0.2/x};
		\addplot+[dashed,no marks,red,domain=7:16]{0.37/x^2};
		\legend{\scriptsize{$\mathcal{E}_0^{\bff{u}}(h)$}, \scriptsize{$\mathcal{E}_1^{\bff{u}}(h)$}, \scriptsize{order 1 line}, \scriptsize{order 2 line}}
	\end{axis}
\end{tikzpicture}
\caption{Spatial convergence order of $\bff{u}$.}
\label{fig:order u spatial 1}
\end{subfigure}
\hspace{1em}
\begin{subfigure}[b]{0.45\textwidth}
\centering
\begin{tikzpicture}
	\begin{axis}[
		title=Plot of $\mathcal{E}_s^{\bff{u}}$ against $1/k$,
		height=1.1\textwidth,
		width=1\textwidth,
		xlabel= $1/k$,
		ylabel= $\mathcal{E}_s^{\bff{u}}$,
		xmode=log,
		ymode=log,
		legend pos=south west,
		legend cell align=left,
		]
		\addplot+[mark=*,red] coordinates {(25,0.076)(50,0.051)(100,0.031)}; 
	\addplot+[mark=*,blue] coordinates {(25,0.1)(50,0.077)(100,0.050)}; 
\addplot+[dashed,no marks,blue,domain=40:100]{0.42/sqrt(x)};
\legend{\scriptsize{$\mathcal{E}_0^{\bff{u}}(k)$}, \scriptsize{$\mathcal{E}_1^{\bff{u}}(k)$}, \scriptsize{order 1/2 line}}
\end{axis}
\end{tikzpicture}
\caption{Temporal convergence order of $\bff{u}$.}
\label{fig:order u time 1}
\end{subfigure}
\caption{Convergence orders of magnetisation vector field $\bff{u}$ in simulation 2.}
\end{figure}

\begin{figure}[!htb]
\centering
\begin{subfigure}[b]{0.21\textwidth}
\centering
\includegraphics[width=\textwidth]{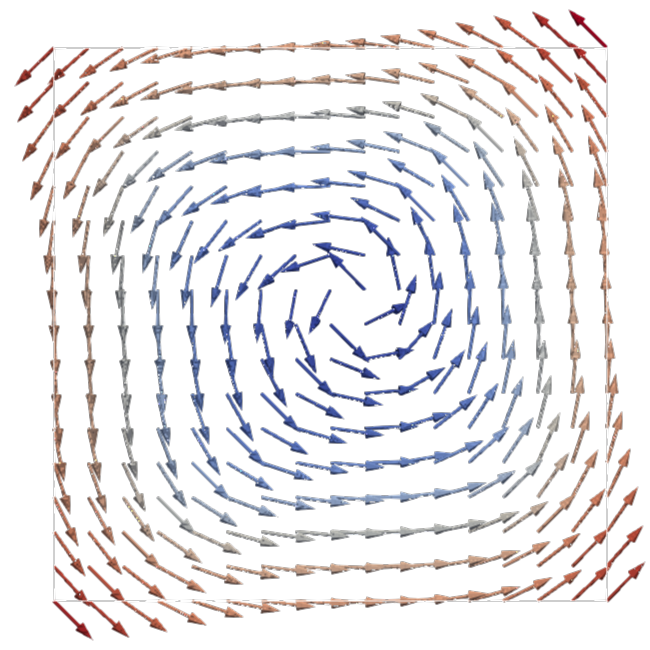}
\caption{$t=0$}
\end{subfigure}
\begin{subfigure}[b]{0.21\textwidth}
\centering
\includegraphics[width=\textwidth]{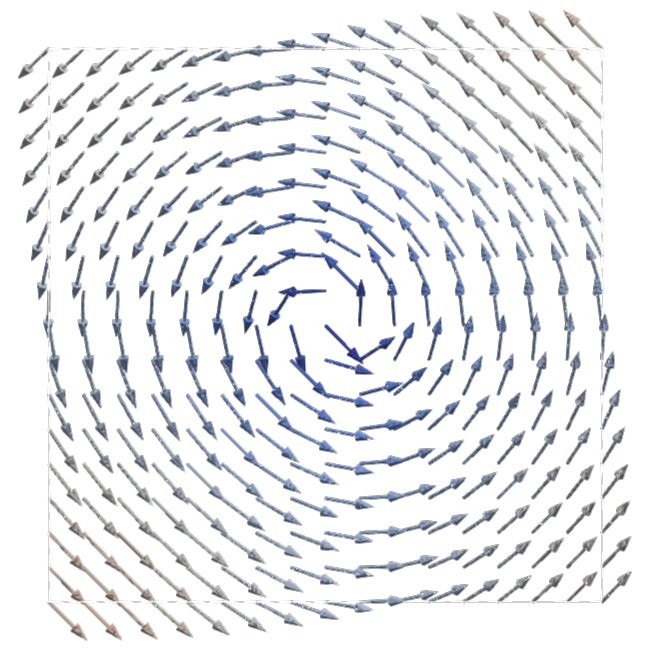}
\caption{$t=0.02$}
\end{subfigure}
\begin{subfigure}[b]{0.21\textwidth}
\centering
\includegraphics[width=\textwidth]{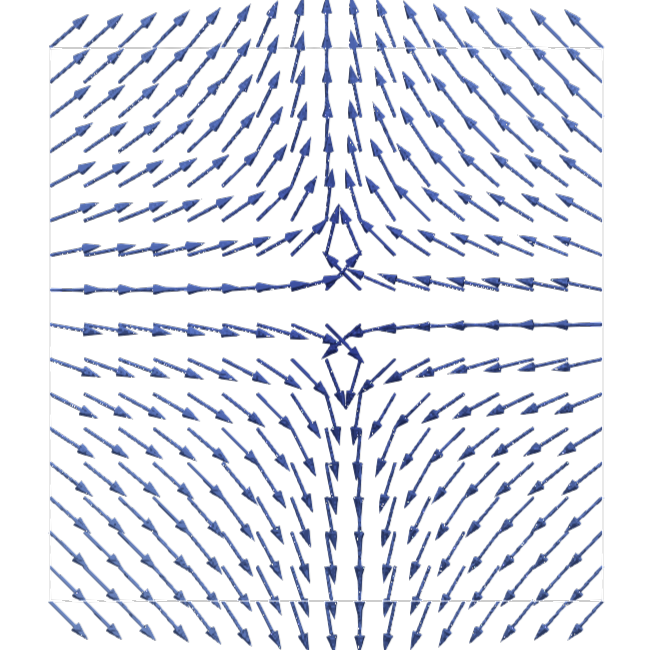}
\caption{$t=0.1$}
\end{subfigure}
\begin{subfigure}[b]{0.21\textwidth}
\centering
\includegraphics[width=\textwidth]{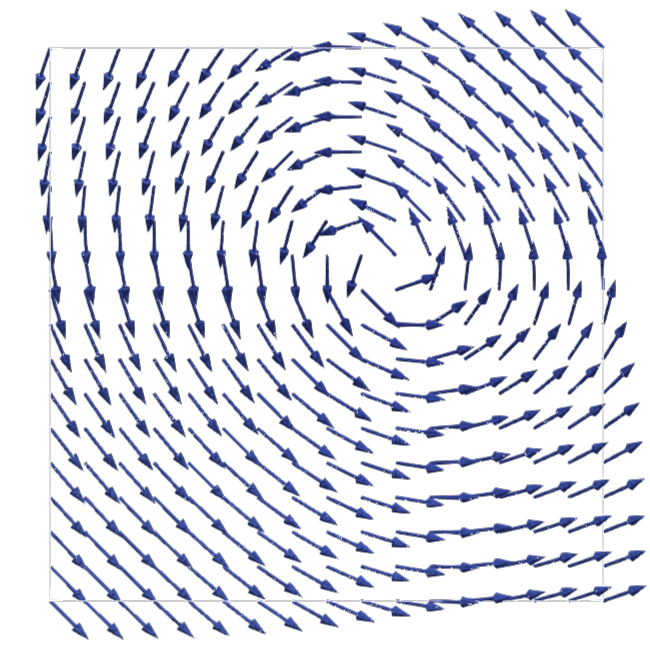}
\caption{$t=0.2$}
\end{subfigure}
\begin{subfigure}[b]{0.08\textwidth}
\centering
\includegraphics[width=\textwidth]{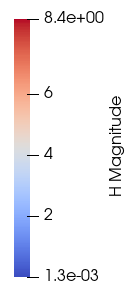}
\end{subfigure}
\caption{Snapshots of a sample path of the effective field $\bff{H}$ in simulation 2.}
\label{fig:snapshots H 1}
\end{figure}

\begin{figure}[!htb]
\begin{subfigure}[b]{0.45\textwidth}
\centering
\begin{tikzpicture}
\begin{axis}[
title=Plot of $\mathcal{E}_s^{\bff{H}}$ against $1/h$,
height=1.2\textwidth,
width=1\textwidth,
xlabel= $1/h$,
ylabel= $\mathcal{E}_s^{\bff{H}}$,
xmode=log,
ymode=log,
legend pos=south west,
legend cell align=left,
]
\addplot+[mark=*,red] coordinates {(4,0.076)(8,0.02)(16,0.0048)}; 
\addplot+[mark=*,blue] coordinates {(4,0.33)(8,0.14)(16,0.063)}; 
\addplot+[dashed,no marks,blue,domain=7:16]{1.9/x};
\addplot+[dashed,no marks,red,domain=7:16]{2/x^2};
\legend{\scriptsize{$\mathcal{E}_0^{\bff{H}}(h)$}, \scriptsize{$\mathcal{E}_1^{\bff{H}}(h)$}, \scriptsize{order 1 line}, \scriptsize{order 2 line}}
\end{axis}
\end{tikzpicture}
\caption{Spatial convergence order of $\bff{H}$.}
\label{fig:order H spatial 1}
\end{subfigure}
\hspace{1em}
\begin{subfigure}[b]{0.45\textwidth}
\centering
\begin{tikzpicture}
\begin{axis}[
title=Plot of $\mathcal{E}_s^{\bff{H}}$ against $1/k$,
height=1.2\textwidth,
width=1\textwidth,
xlabel= $1/k$,
ylabel= $\mathcal{E}_s^{\bff{H}}$,
xmode=log,
ymode=log,
legend pos=south west,
legend cell align=left,
]
\addplot+[mark=*,red] coordinates {(25,0.32)(50,0.24)(100,0.15)}; 
\addplot+[mark=*,blue] coordinates {(25,1.00)(50,0.74)(100,0.45)}; 
\addplot+[dashed,no marks,blue,domain=40:100]{1.2/sqrt(x)};
\legend{\scriptsize{$\mathcal{E}_0^{\bff{H}}(k)$}, \scriptsize{$\mathcal{E}_1^{\bff{H}}(k)$}, \scriptsize{order 1/2 line}}
\end{axis}
\end{tikzpicture}
\caption{Temporal convergence order of $\bff{H}$.}
\label{fig:order H time 1}
\end{subfigure}
\caption{Convergence orders of effective field $\bff{H}$ in simulation 2.}
\end{figure}


\subsection{Simulation 3 (thin slab, moderate noise)}

Set $\mathscr{D}=[0,1]^2$.
In this simulation, we take the parameters to be $\lambda_1=0.5$, $\lambda_2=0.05$, $\gamma=8.0$, $\kappa=0.25$, $\mu=1.0$, $\beta_1=0.2$, and $\beta_2=0.1$. The current density is $\bff{\nu}=(2,0)^\top$. The initial data is specified to be
\[
\bff{u}_0(x)= \big(\sin(2\pi y), \sin(2\pi x), 0 \big),
\]
and the vector field $\bff{g}$ is taken to be
\[
\bff{g}(x)= \big(5(1+x), 10(1+y), 2\cos(2\pi x) \big).
\]

Snapshots of a sample path of the magnetisation vector field $\bff{u}$ and the effective field $\bff{H}$ with mesh-size $h=1/16$ at selected times are shown in Figures~\ref{fig:snapshots u 2} and~\ref{fig:snapshots H 2}, respectively. The colour indicates the relative value of the magnitude.

Figure~\ref{fig:mass energy exp2} shows the energy over 30 independent sample paths for $h=1/16$, $k=1/50$, and for $h=1/32$, $k=1/100$.
Since the noise intensity is moderately large, the energy trajectory exhibits more pronounced pathwise variability around the mean profile.

Finally, we set $T=0.05$ and fix a reference solution with $h=1/64$ and $k=1/800$. Figures~\ref{fig:order u spatial 2} and \ref{fig:order u time 2} display the plots of $\mathcal{E}_s^{\bff{u}}$ against $1/h$ and $1/k$, respectively. Similar plots for $\mathcal{E}_s^{\bff{H}}$ against $1/h$ and $1/k$ are shown in Figures~\ref{fig:order H spatial 2} and~\ref{fig:order H time 2}.

\begin{figure}[!htb]
\centering
\begin{subfigure}[b]{0.21\textwidth}
\centering
\includegraphics[width=\textwidth]{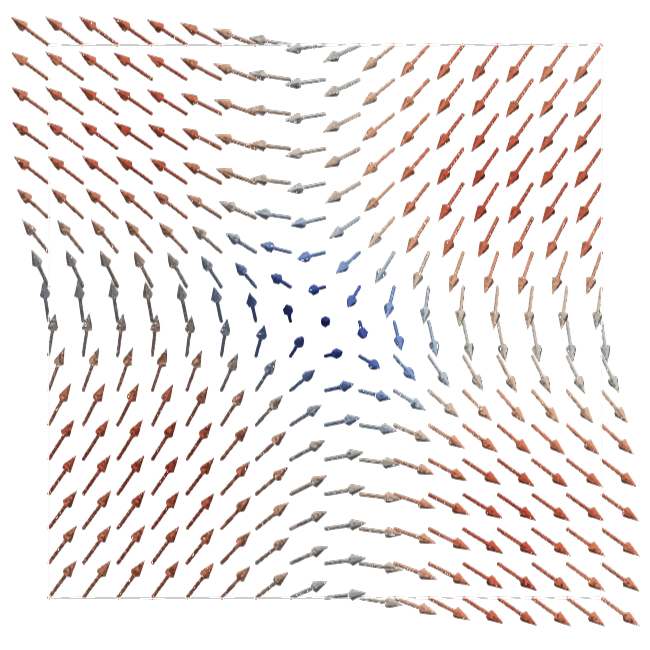}
\caption{$t=0$}
\end{subfigure}
\begin{subfigure}[b]{0.21\textwidth}
\centering
\includegraphics[width=\textwidth]{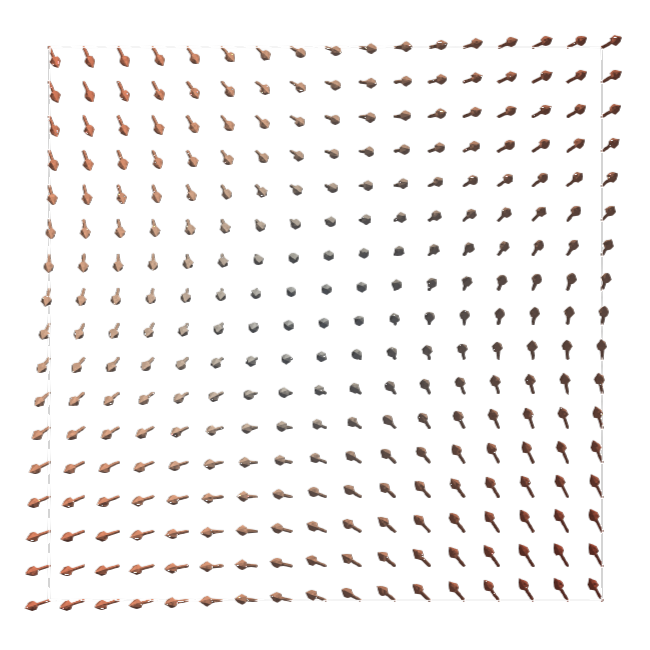}
\caption{$t=0.05$}
\end{subfigure}
\begin{subfigure}[b]{0.21\textwidth}
\centering
\includegraphics[width=\textwidth]{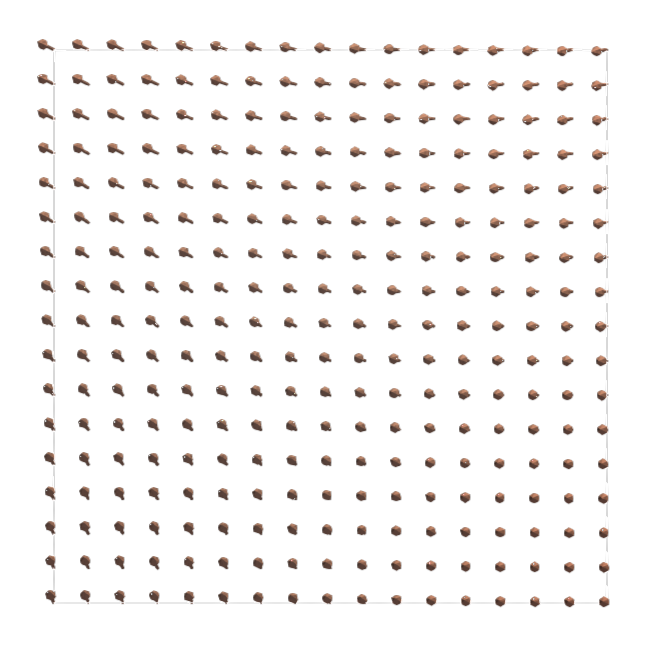}
\caption{$t=0.1$}
\end{subfigure}
\begin{subfigure}[b]{0.21\textwidth}
\centering
\includegraphics[width=\textwidth]{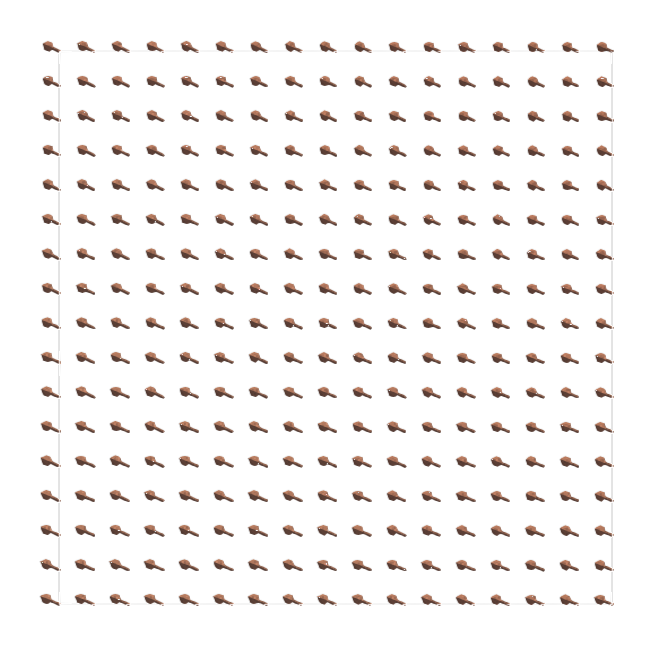}
\caption{$t=0.2$}
\end{subfigure}
\begin{subfigure}[b]{0.08\textwidth}
\centering
\includegraphics[width=\textwidth]{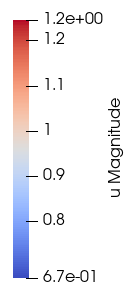}
\end{subfigure}
\caption{Snapshots of a sample path of the magnetisation $\bff{u}$ in simulation 3.}
\label{fig:snapshots u 2}
\end{figure}

\begin{figure}[!htb]
\centering
\begin{subfigure}[b]{0.47\textwidth}
\centering
\includegraphics[width=\textwidth]{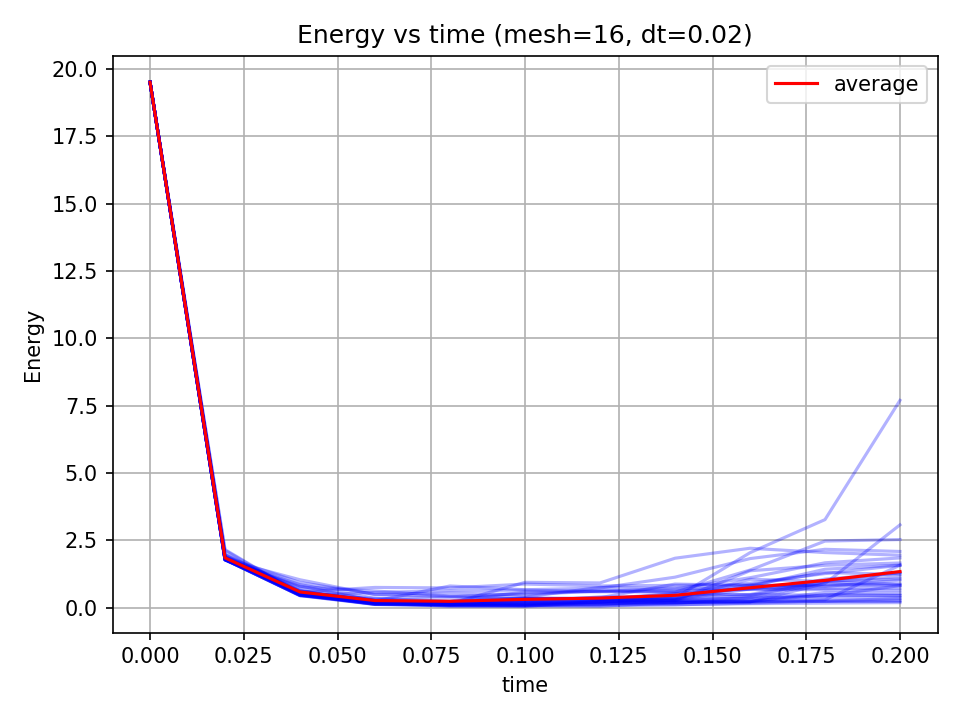}
\caption{Graph of energy vs time with $h=1/16$ and $k=1/50$ for 30 sample paths.}
\end{subfigure}
\hspace{1em}
\begin{subfigure}[b]{0.47\textwidth}
\centering
\includegraphics[width=\textwidth]{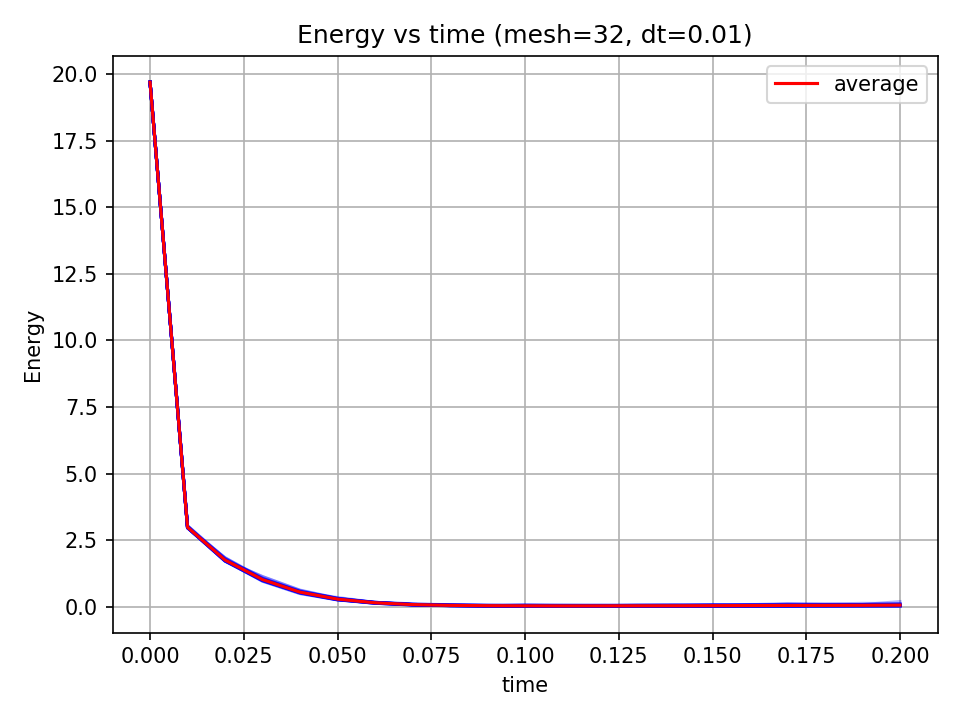}
\caption{Graph of energy vs time with $h=1/32$ and $k=1/100$ for 30 sample paths.}
\end{subfigure}
\caption{Energy evolution in simulation 3}
\label{fig:mass energy exp2}
\end{figure}

\begin{figure}[!htb]
\begin{subfigure}[b]{0.45\textwidth}
\centering
\begin{tikzpicture}
\begin{axis}[
title=Plot of $\mathcal{E}_s^{\bff{u}}$ against $1/h$,
height=1.1\textwidth,
width=1\textwidth,
xlabel= $1/h$,
ylabel= $\mathcal{E}_s^{\bff{u}}$,
xmode=log,
ymode=log,
legend pos=south west,
legend cell align=left,
]
\addplot+[mark=*,red] coordinates {(4,0.12)(8,0.0046)(16,0.0012)}; 
\addplot+[mark=*,blue] coordinates {(4,0.17)(8,0.058)(16,0.026)}; 
\addplot+[dashed,no marks,blue,domain=7:16]{0.8/x};
\addplot+[dashed,no marks,red,domain=7:16]{0.6/x^2};
\legend{\scriptsize{$\mathcal{E}_0^{\bff{u}}(h)$}, \scriptsize{$\mathcal{E}_1^{\bff{u}}(h)$}, \scriptsize{order 1 line}, \scriptsize{order 2 line}}
\end{axis}
\end{tikzpicture}
\caption{Spatial convergence order of $\bff{u}$.}
\label{fig:order u spatial 2}
\end{subfigure}
\hspace{1em}
\begin{subfigure}[b]{0.45\textwidth}
\centering
\begin{tikzpicture}
\begin{axis}[
title=Plot of $\mathcal{E}_s^{\bff{u}}$ against $1/k$,
height=1.1\textwidth,
width=1\textwidth,
xlabel= $1/k$,
ylabel= $\mathcal{E}_s^{\bff{u}}$,
xmode=log,
ymode=log,
legend pos=south west,
legend cell align=left,
]
\addplot+[mark=*,red] coordinates {(25,2.4)(50,1.5)(100,0.7)}; 
\addplot+[mark=*,blue] coordinates {(25,2.6)(50,2.1)(100,1.7)}; 
\addplot+[dashed,no marks,blue,domain=40:100]{13/sqrt(x)};
\legend{\scriptsize{$\mathcal{E}_0^{\bff{u}}(k)$}, \scriptsize{$\mathcal{E}_1^{\bff{u}}(k)$}, \scriptsize{order 1/2 line}}
\end{axis}
\end{tikzpicture}
\caption{Temporal convergence order of $\bff{u}$.}
\label{fig:order u time 2}
\end{subfigure}
\caption{Convergence orders of magnetisation vector field $\bff{u}$ in simulation 3.}
\end{figure}

\begin{figure}[!htb]
\centering
\begin{subfigure}[b]{0.21\textwidth}
\centering
\includegraphics[width=\textwidth]{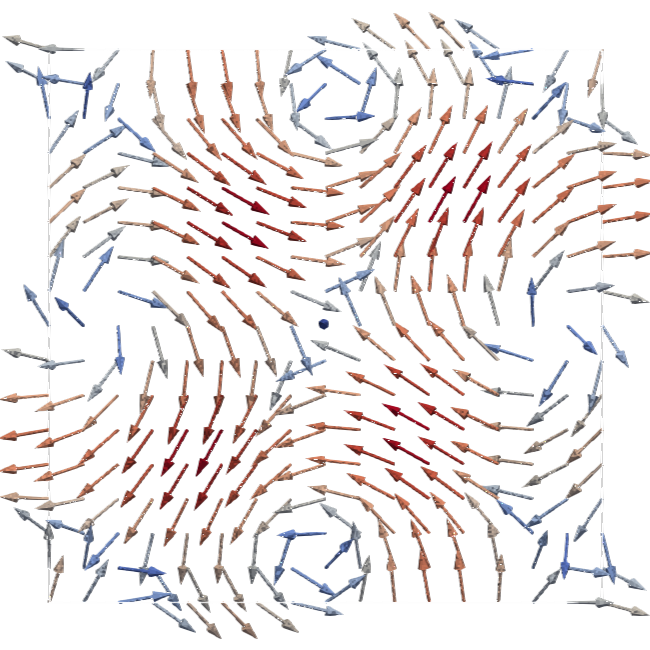}
\caption{$t=0$}
\end{subfigure}
\begin{subfigure}[b]{0.21\textwidth}
\centering
\includegraphics[width=\textwidth]{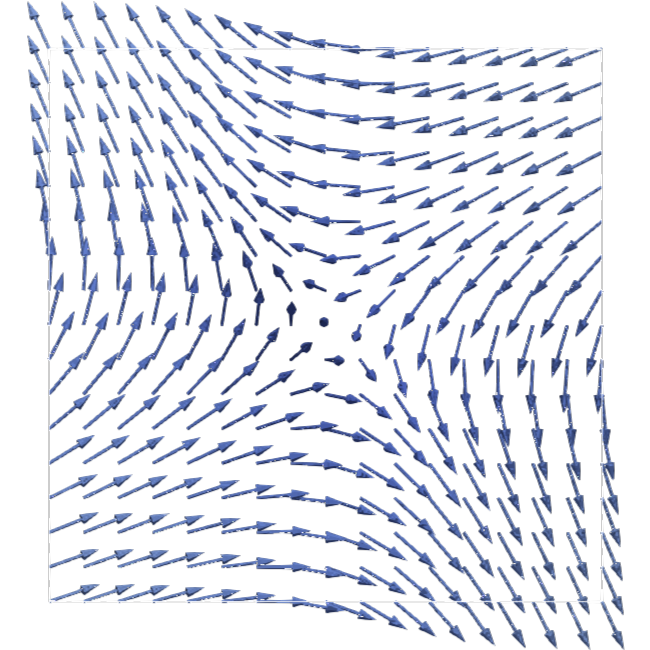}
\caption{$t=0.05$}
\end{subfigure}
\begin{subfigure}[b]{0.21\textwidth}
\centering
\includegraphics[width=\textwidth]{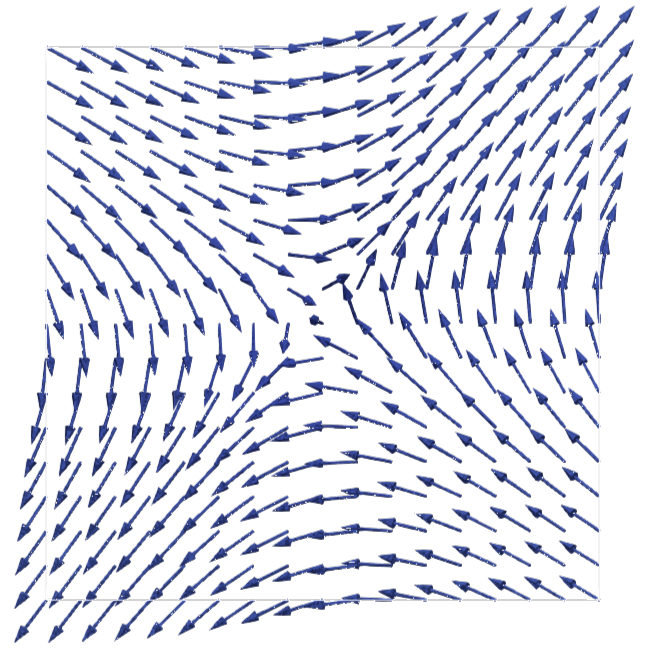}
\caption{$t=0.1$}
\end{subfigure}
\begin{subfigure}[b]{0.21\textwidth}
\centering
\includegraphics[width=\textwidth]{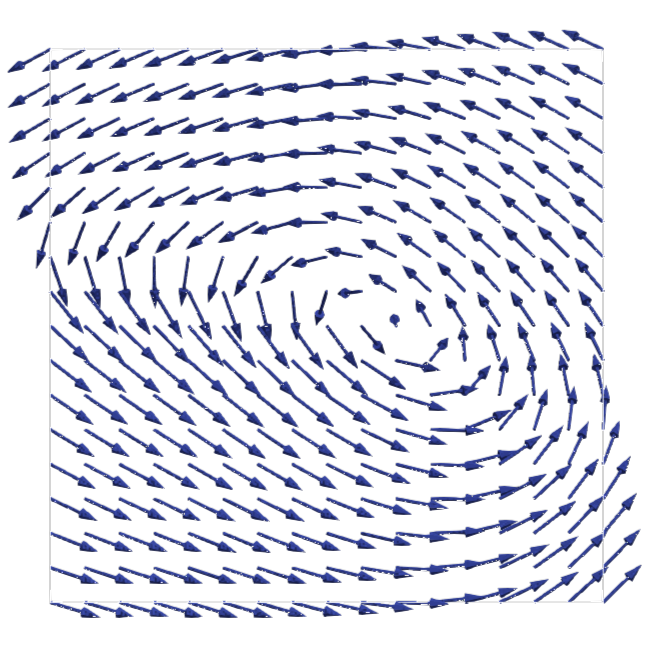}
\caption{$t=0.2$}
\end{subfigure}
\begin{subfigure}[b]{0.08\textwidth}
\centering
\includegraphics[width=\textwidth]{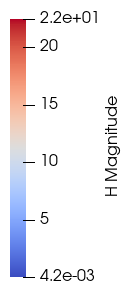}
\end{subfigure}
\caption{Snapshots of a sample path of the effective field $\bff{H}$ in simulation 3.}
\label{fig:snapshots H 2}
\end{figure}

\begin{figure}[!htb]
\begin{subfigure}[b]{0.45\textwidth}
\centering
\begin{tikzpicture}
\begin{axis}[
title=Plot of $\mathcal{E}_s^{\bff{H}}$ against $1/h$,
height=1.1\textwidth,
width=1\textwidth,
xlabel= $1/h$,
ylabel= $\mathcal{E}_s^{\bff{H}}$,
xmode=log,
ymode=log,
legend pos=south west,
legend cell align=left,
]
\addplot+[mark=*,red] coordinates {(4,0.51)(8,0.17)(16,0.044)}; 
\addplot+[mark=*,blue] coordinates {(4,1.63)(8,0.57)(16,0.16)}; 
\addplot+[dashed,no marks,blue,domain=7:16]{7/x};
\addplot+[dashed,no marks,red,domain=7:16]{5.5/x^2};
\legend{\scriptsize{$\mathcal{E}_0^{\bff{H}}(h)$}, \scriptsize{$\mathcal{E}_1^{\bff{H}}(h)$}, \scriptsize{order 1 line}, \scriptsize{order 2 line}}
\end{axis}
\end{tikzpicture}
\caption{Spatial convergence order of $\bff{H}$.}
\label{fig:order H spatial 2}
\end{subfigure}
\hspace{1em}
\begin{subfigure}[b]{0.45\textwidth}
\centering
\begin{tikzpicture}
\begin{axis}[
title=Plot of $\mathcal{E}_s^{\bff{H}}$ against $1/k$,
height=1.1\textwidth,
width=1\textwidth,
xlabel= $1/k$,
ylabel= $\mathcal{E}_s^{\bff{H}}$,
xmode=log,
ymode=log,
legend pos=south west,
legend cell align=left,
]
\addplot+[mark=*,red] coordinates {(25,2.7)(50,1.81)(100,0.7)}; 
\addplot+[mark=*,blue] coordinates {(25,3.5)(50,2.4)(100,1.78)}; 
\addplot+[dashed,no marks,blue,domain=40:100]{14.5/sqrt(x)};
\legend{\scriptsize{$\mathcal{E}_0^{\bff{H}}(k)$}, \scriptsize{$\mathcal{E}_1^{\bff{H}}(k)$}, \scriptsize{order 1/2 line}}
\end{axis}
\end{tikzpicture}
\caption{Temporal convergence order of $\bff{H}$.}
\label{fig:order H time 2}
\end{subfigure}
\caption{Convergence orders of effective field $\bff{H}$ in simulation 3.}
\end{figure}

\section*{Acknowledgements}
The authors acknowledge financial support through the Australian Research Council's Discovery Projects funding scheme (projects DP220101811 and DP240100781).
Agus L. Soenjaya is supported by the Australian Government Research Training Program (RTP) Scholarship awarded at the University of New South Wales, Sydney.

The authors would like to thank the anonymous reviewers for their valuable comments and suggestions which improve the quality of this paper.


\end{document}